\documentclass[12pt,reqno]{amsart}
\usepackage{amsmath,amssymb,amsfonts,amsthm}
\usepackage{mathrsfs}
\usepackage[all]{xy}
\usepackage{stmaryrd}
\usepackage{mathtools}
\usepackage{indentfirst}
\usepackage{verbatim}
\usepackage{hyperref}
\usepackage{color}
\usepackage{soul}
\usepackage[usenames,dvipsnames]{xcolor}
\usepackage{graphicx}
\graphicspath{ {images/} }
\usepackage[normalem]{ulem}
\usepackage{tikz-cd}
\usepackage{nicefrac}
\usepackage{xfrac}
\usepackage[shortlabels]{enumitem}
\DeclareMathOperator{\ve}{\varepsilon}
\newtheorem*{claim}{Claim}
\newtheorem*{theorem*}{Theorem}
\newtheorem{theorem}{Theorem}[section]
\newtheorem{proposition}[theorem]{Proposition}
\newtheorem{lemma}[theorem]{Lemma}
\newtheorem{propdef}[theorem]{Proposition-Definition}

\newtheorem{corollary}[theorem]{Corollary}
\theoremstyle{definition}
\newtheorem{notation}[theorem]{Notation}
\newtheorem{definition}[theorem]{Definition}
\newtheorem{remark}[theorem]{Remark}
\newtheorem{example}[theorem]{Example}
\newtheorem{question}[theorem]{Question}

\newcommand{\mc}{\mathcal}

\newcommand{\mb}{\mathbb}

\newcommand{\on}{\operatorname}
\newenvironment{claimproof}[1]{\par\noindent\emph{Proof of Claim.}\space#1}{}
\usepackage[margin=0.75in]{geometry}

\newcommand*{\sheafhom}{\mathscr{H}\kern -2.5pt om}
\newcommand*{\sheafext}{\mathscr{E}\kern -1.5pt xt}

\newcommand\cE{{\mathcal{E}}}
\newcommand\cF{{\mathcal{F}}}
\newcommand\cG{{\mathcal{G}}}
\newcommand\cH{{\mathcal{H}}}
\newcommand\cI{{\mathcal{I}}}

\newcommand\cM{{\mathcal{M}}}
\newcommand\cN{{\mathcal{N}}}
\newcommand\cO{{\mathcal{O}}}

\newcommand\cT{{\mathcal{T}}}

\newcommand\bA{{\mathbb A}}

\newcommand\bC{{\mathbb C}}

\newcommand\bK{{\mathbb K}}

\newcommand\bN{{\mathbb N}}

\newcommand\bQ{{\mathbb Q}}
\newcommand\bR{{\mathbb R}}

\newcommand\bZ{{\mathbb Z}}

\newcommand{\vp}{\varphi}

\newcommand{\diff}{\textnormal{d}}

\title{On toric and toroidal foliations}
\author{Chih-Wei Chang, Yen-An Chen}
\address{Department of Mathematics, National Taiwan University, Taipei, 106, Taiwan}
\email{cwchang0219@ntu.edu.tw}
\address{School of Mathematics, Korea Institute for Advanced Study, 85 Hoegi-ro, Dongdaemun-gu, Seoul 02455, Republic of Korea}
\email{yachen@kias.re.kr}

\begin{document}

\maketitle

\begin{abstract}
    In this paper, we provide toric descriptions for various foliation singularities on toric varieties, especially for non-dicritical singularities and F-dlt singularities. 
We then show that the toric foliated minimal model program works by demonstrating that non-dicritical singularities and F-dlt singularities are preserved.
\end{abstract}

\tableofcontents

\section{Introduction}\label{sec:intro}
In recent years, there have been numerous advancements in the field of birational geometry of foliations. 
Notably, it has been proven that the minimal model program works for foliations of any rank on a normal variety of dimension at most three (for example, see \cite{mendes2000kodaira, brunella2015birational, mcquillan2005semi, mcquillan2008canonical, spicer2020higher, CS, cascini2025mmp, spicer2022local}), as well as for algebraically integrable foliations (\cite{ambro2021positivity, cascini2023mmp, CHLX, LMX, M7}).

It is natural to ask for the applicability of the foliated minimal model program (FMMP) to toric foliations. 
As $\bQ$-factorial projective toric varieties are Mori dream spaces, the minimal model program works for any Weil divisor $D$ (see \cite{hu2000mori}), and any singularities involving only discrepancies, such as canonical singularities, are preserved under the FMMP. 
Therefore, the main goal for the FMMP for toric foliations is to show that the non-dicritical singularities (see Definition~\ref{nddef}) are preserved under the FMMP. 
In \cite{spicer2020higher}, C. Spicer showed that the FMMP works for toric foliations of corank one with only canonical and non-dicritical singularities. 
 
In this paper, we provide a comprehensive affirmative answer. We first characterize non-dicritical toric foliations in terms of combinatorial data. If the toric variety $X_\Sigma$ is defined by a fan $\Sigma$ in $N\otimes\mb R$ where $N\simeq\mb Z^n$ is a lattice, a toric foliation on $X_\Sigma$ corresponds to a complex vector subspace $W\subseteq N\otimes\mb C$ and is denoted by $\mc F_W$ (\cite{wang2023toric} and Proposition~\ref{1-1}). 

\begin{theorem}[{cf. Theorem~\ref{ND_toroidal}}]
    Let $\cF=\cF_W$ be a toric foliation on a toric variety $X_\Sigma$ of a fan $\Sigma$ in $N\otimes\bR$ where $W\subseteq N\otimes\bC$ is a complex vector subspace. 
    Then $\cF_W$ is non-dicritical if and only if $(\Sigma,W)$ satisfies the condition $(\dagger)$ (see Definition~\ref{defn_dagger}). 
\end{theorem}

Here we introduce a version of non-dicritical singularities for foliations of any rank, which generalizes \cite[Definition 2.10]{CS} and \cite[paragraph before Lemma 2.6]{cascini2025mmp} to any dimension and any rank (see Definition~\ref{nddef}).
It is worth noting that there is another version of non-dicritical singularities in \cite[Definition 3.6]{wang2023toric}. 
We show that (see Proposition~\ref{wang_defn}) Wang's definition and ours coincide on $\mb Q$-factorial toric varieties.
Therefore, we ask the following question:

\begin{question}
    Does Definition~\ref{nddef} agree with \cite[Definition 3.6]{wang2023toric} on any normal variety?
\end{question}

Then we provide toric descriptions for various singularities and study the relations among them.
To determine the singularities, we need to estimate the discrepancies of the exceptional divisors that might not be extracted by a sequence of toric birational morphisms; in other words, 
it is necessary to consider the blow-up along a non-torus-invariant center.
As a result, we get a foliation which is not toric but very close to being so, leading to the discussion of toroidal foliations and toroidal foliated pairs (see Definition~\ref{defn_toroidal}) inspired by 
\cite[Section 3.2]{ambro2021positivity}. 

\begin{proposition}[{$=$ Proposition~\ref{lc}}]
    Let $(\cF,D=\sum_id_iD_i)$ be a toroidal foliated pair (see Definition~\ref{defn_toroidal}) on a normal variety $X$. 
    Then $(\cF,D)$ is log canonical if and only if $d_i\leq\iota(D_i)$ for all $i$. 

    In particular, a toric foliated pair $(\cF_W,D=\sum_{\rho\in\Sigma(1)} d_\rho D_\rho)$ on a toric variety $X_\Sigma$ of a fan $\Sigma$ in $N\otimes\bR$ is log canonical if and only if $d_\rho\leq 1$ for $\rho\subseteq W$ and $d_\rho\leq 0$ for $\rho\nsubseteq W$. 
\end{proposition}

\begin{proposition}[{$=$ Proposition~\ref{can_term_toric}}]
    Let $(\cF,0)$ be a toroidal foliated pair on a normal variety $X$ with associated extended complex $(\Delta,W)$. 
    Then we have the following:
    \begin{enumerate}
        \item For each $\sigma\in\Delta$, $\Pi_{\sigma,\,W(\sigma)}$ has a unique facet not containing the origin. 
        \item $\cF$ is canonical if and only if for any $\sigma\in\Delta$, the only non-zero elements of $\Pi_{\sigma,\,W(\sigma)}\cap W(\sigma)\cap N_\sigma$ are contained in $\Pi_{\sigma,\,W(\sigma)}\cap\{u\mid\phi_{K_\cF}(u)=1\}$, the facet of $\Pi_{\sigma,\,W(\sigma)}$ not containing the origin. 
        \item For any $\sigma\in\Delta$, $\cF$ is terminal at the generic point of $V(\sigma)$ if and only if $\Pi_{\sigma,\,W(\sigma)}\neq\sigma$ and $\on{Relint}(\sigma)\cap\Pi_{\sigma,\,W(\sigma)}\cap W(\sigma)\cap N_\sigma=\emptyset$. 
    \end{enumerate}
\end{proposition}

\begin{proposition}[{$=$ Proposition~\ref{Fdlt_prop}}]
    Let $(\cF,D)$ be a toroidal foliated pair on a normal variety $X$ with the associated extended complex $(\Delta,W)$. 
    Suppose $D$ is effective. 
    Then $(\cF,D)$ is F-dlt (see Definition~\ref{defn_F_dlt}) if and only if the following statements hold true:
    \begin{enumerate}
        \item $\on{Supp}(D)\subseteq\bigcup_{\rho\subseteq W(\rho),\,\rho\in\Delta(1)}V(\rho)$ and $0\leq\on{mult}_{V(\rho)}D\leq 1$ for any $\rho\in\Delta(1)$. 
        \item For any $\sigma\in\Delta$ satisfying $\phi_{(K_{\cF}+D)}\vert_\sigma=0$, we have $\sigma$ is simplicial and $(\sigma,W(\sigma))$ is non-dicritical. The latter means that either $\on{Relint}(\sigma)\cap W(\sigma)\cap N_\sigma=\emptyset$ or $\sigma\subseteq W(\sigma)$. 
    \end{enumerate}
\end{proposition}

Utilizing the toric descriptions above, we conclude the following relations among various singularities:
\begin{theorem}
    Let $(\cF,D)$ be a toroidal foliated pair on a normal variety $X$. 
    \begin{enumerate}
        \item (Proposition~\ref{SimpleCanonical}) Suppose $(\cF,0)$ is a foliated log smooth pair on a smooth variety $X$. 
        Then $\cF$ has only canonical singularities.
        \item (Proposition~\ref{Fdlt_ND_prop}) If $(\cF,D)$ is F-dlt, then $\cF$ is non-dicritical. 
        \item (Proposition~\ref{CanImplyND}) If $(\cF,D)$ is canonical, then $\cF$ is non-dicritical. 
    \end{enumerate}
\end{theorem}

Then we show the existence of foliated log resolution (Theorem~\ref{FLogRes}) and F-dlt modification (Theorem~\ref{Fdlt_modification}) for toroidal foliated pairs of any rank and of any dimension.

Based on the combinatorial characterizations of singularities, we show that the FMMP works for log canonical toric foliated pairs on complete $\bQ$-factorial toric varieties, that is, non-dicritical singularities and F-dlt singularities are preserved under the FMMP. Furthermore, we show that the cone theorem holds true.

\begin{theorem}[{Propositions~\ref{div_contr}, \ref{MFS}, and \ref{flip})}]
    Let $(\cF_W,D)$ be a log canonical toric foliated pair on a complete $\bQ$-factorial toric variety $X_\Sigma$ with $D\geq 0$. 
    Then the FMMP works for $(\cF,D)$. 
    Moreover, being non-dicritical (resp. F-dlt) is preserved under the FMMP. 
\end{theorem}

\begin{theorem}[{$=$ Theorem~\ref{conethm}, Cone Theorem}]
    Let $(\cF_W,D)$ be a log canonical toric foliated pair on a complete $\bQ$-factorial toric variety $X_\Sigma$ with $D\geq 0$. 
    Then 
    \[\overline{\on{NE}}(X)_{K_{\cF_W}+D<0} = \sum\bR_{\geq 0}[M_i]\] 
    where $M_i$ are torus invariant rational curves tangent to $\cF_W$. 
\end{theorem}
Let $\mc F$ be a foliation on a normal variety $X$ and $Y\subseteq X$ be a subvariety.
The challenging part here is how to define that $Y$  is tangent to $\mc F$ when $Y\subseteq\on{Sing}(\mc F)$. We show that our definition of tangency (Definition~\ref{defn_tang_general}) generalizes \cite[Definition 2.12]{CS} (Proposition~\ref{tang_CS}) and has a nice description when the foliation is toric (Proposition~\ref{tang_combinatoric}). 

\section*{Acknowledgments}
The authors would like to thank National Center for Theoretical Sciences (NCTS) for the wonderful research environment. 
They would also like to express their gratitude to Iacopo Brivio, Paolo Cascini, Jungkai Chen, Shin-Yao Jow, Ching-Jui Lai, and Calum Spicer for helpful discussions. 
The authors were partially supported by Lab. Bir. Geom. Grant number 111-2123-M-002-012. 

\section{Preliminaries}\label{sec:preliminaries}
We will exclusively work over the field of complex numbers $\bC$. 
For any sheaves $\cM$ and $\cN$ on a normal variety $X$, we denote $(\cM\otimes\cN)^{**}$ and $(\cM^{\otimes n})^{**}$ as $\cM\boxtimes\cN$ and $\cM^{[n]}$, respectively. 

\subsection{Basics on foliations}
In this subsection, most of the definitions follow from \cite{CS} and \cite{druel2021foliation}. 
Let $X$ be a normal variety. A \emph{foliation} is a coherent subsheaf $\cF$ of the tangent sheaf $\cT_X$ such that
\begin{enumerate}
\item $\cF$ is saturated, that is $\cT_X/\cF$ is torsion-free, and
\item $\cF$ is closed under the Lie bracket.
\end{enumerate}

Let $r=\on{rank}(\cF)$ be the \emph{rank} of the foliation and $c=\dim X-r$ be the \emph{corank} of the foliation. 
The \emph{canonical divisor} $K_\cF$ is a Weil divisor on $X$ such that $\cO_X(-K_\cF)\cong\det\cF$. 

We define the \emph{normal sheaf} of $\cF$ as $\cN_{\cF}:=(\cT_X/\cF)^{[1]}$. 
By taking the $r$-th wedge product of $\cN_{\cF}^*\to\Omega_X^{[1]}$, we obtain a twisted form $\omega\in H^0(X,\Omega_X^{r}\boxtimes\on{det}\cN_{\cF})$. 
Here $\omega$ satisfies the following  properties:
\begin{enumerate}
\item The zero locus of $\omega$ has codimension at least two. 
\item $\omega$ is locally decomposable, meaning that locally $\omega=\bigwedge_i\omega_i$ where $\omega_i$ are $1$-forms.
\item $\omega$ is integrable, that is, $\diff\omega_i\wedge\omega=0$ for all $i$.
\end{enumerate}

Conversely, let $D$ be a Weil divisor and $\omega\in H^0(X,\Omega_X^{r}\boxtimes\cO_X(D))$ be a twisted form  
whose zero locus has codimension at least two in $X$. If $\omega$ is locally decomposable and integrable, then the kernel of
$\cT_X\to \Omega_X^{r-1}\boxtimes\cO_X(D)$ given by the contraction via $\omega$ is a foliation.

Let $\pi\colon Y \dashrightarrow X$ be a dominant rational map between normal varieties and $\cF$ be a foliation on $X$. 
We denote by $\pi^{-1}\cF$ the \emph{pullback foliation} on $Y$ (see, for example, \cite[Section 3.2]{druel2021foliation}). 
If $f\colon X\dashrightarrow X'$ is birational, then $f_*\mathcal{F}$ represents the pullback foliation on $X'$ induced by $f^{-1}$.

Let $X^\circ$ be the open subset of $X$ such that $\cF\vert_{X^\circ}$ is a subbundle of $T_{X^\circ}$. 
A \emph{leaf} $L$ is a maximal connected and immersed holomorphic submanifold $L\subseteq X^\circ$ such that $T_L = \cF\vert_L$. 

A foliation $\cF$ is called \emph{algebraically integrable} if its leaves are algebraic. 
Equivalently, an algebraically integrable foliation $\cF$ on $X$ is induced from a dominant rational map $f\colon X\dashrightarrow Y$ for some normal variety $Y$ (see, for example, \cite[Sections 3.2 and 3.6]{druel2021foliation}).

\begin{definition}[Singular locus]
Let $\cF$ be a foliation of rank $r$ on a normal variety $X$. 
We obtain a morphism $\phi\colon \Omega_X^{[r]} \to \cO_X(K_\cF)$ by taking the double dual of the $r$-th wedge product of $\Omega_X^{[1]}\to \cF^*$, which is induced by the inclusion $\cF\subseteq \cT_X$. 
We define the \emph{singular locus} of $\cF$, denoted by $\on{Sing}(\cF)$, as the co-support of the image of $\phi'\colon \Omega_X^{[r]}\boxtimes\cO_X(-K_\cF) \to \cO_X$. 
\end{definition}

\begin{definition}[Invariance]\label{defn_invariance}\hspace{1em}
\begin{enumerate}
\item
    Let $\cF$ be a foliation of rank $r$ on a normal variety $X$. 
    We say that a subvariety $S\subseteq X$ is $\cF$-\emph{invariant} if for any subset $U\subseteq X$ and any section $\partial\in H^0(U,\cF)$, we have 
    $\partial(\cI_{S\cap U})\subseteq\cI_{S\cap U}$ 
    where $\cI_{S\cap U}$ is the ideal sheaf of $S\cap U$. 
\item
    For any prime divisor $D\subseteq Y \xrightarrow{\pi} X$ over $X$ where $\pi$ is a birational morphism, we define $\iota(D)=0$ if $D$ is $\pi^{-1}\cF$-invariant and $\iota(D)=1$ if $D$ is non-$\pi^{-1}\mc F$-invariant. One can show that $\iota(D)$ is independent of the choice of the birational morphism $\pi$ that extracts $D$.
\end{enumerate}
\end{definition}

\begin{proposition}\label{sing_X_inv_prop}
    Let $\mc F$ be a foliation on a normal variety $X$. 
    Then $\on{Sing}(X)$ is $\cF$-invariant.
\end{proposition}
\begin{proof}
    By \cite[Theorem 5]{seidenberg67differential}, $\on{Sing}(X)$ is invariant under any derivation. 
    In particular, it is $\cF$-invariant. 
\end{proof}

\begin{lemma}[{\cite[Lemma 3.5]{druel2021foliation}}]\label{sing_F_inv_lem}
    Let $\cF$ be a foliation on a smooth variety $X$. 
    Then $\on{Sing}(\cF)$ is $\cF$-invariant. 
\end{lemma}

We recall the definition that a general subvariety $Z$ is tangent to a foliation $\cF$. 
\begin{definition}[Tangency]\label{tangency}
Let $X$ be a normal variety and $\cF$ be a foliation of any rank.
Given a (possibly analytic) subvariety $Z\subseteq X$ not contained in $\on{Sing}(X)\cup\on{Sing}(\cF)$, we say $Z$ is \emph{tangent} to $\cF$ if, over $X\setminus\big(\on{Sing}(X)\cup\on{Sing}(\cF)\cup\on{Sing}(Z)\big)$, the inclusion $\cT_Z\subseteq\cT_X$ factors through $\cF$. 
\end{definition}

\subsection{Basics on toric varieties}
In this paper, every toric variety is assumed to be normal. Our notations closely follow \cite{CLS11}. 

Let $N\simeq \bZ^n$ be a lattice of rank $n$ and $M:=\on{Hom}(N,\mb Z)$ be its dual lattice. We write $N\otimes\mb R$ and $N\otimes\mb C$.
A \emph{fan} $\Sigma$ in $N\otimes\bR$ is a finite collection of rational, strongly convex, polyhedral cones $\sigma\subseteq N\otimes\bR$, such that each face $\tau$ of a cone $\sigma\in\Sigma$ belongs to $\Sigma$ and the intersection of any two cones in $\Sigma$ is a face of each. 
For any $k\in\bZ_{\geq 0}$, denote the set of all $k$-dimensional cones in $\Sigma$ by $\Sigma(k)$, and denote the set of all $k$-dimensional faces of $\sigma\in\Sigma$ by $\sigma(k)$.
We write $\tau\preceq\sigma$ when $\tau$ is a face of $\sigma$. 

For each cone $\sigma\in\Sigma$, the affine toric variety associated with $\sigma$ is  $U_{\sigma,\,N}=\on{Spec}\bC[\sigma^\vee\cap M]=\on{Spec}\bC[\chi^m\mid m\in\sigma^\vee\cap M]$ where $\sigma^\vee$ is the dual cone of $\sigma$. 
A cone $\sigma\in\Sigma$ is said to be \emph{smooth with respect to $N$} if the primitive generators of the rays in $\sigma(1)$ form part of a $\bZ$-basis for $N$ (or equivalently, $U_{\sigma,\,N}$ is smooth).
If $\tau\preceq \sigma$ are two cones in $\Sigma$, there is an open immersion $U_{\tau,N}\hookrightarrow U_{\sigma, N}$.
The toric variety $X_{\Sigma,\,N}$ of the fan $\Sigma$ is constructed by gluing all $U_{\sigma,\,N}$ together via $\Sigma$. 
The dense torus $U_{\{0\},\,N}=\on{Spec}\bC[M]\subseteq X_{\Sigma,N}$ is denoted by $T_N$. 
The action of $T_N$ on itself can be extended to an action on $X_{\Sigma,\,N}$. 
We will omit $N$ in the subscript when $N$ is clear. 

For each $\sigma\in\Sigma$, $\on{Relint}(\sigma)$ denotes the relative interior of $\sigma$, $O_\sigma$ denotes the $T$-orbit of the distinguished point $x_\sigma$, and $V_\sigma$ denotes the closure of $O_\sigma$ in $X_\Sigma$ (see \cite[Chapter 3]{CLS11} for further details). 
If $\rho\in\Sigma(1)$ is a ray, then $V_\rho$ is a divisor and will also be denoted by $D_\rho$.

\section{Toric foliations}\label{sec:toric_fol} 
Let $X=X_\Sigma$ be the toric variety defined by a fan $\Sigma$ in $N\otimes\bR$. 
A subsheaf $\cF\subseteq \cT_X$ is called $T$-\emph{invariant} or \emph{torus invariant} if for any $t\in T$ we have $t^*\mc F=\mc F$ as subsheaves under the natural isomorphism $t^*\cT_X\simeq \cT_X$. 
A foliation $\cF\subseteq \cT_X$ is called a \emph{toric foliation} if $\cF$ is $T$-invariant. 

\begin{proposition}\label{1-1}
Let $\Sigma$ be a fan in $N\otimes\bR$ and $X_{\Sigma,\,N}$ be the toric variety defined by $\Sigma$. 
Then there is a one-to-one correspondence between the set of toric foliations on $X_{\Sigma,\,N}$ and the set of complex vector subspaces $W\subseteq N\otimes\bC$. 
\end{proposition} 
\begin{proof}
    If $\cF$ is a toric foliation, then $\cF\vert_T$ is a $T$-invariant vector sub-bundle of the tangent bundle $\cT_T$, which gives rise to a complex vector subspace $W:=(\cF\vert_T)_{1}\subseteq \cT_{T,\,1}=N\otimes\bC$. 
    By \cite[Lemma 1.8]{hacon2021birational}, any two foliations that agree on a Zariski open dense subset must be the same.
    Therefore, $\cF$ is uniquely determined by $W$. 
    
    Conversely, given any complex vector subspace $W\subseteq N\otimes\bC$, we can extend it via the $T$-action to a $T$-invariant subbundle $\cE\subseteq \cT_T$. 
    Since the Lie bracket on $\cT_{T}$ is trivial, $\cE$ becomes a foliation.  
    We can then uniquely extend $\cE$ to a foliation $\cF$ on $X_{\Sigma,\,N}$, and it is easy to see that $\mc F$ is $T$-invariant. 
\end{proof}

We will use $\cF_{W,\,\Sigma,\,N}$ to denote the toric foliation on $X_{\Sigma,\,N}$ corresponding to the complex vector subspace $W\subseteq N\otimes\bC$. 
If we have another fan $\Sigma'$ in the same $N\otimes\bR$, the pullback foliation on $Y = X_{\Sigma',\,N}$ is nothing but $\cF_{W,\,\Sigma',\,N}$. Hence we can unambiguously write $\mc F_W$ to denote the pullback foliation on any birational model obtained by modifying the defining fan.

\subsection{Local generators}\label{loc_generator}
In this subsection, we consider a fixed toric foliation $\mc F_W$ on a toric variety $X_{\Sigma,\,N}$ defined by a fan $\Sigma$ in $N\otimes\bR$. In \cite{pang2015harder}, a set of local generators for $\cF_W$ is provided.\footnote{Note that in \cite{pang2015harder}, it should be $N\otimes\bC$ instead of $N\otimes\bR$.} We include it here for the convenience of the readers. 
Recall that $W\subseteq N\otimes\bC$ is a complex vector subspace and $M$ denotes the dual lattice of $N$.  For any $v\in N\otimes\bC$, define
\[
\delta_v\colon \mb C[M]\rightarrow \mb C[M],\, \chi^m\mapsto \langle m,v \rangle\chi^m.
\]
Then $\delta_v\in\on{Der}_{\mb C}(\mb C[\sigma^\vee\cap M],\mb C[\sigma^\vee\cap M])$ for any strongly convex rational polyhedral cone $\sigma$, and we can regard $\delta_v$ as a $T_N$-invariant global section of $\cT_{U_\sigma}$.
If $\{m_1,\ldots,m_n\}$ is a basis for $M$ and $x_i=\chi^{m_i}$, then $\delta_v$ can be written as 
\[
    \delta_v=\langle m_1,v\rangle x_1\frac{\partial}{\partial x_1}+\cdots+\langle m_n,v\rangle x_n\frac{\partial}{\partial x_n}.
\]
We have the following lemma:

\begin{lemma}[{\cite[Lemma 2.1.10, 2.1.12]{pang2015harder}}]\label{local_gen}
    Let $\Sigma$ be a fan in $N\otimes\bR$ and $W$ be an $r$-dimensional complex vector subspace of $N\otimes\bC$. 
    For any ray $\rho\in\Sigma(1)$ with the primitive generator $v_\rho$, we make the following choices:
    \begin{itemize}
        \item  If $\rho\subseteq W$, choose $v_2,\ldots, v_n$ in $N\otimes\bC$ so that $\{v_\rho, v_2,\ldots,v_r\}$ is a basis for $W$.
        \item  If $\rho\nsubseteq W$, choose a basis $\{v_1, \ldots, v_r\}$ for $W$.
        \item  Choose a basis $\{m_2, \ldots, m_n\}$ for $\rho^\perp\cap M$.
        \item  Choose an element $m_\rho\in \rho^\vee\cap M$ such that $\langle m_\rho, v_\rho\rangle=1$.  Hence we have $\rho^\vee\cap M\cong\bZ_{\geq 0}m_\rho\oplus\bigoplus_{i=2}^n\bZ m_i$
        and $U_\rho=\on{Spec}\mb C[\chi^{m_\rho},\chi^{\pm m_2},\ldots,\chi^{\pm m_n}]$.
    \end{itemize}
    Then we have the following:
    \begin{enumerate}
        \item For any $v\in N\otimes\bC$, we have 
        \[\delta_{v}\vert_{U_\rho} = \langle m_\rho, v\rangle\chi^{m_\rho}\frac{\partial}{\partial \chi^{m_\rho}}+\sum_{i=2}^n\langle m_i, v\rangle\chi^{m_i}\frac{\partial}{\partial \chi^{m_i}}.\]
        \item On $U_\rho$, $\mc F_W$ is generated by
        \[\begin{array}{ll}
        \delta_{v_1},\ldots, \delta_{v_r} & \textnormal{if $\rho\nsubseteq W$} \\
        \frac{1}{\chi^{m_\rho}}\delta_{v_\rho}, \delta_{v_2},\ldots, \delta_{v_r} & \textnormal{if $\rho\subseteq W$.} 
        \end{array}\] 
    \end{enumerate}
\end{lemma}

\begin{corollary}\label{basics}
    Let $\cF_W$ be a toric foliation on a toric variety $X_\Sigma$ defined by a fan $\Sigma$ in $N\otimes\bR$ and a complex vector subspace $W\subseteq N\otimes\bC$. Then for any $\rho\in\Sigma(1)$, $D_\rho$ is $\mc F_W$-invariant if and only $\rho\nsubseteq W$.
\end{corollary}
\begin{proof}
    By considering $\mc F_W\vert_{U_\rho}$ and $D_\rho\cap U_\rho$, we may assume that $X_\Sigma=U_\rho$. 
    To check if $D_\rho$ is $\cF_W$-invariant, it suffices to check whether the ideal $\cI_{D_\rho}\subseteq\mb C[U_\rho]$ is invariant under the derivations in $\Gamma(U_\rho,\mc F_W)$ as $\cF_W$ is the sheaf of $\cO_{U_\rho}$-modules given by the $\bC[U_\rho]$-module $\Gamma(U_\rho,\cF_W)$. 
    We will use the notations in Lemma~\ref{local_gen}. 
    
    If $\rho\nsubseteq W$, then the ideal $(\chi^{m_\rho})\subseteq \mb C[\chi^{m_\rho},\chi^{\pm m_2},\ldots,\chi^{\pm m_n}]$ is invariant under the generators $\delta_{v_1},\ldots,\delta_{v_n}$ of $\mc F_W$. 
    Hence, $D_\rho$ is $\mc F_W$-invariant. 
    On the other hand, if $\rho\subseteq W$, then $\frac{\partial}{\partial\chi^{m_\rho}}$ is a global section of $\mc F_W$ and $\frac{\partial}{\partial\chi^{m_\rho}}\chi^{m_\rho}=1\notin(\chi^{m_\rho})$. 
    Therefore, $D_\rho$ is not $\cF_W$-invariant.
\end{proof}

\begin{remark}\label{local_gen_extend} 
    Let $N\simeq\mb Ze_1\oplus\cdots\oplus\mb Z e_n$ and let $\sigma=\on{Cone}(e_1,\ldots,e_n)$.
    Let $\{m_1,\ldots,m_n\}$ be the $\mb Z$-basis for $M$ which is dual to $\{e_1,\ldots,e_n\}$. 
    Then we have $\sigma^{\vee}=\on{Cone}( m_1,\ldots,m_n)$. 
    After re-indexing, we can assume that 
    $e_i\in W$ if and only if $1\leq i\leq \ell$. 
    Let $\{v_1,\ldots,v_r\}$ be a $\mb C$-basis for $W$ such that $v_i=e_i$ for $1\leq i\leq \ell$. 
    Then by Lemma \ref{local_gen}, $\mc F_W$ is generated by
    \[\frac{1}{\chi^{m_1}}\delta_{v_1},\ldots,\frac{1}{\chi^{m_\ell}}\delta_{v_\ell},\delta_{v_{\ell+1}},\ldots,\delta_{v_{r}}\in\on{Der}_{\mb C}(\mb C[\sigma^\vee\cap M],\mb C[\sigma^\vee\cap M])\] on $U=\bigcup_{\rho\in\sigma(1)}{U_\rho}$.
    Since $\mc F_W$ is reflexive, it is normal in the sense of \cite[Definition 1.1.11]{MR2815674}.
    We conclude that $\mc F_W$ is generated on $U_\sigma$ by the same set of generators. In particular, $\mc F_W$ is always locally free if $X_\Sigma$ is smooth, since the above argument shows that the fiber dimension of $\mc F_W$ is $r$ at any closed point.
\end{remark}

Let $X$ be a smooth variety and $D=\sum D_i$ be a simple normal crossing divisor. Then for each component $D_k$ of $D$, one can define the Poincar\'e residue map
$R_k\colon \Omega^1_X(\log D)\rightarrow \mc O_{D_k}$
which fits into the short exact sequence 
\[
0\rightarrow \Omega_X^1\rightarrow \Omega_X^1(\log D)\rightarrow \bigoplus \mc O_{D_i}\rightarrow 0. 
\]
See \cite[(8.1.6)]{CLS11} and \cite[p.254]{psmixedhodge} for details.
Taking the induced long exact sequence with respect to $\sheafhom_{\mc O_X}(-,\mc O_X)$ and noticing that $\sheafext^1_{\mc O_X}(\mc O_{D_k},\mc O_X)\simeq \mc O_{D_k}(D_k)$, we get the exact sequence
\[
0\rightarrow\cT_X(-\log D)\rightarrow\cT_X\rightarrow \bigoplus \mc O_{D_i}(D_i)\rightarrow 0,
\]
where $\cT_X(-\log D)$ is the sheaf of vector fields that vanish along $D$.
The morphism
\[
R_k^\vee\colon \cT_X \rightarrow \mc O_{D_k}(D_k)
\]
which appears in the connecting morphism can be thought of as the dual version of the  Poincar\'e residue map.

\begin{example}
Let $X=\mb A^n$ with coordinates $x_1,\ldots,x_n$ and let $D_1$ be the divisor defined by $x_1=0$.
We can write a vector field as $\delta=f_1\frac{\partial}{\partial x_1}+\cdots+f_n\frac{\partial}{\partial x_n}$ where each $f_k$ is regular. Then $R_1^\vee(\delta)$ is just $\bar f_1\otimes \frac{1}{x_1}\in \mb C[x_1,\ldots,x_n]/(x_1)\otimes_{\mb C[x_1,\ldots,x_n]}\mb C(x_1,\ldots,x_n)$, which is a section of $\mc O_{D_1}\otimes_{\mc O_X}\mc O_X(D_1)=\mc O_{D_1}(D_1)$. One can verify that  $R_1^\vee(\delta)$ is independent of the choice of coordinates.
\end{example}

Lemma~\ref{local_gen} can be reformulated as follows: Let $D=\sum_{\rho\in\Sigma(1)} D_\rho$ and $U=\bigcup_{\rho\in\Sigma(1)}U_\rho$.
Lemma~\ref{local_gen} shows that the map $v\otimes 1\mapsto \delta_v$ defines a map $W\otimes_{\mb C}\mc O_U\rightarrow \mc F_W\vert_{U}$. On $U_\rho$, we have the short exact sequence
\[
 0\rightarrow W\otimes_\mb C\mc O_{U_\rho}\rightarrow  \mc F_W\vert_{U_\rho}\overset{R^\vee}{\rightarrow} \mc O_{U_\rho}(D_\rho\cap U_\rho)\rightarrow 0
\]
if $\rho\subseteq W$, and $W\otimes_\mb C\mc O_{U_\rho}\simeq \mc F_W\vert_{U_\rho}$ if $\rho\nsubseteq W$. Hence there is a short exact sequence
\[
0 \rightarrow W\otimes_{\mb C}\mc O_U \rightarrow  \mc F_W\vert_{U}\overset{R^\vee}{\rightarrow} \bigoplus_{\rho\in\Sigma(1),\rho\subseteq W}\mc O_{D_\rho\cap U}(D_\rho\cap U)\rightarrow 0.   
\]
The induced long exact sequence with respect to $\sheafhom_{\mc O_U}(-,\mc O_U)$ gives
\begin{equation}\label{ses_on_U}
0\rightarrow \mc F_W^*\vert_U\rightarrow W^*\otimes_{\mb C}\mc O_U\rightarrow \bigoplus_{\rho\in\Sigma(1),\rho\subseteq W}\mc O_{D_\rho\cap U}\rightarrow 0.
\end{equation}
Tracing the maps along the process, we see that $W^*\otimes_{\mb C}\mc O_{U}\rightarrow \mc O_{D_\rho\cap U}$ is given by $f\otimes 1\mapsto \langle f,v_\rho\rangle 1$.

\begin{theorem}\label{first_ses}
    Let $\cF_W$ be a toric foliation on a toric variety $X_\Sigma$ defined by a fan $\Sigma$ in $N\otimes\bR$ and a complex vector subspace $W\subseteq N\otimes\bC$. 
    \begin{enumerate}
        \item There is an exact sequence
        \[
        0\rightarrow \mc F_W^*\rightarrow W^*\otimes_{\mb C}\mc O_{X_\Sigma}\rightarrow \bigoplus_{\rho\in\Sigma(1),\,\rho\subseteq W}\mc O_{D_\rho}.
        \]
        Here the map $W^*\otimes_{\mb C}\mc O_{X_\Sigma}\rightarrow \mc O_{D_\rho}$ is given by $f\otimes 1\mapsto\langle f,v_\rho\rangle 1$.
        \item If $X_\Sigma$ is $\mb Q$-factorial, then the map on the right is surjective, that is, 
        \[
        0\rightarrow \mc F_W^*\rightarrow W^*\otimes_{\mb C}\mc O_{X_\Sigma}\rightarrow \bigoplus_{\rho\in\Sigma(1),\,\rho\subseteq W}\mc O_{D_\rho}\rightarrow 0
        \]
        is exact.
    \end{enumerate}
\end{theorem}
\begin{proof}
    Let $U=\bigcup_{\rho\in\Sigma(1)}U_\rho$ and consider the push-forward of (\ref{ses_on_U}) via the inclusion $j\colon U\rightarrow X_\Sigma$. The rest is similar to the proof of \cite[Theorem 8.1.4]{CLS11}. 
\end{proof}

\begin{proposition}\label{can_divisor_prop}
Let $\cF=\cF_W$ be a toric foliation on a toric variety $X_\Sigma$ defined by a fan $\Sigma$ in $N\otimes\bR$ and a complex vector subspace $W\subseteq N\otimes\bC$. 
Then $K_\cF+\sum_{\rho\in\Sigma(1),\,\rho\subseteq W}D_\rho\sim 0$. 
In particular, we can choose $K_\cF=-\sum_{\rho\in\Sigma(1),\,\rho\subseteq W}D_\rho$. 
\end{proposition}
\begin{proof}
    Apply Equation~(\ref{ses_on_U}) and note that $\on{codim}(X_\Sigma\setminus U)\geq 2$.
\end{proof}

\subsection{Singular locus of a toric foliation}
In this subsection, we present a combinatorial criterion to determine whether the orbit closure is contained in the singular locus of a toric foliation. 
To establish this criterion, we rely on the following lemma, which allows us to reduce the problem to the smooth case. 
\begin{lemma}\label{fol_sing_lem}
    Let $N$ be a lattice of rank $n$, $\sigma$ be a simplicial strongly convex rational polyhedral cone of dimension $n$, and $W$ be a complex vector subspace of $N\otimes\bC$.  
    There is a sublattice $N'$ of $N$ such that $\sigma$ is smooth with respect to $N'$ and $\on{Sing}(\cF_{W,\,N'}) = \pi^{-1}(\on{Sing}(\cF_{W,\,N}))$ where $\pi\colon U_{\sigma,\,N'}\to U_{\sigma,\,N}$ is a finite cover induced by the inclusion $N'\hookrightarrow N$ and $\cF_{W,\,N'}$ (resp. $\cF_{W,\,N}$) is the toric foliation on $U_{\sigma,\,N'}$ (resp. $U_{\sigma,\,N'}$) given by $W$. 
\end{lemma}
\begin{proof}
    Let $N'$ be the sublattice of $N$ generated by all $v_\rho$ for $\rho\in\sigma(1)$. 
    So $\sigma$ is a smooth cone with respect to $N'$. 
    Moreover, it introduces a finite covering $\pi\colon U_{\sigma,\,N'}\to U_{\sigma,\,N}$.  

    As $\on{Sing}(\cF_{W,\,N})$ is torus invariant, there are some cones $\tau_i\preceq\sigma$ such that $\on{Sing}(\cF_{W,\,N}) =\bigcup_{i=1}^\ell V_{\tau_i,\,N}$, where each $V_{\tau_i,\,N}$ is an irreducible component of $\on{Sing}(\cF_{W,\,N})$.
    Now we consider 
    \begin{align*}
        \mc C&=\{\tau\mid\tau_i\preceq\tau\preceq\sigma \textnormal{ for some } i\}, \\
        \Sigma'_0 &=\{\tau\mid\tau\preceq\sigma\}\setminus \mc C, \textnormal{ and}\\
        \Sigma'_i &=\Sigma'_0\cup\{\tau_i\}.
    \end{align*}
    Note that $\Sigma'_0$ and $\Sigma'_i$ are indeed fans. 
    Actually, we have $X_{\Sigma'_0,\,N}=U_\sigma\setminus\on{Sing}(\mc F_{W,\,N})$ and $X_{\Sigma'_i,\,N}=(U_\sigma\setminus\on{Sing}(\mc F_{W,\,N}))\cup O_{\tau_i}$. 
    One can check that $X_{\Sigma'_i,\,N}$ is an open subscheme of $U_{\sigma,\,N}$  for $0\leq i\leq\ell$, and thus the base change $\pi'\colon X_{\Sigma'_i,\,N'}\to X_{\Sigma'_i,\,N}$ is finite and surjective. 

    Since $X_{\Sigma'_0,\,N}$ has no foliation singularities, by \cite[Proposition 5.13]{druel2021foliation}, $X_{\Sigma'_0,\,N'}$ has no foliation singularities, from which we have $\on{Sing}(\cF_{W,\,N'})\subseteq\bigcup_{i=1}^\ell V_{\tau_i,\,N'}$. 
    If the containment is strict, then there is an $i\neq 0$ such that $X_{\Sigma'_i,\,N'}$ has no foliation singularities. 
    Thus $X_{\Sigma'_i,\,N}$ has no foliation singularities again by \cite[Proposition 5.13]{druel2021foliation}, which contradicts $V_{\tau_i,\,N}\subseteq\on{Sing}(\cF_{W,\,N})$.  
    We conclude that $\on{Sing}(\cF_{W,\,N'}) = \bigcup_{i=1}^\ell V_{\tau_i,\,N'}$ and therefore $\on{Sing}(\cF_{W,\,N'}) = \pi^{-1}(\on{Sing}(\cF_{W,\,N}))$. 
\end{proof}

\begin{proposition}\label{singularlocus}
    Let $\cF_W$ be a toric foliation on a $\bQ$-factorial toric variety $X_\Sigma$ defined by a fan $\Sigma$ in $N\otimes\bR$ and a complex vector subspace $W\subseteq N\otimes\bC$. 
    Then for any $\tau\in\Sigma$, $V_\tau\nsubseteq\on{Sing}(\cF_W)$ if and only if $W\cap\bC\tau=\on{Span}_\bC(S)$ for some $S\subseteq \tau(1)$ with the convention $\on{Span}_\bC(\emptyset)=0$.
\end{proposition}

\begin{example}
    Let $N=\mb Ze_1\oplus\mb Ze_2\oplus\mb Ze_3$, $\tau=\on{Cone}(e_1,e_2)$, $W_1 = \bC e_3$, and $W_2 = \bC(e_1+ie_2)$. 
    We have $W_1\cap\bC\tau = \{0\}=\on{Span}_\bC(\emptyset)$, so $V_\tau\nsubseteq\on{Sing}(\mc F_{W_1})$ by Proposition~\ref{singularlocus}. 
    On the other hand, $W_2\cap\bC\tau=W_2$, which is not $\{0\}$, $\bC e_1$, $\bC e_2$, or $\bC e_1 + \bC e_2$. Hence $V_\tau\subseteq\on{Sing}(\mc F_{W_2})$.
\end{example}

\begin{proof}[{Proof of Proposition~\ref{singularlocus}}]
    By Lemma~\ref{fol_sing_lem}, we can assume that $X=X_\Sigma$ is smooth.  
    As this is a local problem, 
    we may assume that $N=\mb Ze_1\oplus\cdots\oplus\mb Ze_n$, $\sigma=\on{Cone}(e_1,\ldots,e_n)$, $X=U_\sigma$, and $\tau\preceq\sigma$. Note that $V_\tau\subseteq \on{Sing}(\mc F_W)$ if and only if $x_\tau\in\on{Sing}(\mc F_W)$ where $x_\tau$ is the distinguished point corresponding to $\tau$. Since both $\mc F_W$ and $\cT_X$ are locally free (Remark~\ref{local_gen_extend}), we have
    \begin{align*}
    & x_\tau\notin\on{Sing}(\mc F_W)\\
    \Leftrightarrow {} & \mc F_W\ \text{is a subbundle of}\ \cT_{X}\ \text{at}\ x_\tau\\
    \Leftrightarrow {} & \mc F_{W}\otimes_{\mc O_{X}}\mb C(x_\tau)\rightarrow \mc T_{X}\otimes_{\mc O_{X}}\mb C(x_\tau)\ \text{is injective}\\
    \Leftrightarrow {} & \Omega^1_{X}\otimes_{\mc O_{X}}\mb C(x_\tau)\rightarrow \mc F_{W}^*\otimes_{\mc O_{X}}\mb C(x_\tau)\ \text{is surjective}\\
    \Leftrightarrow {} & \Omega^1_X\otimes_{\mc O_X}\mc O_{X,x_\tau}\rightarrow \mc F_W^*\otimes_{\mc O_X}\mc O_{X,x_\tau}\ \text{is surjective}.
    \end{align*}
    Applying Theorem~\ref{first_ses} and \cite[Theorem 8.1.4]{CLS11} and localizing at $x_\tau$, we have the commutative diagram
    \[
    \begin{tikzcd}
        0\arrow[r]& \Omega^1_X\otimes_{\mc O_X}\mc O_{X,x_\tau}\arrow[r]\arrow[d,"\phi"] &M_\mb C\otimes_{\mb C}\mc O_{X,x_\tau}\arrow[r]\arrow[d,"\alpha"] &\bigoplus_{\rho\in\tau(1)}\mc O_{D_\rho}\otimes_{\mc O_X}\mc O_{X,x_\tau} \arrow[r]\arrow[d,"\beta"] &0 \\ 
        0\arrow[r]& \mc F_W^*\otimes_{\mc O_X}\mc O_{X,x_\tau}\arrow[r] & W^*\otimes_{\mb C}\mc O_{X,x_\tau}\arrow[r] &\bigoplus_{\rho\in\tau(1),\rho\subseteq W}\mc O_{D_\rho}\otimes_{\mc O_X}\mc O_{X,x_\tau} \arrow[r] &0
    \end{tikzcd}
    \]
    where the rows are exact. The induced exact sequence
    \[
    \ker\alpha\rightarrow \ker\beta\rightarrow \on{coker}\phi\rightarrow \on{coker}\alpha=0
    \]
    tells us that $\phi$ is surjective if and only if $\ker\alpha\rightarrow \ker\beta$ is surjective. 
    We have 
    \begin{align*}
        \ker\alpha &= W^{\perp}\otimes_{\bC}\cO_{X,x_\tau} \\
        \ker\beta &= \bigoplus_{\rho\in\tau(1),\rho\nsubseteq W}\cO_{D_\rho}\otimes_{\cO_X}\cO_{X,x_\tau}
    \end{align*}
    and $W^{\perp}\otimes_{\mb C}\mc O_{X,x_\tau}\rightarrow \mc O_{D_\rho}\otimes_{\mc O_X}\mc O_{X,x_\tau}$ is defined by
    \[
    f\otimes 1\mapsto \langle f,v_\rho \rangle 1\otimes 1.
    \]
    Hence, $\ker\alpha\rightarrow\ker\beta$ is surjective if and only if $\lambda\colon W^\perp \rightarrow V:=\sum_{\rho\in\tau(1),\rho\nsubseteq W} \mb C v_\rho$ defined by $\lambda(f)= \sum_{\rho\in\tau(1),\rho\nsubseteq W}\langle f,v_\rho\rangle v_\rho$ is surjective. 
    The map $M_\mb C\rightarrow V$ defined by $m\mapsto \sum_{\rho\in\tau(1),\rho\nsubseteq W}\langle m,v_\rho\rangle v_\rho$ is surjective with kernel $V^\perp$. Hence $\lambda$ is surjective if and only if $W^\perp+V^\perp=M_\mb C$, which is equivalent to $W\cap V=0$. One can check that this is exactly what we want.
\end{proof}

\subsection{Properties}
\begin{proposition}
    Let $\cF_1$ and $\cF_2$ be two foliations on a normal variety $X$. 
    The intersection $\cF_1\cap\cF_2$ also gives a foliation. 
\end{proposition}
\begin{proof}
    It is clear that $\cF_1\cap\cF_2$ is closed under the Lie bracket as both $\cF_1$ and $\cF_2$ are closed under the Lie bracket. 
    It remains to show that $\cF_1\cap\cF_2$ is saturated, that is, $\cT_X/(\cF_1\cap\cF_2)$ is torsion-free. Hence, we need to show that  
    the stalk $(\cT_X/(\cF_1\cap\cF_2))_p$ is torsion-free for each $p\in X$. 
    
    It suffices to show that $M/(N_1\cap N_2)$ is torsion-free if each $N_i$
    is an $R$-submodule of $M$ such that $M/N_i$ is torsion-free.
    Suppose $\overline{m}\in M/(N_1\cap N_2)$ and $\overline{rm}=0\in M/(N_1\cap N_2)$ for some $r\in R\setminus\{0\}$. Then $\overline{rm}=0\in M/N_1$. Hence $m\in N_1$ as $M/N_1$ is torsion-free. Similarly, $m\in N_2$, and thus $m\in N_1\cap N_2$.
\end{proof}

\begin{remark}\label{rmk_properties}
\begin{enumerate}
    \item Let $X_\Sigma$ be the toric variety of a fan $\Sigma$ in $N\otimes\bR$. 
    If $\cF_1$ and $\cF_2$ are toric foliations on $X_\Sigma$ given by complex vector subspaces $W_1$ and $W_2$ in $N\otimes\bC$, respectively, then $\cF_1\cap\cF_2$ is the toric foliation given by $W_1\cap W_2$. 
    In other words, $\cF_{W_1}\cap\cF_{W_2} = \cF_{W_1\cap W_2}$. 
    
    \item Let $N=\bZ e_1\oplus\cdots\oplus\bZ e_n$ with dual basis $\{m_1,\ldots,m_n\}$. 
    We consider the toric foliation $\cF_H$ of corank one given by a hyperplane $H\subseteq N\otimes\bC = \cT_{T,\,1}$ on a toric variety $X_\Sigma$, where $H$ can be written as $\{v\in N\otimes\bC\mid \big\langle \sum_{i=1}^n a_im_i, v\big\rangle = 0\}$ for some $a_i\in\bC$. 
    Note that the torus $T:=U_{0}$ has coordinates $x_i:=\chi^{m_i}$ with $i\in \{1,\ldots,n\}$.  
    Then $\cF_H\vert_T$ is given by $\ker(\omega)$ where $\omega=\sum_{i=1}^n a_i\frac{\diff x_i}{x_i}$ is a $T$-invariant $1$-form. 

    Let $v=\sum_{i=1}^n b_ie_i \in N\otimes\bC$ where $b_i\in\bC$. 
    In Section \ref{loc_generator}, we have seen that there is a $T$-invariant derivation $\delta_v=\sum_{i=1}^n b_ix_i\frac{\partial}{\partial x_i}$.
    If $v\in H$, then we have $\sum_{i=1}^n a_ib_i=0$ and thus
    $\delta_v\in\ker(\omega)$. That is, $\delta_v$ gives a $T$-invariant global section of $\mc F_H$ (recall that $\delta_v$ is a $T$-invariant derivation on any $U_\sigma$). This can be seen from Lemma~\ref{local_gen} as well.

    \item Moreover, if $\cF_{H_1},\ldots,\cF_{H_\ell}$ are distinct toric foliations on a toric variety $X_\Sigma$ given by hyperplanes $H_1,\ldots,H_\ell$, respectively, then the intersection foliation $\bigcap_{i=1}^\ell\cF_{H_i} = \cF_W$ is a toric foliation given by $W=\bigcap_{i=1}^\ell H_i$. 
    Let $\cF_{H_i}$ be given by some $T$-invariant $1$-form $\omega_i$. 
    Then $\cF_W$ is also given by the kernel of the contraction via $\tilde{\omega}$, where $\omega=\bigwedge_{i=1}^\ell\omega_i$, $\omega=f\tilde{\omega}$ for some regular function $f$ on $X_\Sigma$, and the zero locus of $\tilde{\omega}$ has codimension at least two in $X_\Sigma$. 
\end{enumerate}
\end{remark}

\begin{proposition}\label{restriction}\hspace{1em}
\begin{enumerate}
\item Suppose $f\colon X_{\Sigma,\,N}\to X_{\Sigma',\,N'}$ is a surjective toric morphism defined by a surjective map $\tilde{f}\colon N\to N'$ between lattices. 
Let $W=\on{ker}(\tilde{f})\otimes \bC\subseteq N\otimes\bC$.
Then any fiber of $f$ that intersects $T_N$ is the closure of a leaf of $\mc F_W$. 
\item Let $\cF_W$ be a toric foliation on a toric variety $X_\Sigma$, where $\Sigma$ is a fan in $N\otimes\bR$ and $W\subseteq N\otimes\bC$ is a complex vector subspace. 
Then for any $\rho\in\Sigma(1)$, there is an induced toric foliation on $D_\rho$, given by $\overline{W+\bC\rho}\subseteq (N\otimes\bC)/\bC\rho=(N/\bZ v_\rho)\otimes\bC$ where $v_\rho$ is the primitive generator of $\rho$. 
\end{enumerate}
\end{proposition}
\begin{proof}
    For (1), we may replace $X_{\Sigma,\,N}$ and $X_{\Sigma',\,N'}$ by $T_N$ and $T_{N'}$, respectively. Then fibers of $f$ correspond to leaves of $\mc F_W$.
    For (2), note that $D_\rho$ is a toric variety given by $\on{Star}(\rho)$, which is a fan in $(N\otimes\bR)/\mb R\rho$. (See \cite[paragraph before Proposition 3.2.7]{CLS11} for more details on $\on{Star}(\rho)$)
\end{proof}

\begin{proposition}\label{alg_int_prop}
    Let $X_\Sigma$ be a toric variety of a fan $\Sigma$ in $N\otimes\bR$. 
    Then the following two statements are equivalent:
    \begin{enumerate}
        \item $W=N'\otimes\bC$ for some sublattice $N'\subseteq N$. 
        \item The toric foliation $\cF_W$ given by $W$ is algebraically integrable. 
    \end{enumerate}
\end{proposition}
\begin{proof}
    Suppose $W=N'\otimes\bC$ for some sublattice $N'\subseteq N$. 
    We consider the quotient lattice $\overline{N}=N/N'$. 
    Then the image of $W$ is $\{\overline{0}\}$. 
    This introduces a toric morphism $T_N\to T_{\overline{N}}$. 
    As $T_N\subseteq X_\Sigma$, we have a dominant rational map $f\colon X_\Sigma\dashrightarrow T_{\overline{N}}$, which induces the foliation $\cF_W$. 
    Hence, $\cF_W$ is algebraically integrable. 

    Conversely, suppose $\cF_W$ is algebraically integrable.  
    Let $T$ be the torus in $X_\Sigma$. 
    Then the leaf $L$ of $\cF_W\vert_T$ through $1\in T$ is algebraic. 
    Thus, $\cT_{L,\,1}$ is a rational vector subspace of $\cT_{T,\,1}=N\otimes\bC$. 
    Consequently, $\cT_{L,\,1}=N'\otimes_\bZ\bC$ for some sublattice $N'\subseteq N$ and therefore, $W = \cF_{W,\,1} = \cT_{L,\,1} = N'\otimes\bC$. 
\end{proof}

\section{Toroidal foliations and extended complexes}\label{sec:toroidal_foliation}

Let us start by recalling some definitions.
\begin{definition}
Let X be a normal variety.
\begin{enumerate}
    \item A \emph{foliated pair} $(\cF, D)$ on $X$ consists of a foliation $\cF$ on $X$ and an $\bR$-divisor $D$ such that $K_\cF+D$ is $\mb R$-Cartier. 
    Note that $D$ is not required to be effective although we are mainly interested in the case when $D\geq 0$. 
    \item Let $(\cF,D)$ be a foliated pair on a normal variety $X$ and $\pi\colon \widetilde{X}\to X$ be a birational morphism. We can write
    $K_{\pi^{-1}\cF}+\pi_*^{-1}D=\pi^*(K_\cF+D)+\sum_{E} a(E,\cF,D)E$ 
    where the sum is over all $\pi$-exceptional prime divisors and $a(E,\cF,D)$ is called the \emph{discrepancy} of $(\cF,D)$ with respect to $E$. 
    \item Let $(\cF,D)$ be a foliated pair on a normal variety $X$. 
    We say that 
    \[(\cF,D) \textnormal{ is } \left\{\begin{array}{l}
        \textnormal{terminal} \\
        \textnormal{canonical} \\
        \textnormal{log terminal} \\
        \textnormal{log canonical} \\
        \ve\textnormal{-log canonical} 
    \end{array}\right.
    \textnormal{if } a(E,\cF,D) 
    \left\{\begin{array}{l}
        >0 \\
        \geq 0\\
        >-\iota(E)\\
        \geq -\iota(E)\\
        \geq -\iota(E)+\ve
    \end{array}\right.
    \]
    for any birational morphism $\pi \colon \widetilde X \to X$ and for any prime $\pi$-exceptional divisor $E$ on $\widetilde X$. 
    Here $\ve$ is a nonnegative real number and recall that $\iota(E)=0$ if $E$ is foliation invariant and $\iota(E)=1$ otherwise. 

    Let $P\in X$ be a point of $X$ which is not necessarily closed. 
    We say the foliated pair $(\cF,D)$ is terminal (resp. canonical, log terminal, log canonical, $\ve$-log canonical) \emph{at} $P$ if the requirement on discrepancy is satisfied for any exceptional divisor $E$ whose center in $X$ is the Zariski closure of $P$.
    
    Let $Z$ be an irreducible subvariety of $X$. 
    We say that the foliated pair $(\cF,D)$ is terminal (resp. canonical, log terminal, log canonical, $\ve$-log canonical) at the \emph{generic} point of $Z$ if it is such at $\eta_Z$, the generic point of $Z$. 
    And we say that the foliated pair $(\cF,D)$ is terminal (resp. canonical, log terminal, log canonical, $\ve$-log canonical) at the \emph{general} point of $Z$ if it is such at the general \emph{closed} point of $Z$. 

    We say $\cF$ is terminal (resp. canonical, log terminal, log canonical, $\ve$-log canonical) if the foliated pair $(\cF,0)$ is such. 

    \item Let $(\cF,D)$ be a foliated pair on a normal variety $X$. 
    We say $W\subseteq X$ is a \emph{log canonical center} (in short, \emph{lc center}) if $(\cF,D)$ is log canonical at the generic point of $W$ and there is some divisor $E$ of discrepancy $-\iota(E)$ on some model of $X$ dominating $W$. 
\end{enumerate}
\end{definition}

\subsection{Toric and toroidal foliated pairs}
In this subsection, we introduce toric and toroidal foliated pairs. 

\begin{notation}\label{notation:FanGeneratedByACone}
    For any rational, strongly convex, polyhedral cone $\sigma\subseteq N\otimes\bR$, we write $\Sigma_{\sigma,N}:=\{\tau\mid\tau\preceq\sigma\}$. We often omit $N$ from the notation when it is clear from context.
\end{notation}

\begin{definition}[Toric foliated pairs]\label{tor_fol_pair_defn}
Let $\Sigma$ be a fan in $N\otimes\bR$.
A \emph{toric foliated pair} $(\cF_W,D)$ on the toric variety $X_\Sigma$ consists of a
toric foliation $\cF_W$ on $X_\Sigma$ where $W\subseteq N\otimes\bC$ is a complex vector subspace and a torus invariant $\bR$-divisor $D$ on $X_\Sigma$ such that $K_{\cF_W}+D$ is $\bR$-Cartier. 
\end{definition}

\begin{definition}[{Toroidal embedding, \cite[Chapter 2, Definition 1]{KKMS1973toroidal}}]
    Let $X$ be a normal irreducible variety and $U\subseteq X$ be a Zariski open subset. 
    \begin{enumerate}
        \item The embedding $U\hookrightarrow X$ is called \emph{toroidal} if for any closed point $x\in X$, there exist an affine toric variety $U_{\sigma,N}$ and an isomorphism of complete local algebras $\psi_x\colon \widehat{\cO}_{X,\,x} \xrightarrow{\sim} \widehat{\cO}_{U_{\sigma,N},\,\gamma_\sigma}$ such that the ideal of $X\setminus U$ maps isomorphically to the ideal of $U_{\sigma,N}\setminus T_N$, where $T_N$ is the torus and $\gamma_\sigma$ is the distinguished point of $U_{\sigma,N}$. We call $U_{\sigma,N}$ a \emph{local model} of the toroidal embedding $U\hookrightarrow X$ at $x$.
        \item The divisor $\Xi=X\setminus U$ is called the \emph{associated reduced divisor} of the toroidal embedding $U\hookrightarrow X$. 
        \item A toroidal embedding $U\hookrightarrow X$ is called \emph{strict} if every irreducible component of $X\setminus U$ is normal. 
        \item A pair $(X,D)$ is called toroidal (resp. strict toroidal) if there is a toroidal embedding (resp. strict toroidal embedding) $U\hookrightarrow X$ such that $\on{Supp}(D)\subseteq X\setminus U$ and $K_X+D$ is $\bR$-Cartier.    
    \end{enumerate}
\end{definition}

\begin{definition}[Toroidal foliated pairs]\label{defn_toroidal}\hspace{1em}
    \begin{enumerate}
         \item A foliation $\cF$ on a normal variety $X$ is \emph{toroidal} if it is formally locally toric adapted to $\Xi$ for some toroidal embedding $(X\setminus\Xi)\hookrightarrow X$. 
        That is, $\Xi$ is a reduced divisor on $X$ such that for any closed point $x\in X$, there exists a couple $(U_{\sigma,N},W_N)$, called a \emph{local model} of $\mathcal F$ at $x$, which satisfies the following:
        \begin{itemize}
            \item $U_{\sigma,N}$ is the affine toric variety of the rational, strongly convex, polyhedral cone $\sigma\subseteq N\otimes\bR$, where $N$ is a lattice of finite rank.
            \item $W_N\subseteq N\otimes\bC$ is a complex vector subspace. 
            \item There exists an isomorphism of complete local algebras $\psi_x\colon \widehat{\cO}_{X,\,x}\cong\widehat{\cO}_{U_{\sigma,N},\,\gamma_\sigma}$, where $\gamma_\sigma$ is the distinguished point of $U_\sigma$, that maps the ideal of $\Xi$ to the ideal of $U_{\sigma,N}\setminus T_N$. Moreover,  the induced map between the modules of Kähler differentials over $\bC$ maps $\cF^*$ to $\cF_{W_N}^*$.
        \end{itemize} 
    The divisor $\Xi$ is called the \emph{associated reduced divisor} for the toroidal foliation $\cF$. 
    \item A toroidal foliation $\cF$ is called \emph{strict} if all irreducible components of the associated reduced divisor for $\cF$ is normal. 

    \item We say $(U_\sigma,W_\sigma)$ is a \emph{semi-local model} of $\mathcal F$ at $x$ if it satisfies the requirements for a local model with the following modifications:
    \begin{itemize}
        \item $\sigma\subseteq N\otimes\bR$ is a full-dimensional cone.
        \item $\psi_x$ maps the ideal of $\Xi$ to the ideal of some torus-invariant divisor in $U_{\sigma,N}$.
    \end{itemize} 

    \item Let $U\hookrightarrow X$ be a toroidal embedding and write $\Xi=X\setminus U=\bigcup_{i\in I}\Xi_i$. A \emph{stratum} is a component of the set $\bigcap_{i\in J}\Xi_i\setminus\bigcup_{i\notin J}\Xi_i$ for some $J\subseteq I$. Here we adopt the convention that the intersection over the empty index set is the whole space $X$ and the union over the empty index set is empty. Note that $X$ can be written as the disjoint union of strata. 
    We denote by $\on{Strata}(X)$ the set of all strata of the toroidal embedding $U\hookrightarrow X$. 

    \item
    We say a foliated pair $(\cF,D)$ on a normal variety $X$ is \emph{(strict) toroidal} if $\mc F$ is (strict) toroidal with associated reduced divisor $\Xi$, and $\on{Supp}(D)\subseteq \Xi$. 
    Let $(U_{\sigma,N},W_N)$ be a local model of $\mc F$ at $x\in X$. Then there exists a unique torus invariant divisor $D_{\sigma,N} = \sum_{\rho\in\sigma(1)} a_\rho D_\rho$ such that in a formal neighborhood of $x\in X$, $D$ is given by $D_{\sigma,N}$ via the isomorphism $\psi_x$. 
    The triple $(U_{\sigma,N},W_N,D_{\sigma,N})$ is called a \emph{local model} of $(\cF,D)$ at $x\in X$. 
    A semi-local model of $(\cF,D)$ at $x\in X$ is defined in the similar way. 
    \end{enumerate}
\end{definition}

\begin{remark}\label{toroidal_pair_rmk}
    Let $(\cF_W,D)$ be a toric foliated pair on a toric variety $X_\Sigma$, where $\Sigma$ is a fan in $N\otimes\bR$ and $W\subseteq N\otimes\bC$ is a complex vector subspace.
    Suppose $x\in O_\sigma$ for some $\sigma\in\Sigma$. Then there is a torus action $\phi_t$ which gives an automorphism on $X$, sending $x$ to $\gamma_{\sigma}$. Hence, $(U_{\sigma,N},W)$ is a local model of $\cF_W$ at $x$, and as a result, 
    $\cF_W$ is strict toroidal with the associated reduced divisor $\Xi=\sum_{\rho\in\Sigma(1)}D_\rho$. Any toric foliated pair $(\mc F_W,D)$ is strict toroidal.
\end{remark}

\subsection{Conical complexes}\label{sec:conical_cplx}
In this subsection, most definitions and properties of conical complexes are from \cite{KKMS1973toroidal} and \cite{AK00}. 

\begin{definition}
    A (rational polyhedral) conical complex $\Delta$ is a topological space $|\Delta|$ with a finite collection $S$ of closed subsets $\sigma$ of $|\Delta|$, each equipped with a lattice $N_{\sigma}$ of finite rank and a continuous function $\phi_\sigma\colon \sigma \to N_{\sigma}\otimes\bR$, satisfying the following:
    \begin{enumerate}
        \item $\sigma \xrightarrow{\phi_\sigma}\phi_\sigma(\sigma)$ is a homeomorphism. 
        \item $\phi_\sigma(\sigma)$ is a full-dimensional, rational, strongly convex, polyhedral cone in $N_\sigma\otimes\bR$. 
        \item For any closed subsets $\sigma_1\subseteq \sigma_2$ in $S$, there is a saturated $\bZ$-linear map $\iota_{\sigma_1\sigma_2}\colon N_{\sigma_1} \to N_{\sigma_2}$ such that the following diagram commutes:
        \[\begin{tikzcd}[column sep=5em]
            \sigma_1 \arrow[hook]{r}\arrow{d}[swap]{\phi_{\sigma_1}} & \sigma_2 \arrow{d}{\phi_{\sigma_2}} \\
            N_{\sigma_1}\otimes\bR \arrow{r}{\iota_{\sigma_1\sigma_2}\otimes\bR} & N_{\sigma_2}\otimes\bR.
        \end{tikzcd}\] 
        Here a $\mb Z$-linear map between lattices of finite rank is called saturated if the image is a saturated sublattice.
        \item For any closed subsets $\sigma_1\subseteq \sigma_2$ in $S$, $\phi_{\sigma_2}(\sigma_1)\preceq\phi_{\sigma_2}(\sigma_2)$.
        \item $|\Delta| = \bigsqcup_{\sigma\in S} \on{Relint}(\sigma)$. 
    \end{enumerate}
    Note that if $F$ is a face of $\phi_\sigma(\sigma)$ for some $\sigma\in S$, then $\sigma':=\phi_\sigma^{-1}(F)\in S$.
    If $\sigma_1, \sigma_2\in S$, then $\sigma_1\cap\sigma_2$ is in general a union of closed subsets in $S$, but not necessarily a single closed subset in $S$. 
\end{definition}

Hence, we may regard each $\sigma\in S$ as a closed subset of $|\Delta|$, or a full-dimensional, rational, strongly convex, polyhedral cone in a lattice $N_{\sigma}$. We will write $\sigma\in \Delta$ from now on, and omit the inclusion maps when no confusion seems to arise.

For any $i\in\bN$, we define $\Delta(i):=\{\sigma\in\Delta\mid\dim\sigma=i\}$. 

An \emph{extended complex} $(\Delta,W)$ is a conical complex $\Delta$ together with a map $W$ that assigns to each $\sigma\in\Delta$ a $\bR$-vector subspace $W(\sigma)$ in $N_\sigma\otimes\bR$ such that $W(\tau)=(W(\sigma)\cap N_\tau)\otimes\bR$ for all $\tau\preceq\sigma\in\Delta$. 

A \emph{map between extended complexes} $f\colon (\Delta,W)\to (\Delta',W')$ is a continuous map $\overline{f} \colon |\Delta|\to |\Delta'|$ such that for each $\sigma \in \Delta$, $\overline{f}(\sigma)$ is contained in some $\sigma' \in \Delta'$, and for any such $\sigma'$, $\overline{f}\vert_{\sigma}\colon \sigma\to\sigma'$ is induced by a $\bZ$-linear map $f_{\sigma,\sigma'}\colon N_\sigma\to N_{\sigma'}$. Moreover, we have $(f_{\sigma,\sigma'}\otimes\bR)(W(\sigma))=\bR(W'(\sigma')\cap (f_{\sigma,\sigma'}\otimes\bR)(\sigma))$. A map between the conical complexes $\Delta$ and $\Delta'$ is just a map $(\Delta,0)\to(\Delta',0)$ between extended complexes.

\begin{propdef}\label{propdef:star_subdiv_complex}
\begin{enumerate}
\item
    A conical complex $\Delta'$ is called a \emph{subdivision} of a conical complex $\Delta$ if $|\Delta'|=|\Delta|$ and the identity map on $|\Delta|$ defines a map $\Delta'\to \Delta$ of conical complexes. Moreover, for each $\sigma' \in \Delta'$, $f_{\sigma',\sigma}\colon N_{\sigma'}\to N_{\sigma}$ is saturated for any $\sigma \in \Delta$ that contains $\sigma'$. 
    As a result, every cone in $\Delta$ is a union of cones in $\Delta'$. 
     
    Let $(\Delta,W)$ be an extended complex and $\Delta'$ be a subdivision of $\Delta$. For each $\sigma'\in\Delta'$, define $W'(\sigma')=f_{\sigma',\sigma}^{-1}(W(\sigma) \cap f_{\sigma',\sigma}(N_{\sigma'}))\otimes\bR$, where  $\sigma\in\Delta$ is the minimal cone that contains $\sigma'$. Then one can show that  $(\Delta',W')$ is indeed an extended complex. 
\item
    Let $\Delta$ be a conical complex and $u\in \on{Relint}(\sigma)\cap N_\sigma$ be a primitive element for some $\sigma\in\Delta$. 
    Then we define $\Delta^*(u)$, the \emph{star subdivision} of $\Delta$ at $u$, as the set consisting of the following cones:
    \begin{itemize}
        \item $\tau$, where $u\notin\tau\in\Delta$. 
        \item $\tau':=\on{Cone}(\tau,u)\subseteq N_{\widetilde{\tau}}\otimes\bR$, where $\tau$, $\widetilde{\tau}\in \Delta$, $u\notin \tau$, and $\sigma$, $\tau\preceq \widetilde{\tau}$. We set $N_{\tau'}=N_{\widetilde{\tau}}\cap\bR\tau'$. 
    \end{itemize}
    One can show that $\tau'\subseteq N_{\tau'}\otimes\bR$ is independent of the choice of $\widetilde{\tau}$, and $\Delta^*(u)$ has a conical complex structure induced from $\Delta$.
\end{enumerate}
\end{propdef}
The proof is straightforward and we omit it here.

\subsection{Extended complexes associated with strict toroidal foliations}
Let $X\setminus\Xi\hookrightarrow X$ be a strict toroidal embedding. 
For any stratum $Y$, we define 
    \begin{enumerate}
        \item $\on{Star}(Y) = \bigcup_{Z\in\on{Strata}(X),\,Y\subseteq\overline{Z}}Z$. 
        \item $M^Y$ is the group of Cartier divisors in $\on{Star}(Y)$ supported in $\on{Star}(Y)\setminus U_X$. 
        \item $N^Y:=\on{Hom}(M^Y,\bZ)$. 
        \item $M^Y_+\subseteq M^Y$ is the subgroup containing the effective Cartier divisors. 
        \item $\sigma^Y\subseteq N^Y\otimes\bR$ is the dual of $M^Y_+\otimes\bR$. 
    \end{enumerate}
    \begin{corollary}[{\cite[Chapter2, Corollary 1 and page 71]{KKMS1973toroidal}}]\label{cor:stratum_vs_local_model} 
        If $U_{\sigma,N}$ is a local model of $X$ at $x\in Y$, then there are canonical isomorphisms
        \begin{enumerate}
            \item $M^Y\cong M/(\sigma^\perp\cap M)$, where $M$ is the dual lattice of $N$, 
            \item $N^Y \cong N \cap\on{Span}(\sigma)$, and 
            \item $\sigma^Y\cong\sigma$. 
        \end{enumerate}
        Moreover, $\{\sigma^Y\mid Y\mbox{ is a stratum of } X\}$ forms a conical complex $\Delta_X$, which will be called the \emph{associated conical complex} of the strict toroidal embedding $X\setminus\Xi\hookrightarrow X$. 
    \end{corollary}

\begin{notation}\label{notn:toroidal_orbit}
    For any $\sigma\in\Delta_X$, we denote by $O(\sigma)$ the stratum of $X$ such that $\sigma^{O(\sigma)}=\sigma$ and by $V(\sigma)$ the closure of $O(\sigma)$. 
\end{notation}

A morphism $f\colon X\to Y$ between toroidal embeddings is called \emph{toroidal} if for any $x\in X$, there exist a local model $U_{\sigma,N}$ of $X$ at $x$, a local model $U_{\sigma',N'}$ of $Y$ at $y=f(x)$, and a toric morphism $\phi\colon U_{\sigma,N}\to U_{\sigma',N'}$ such that $\phi(\gamma_\sigma)=\gamma_{\sigma'}$ and $\widehat f\colon  \widehat\cO_{Y,y}\to\widehat\cO_{X,x} $ comes from $\widehat\phi\colon \widehat\cO_{U_{\sigma',N'},\gamma_{\sigma'}}\to \widehat\cO_{U_{\sigma,N},\gamma_\sigma}$. One can verify that $f$ induces a map between conical complexes $\Delta_X\to\Delta_{Y}$. The following theorem says that we can do the opposite in some sense.

\begin{theorem}[{\cite[section 1.4]{AK00}, see also \cite[Theorem 6]{KKMS1973toroidal}}]\label{thm:toroidal_subdivision}
    Let $X\setminus\Xi\hookrightarrow X$ be a strict toroidal embedding with the associated complex $\Delta_X$, and let $\Delta'\to \Delta_X$ be a subdivision. Then there exists a proper birational toroidal morphism $f\colon X'\to X$ such that the associated conical complex of $X'$ is $\Delta'$, and the induced map between conical complexes is $\Delta'\to \Delta_X$. 
\end{theorem}

In fact, if $Y$ is a stratum in $X$ and $\sigma'\subseteq\sigma^Y$ is a cone in the subdivision, then we define 
\[U_{\sigma'} = \on{Spec}_{\on{Star}(Y)}\sum_{D\in\sigma'^\vee\cap M^Y}\cO_{\on{Star}(Y)}(-D).\]
Then $X'$ is formed by gluing together the open subsets $U_{\sigma'}$.

\begin{propdef}\label{propdef:exist_extended_complex}
    A strict toroidal foliation $\cF$ on $X$ uniquely determines an extended complex $(\Delta_X,W)$, which will be called the \emph{associated extended complex} of $\cF$. Note that in general $\on{rank}\cF$ can be larger than $\dim_{\bR} W(\sigma)$ for any $\sigma\in\Delta_X$.
\end{propdef}
\begin{proof}
For any $\tau\in \Delta_X$, choose $\sigma\in\Delta_X$ that contains $\tau$, and choose a local model $(U_{\sigma,N},W_N)$ of $\cF$ at some point $x\in O(\sigma)$. We may identify $\tau$ with a face $\tau\preceq\sigma$ together with the integral structures. Define $W(\tau)=\bR(W_N \cap N \cap \tau)\subseteq N_\tau\otimes\bR$.

We claim that $W(\tau)$ is independent of the choices of $\sigma$, $x\in O(\sigma)$, and the local model $(U_{\sigma,N},W_N)$. Let $u\in N \cap \tau$ be a primitive vector and write $\rho=\bR_{\geq 0}u$. Consider the toroidal morphism $f\colon X'\to X$ given by $\Delta_X^*(u)\to \Delta_X$. Then $V(\rho)\subseteq X'$ is invariant under $f^{-1}\cF$ if and only if $D_{\rho}\subseteq X_{\Sigma_{\sigma,N}^*(u)}$ is invariant under $\cF_{W_N}$, where $\Sigma_{\sigma,N}^*(u)$ is the star subdivision of $\Sigma_{\sigma,N}$ at $u$. By Corollary~\ref{basics}, it is equivalent to $\rho\nsubseteq W_N$, which is the same as $u\notin W(\tau)$.

One can check directly that $(\Delta_X,W)$ is an extended complex.
\end{proof}
In the proof above, we have also shown the following:

\begin{lemma}\label{lem:complex_div_inv}
    Let $\cF$ be a strict toroidal foliation on $X$ with associated extended complex $(\Delta_X,W)$. 
    For any ray $\rho\in\Delta_X(1)$, the divisor $V(\rho)$ is $\cF$-invariant if and only if $W(\rho)=\{0\}$. 
\end{lemma}
\begin{proof}
    It follows from the proof of Proposition-Definition~\ref{propdef:exist_extended_complex} and Corollary~\ref{basics}. 
\end{proof}

\begin{propdef}\label{lem:restr_blowup_complex}
    Let $\cF$ be a strict toroidal foliation on $X$ with associated extended complex $(\Delta_X,W)$. Then for every $\tau\in\Delta_X$, there is a strict toroidal foliation $\overline{\cF}$ on $V(\tau)$, whose associated extended complex $(\Delta_{V(\tau)},\overline{W}^\tau)$ is defined as follows:
    \begin{itemize}
        \item $\Delta_{V(\tau)}=\on{Star}(\tau):=\{\overline{\sigma}\mid\tau\preceq\sigma\in\Delta\}$, where $\overline{\sigma}$ is the image of $\sigma$ via $N_\sigma\otimes\bR\to N_\sigma\otimes\bR/(N_\tau \otimes\bR)=N_{\overline{\sigma}}\otimes\bR$ and $N_{\overline{\sigma}}=N_\sigma/N_\tau$.
        \item $\overline{W}^\tau(\sigma) = (W(\sigma)+\bR N_\tau)/\bR N_\tau$ for any $\tau\preceq\sigma\in\Delta$.
    \end{itemize}
    Indeed, $\overline{\cF}$ is obtained by successive restrictions: Choose $\rho_1,\ldots,\rho_k\in\Delta_X(1)$ such that $V(\tau)=V(\rho_1)\cap\cdots\cap V(\rho_k)$. Let $\cF_0=\cF$, and define $\cF_i$ to be the foliation on $V(\rho_1)\cap\cdots\cap V(\rho_i)$ such that $\cF_i=\cF_{i-1}\vert_{V(\rho_1)\cap\cdots\cap V(\rho_i)}$. Then $\overline{\cF}=\cF_k$. We will denote this foliation by $\cF\vert_{V(\tau)}$, and denote $(\cF\vert_{V(\tau)})\vert_{O(\tau)}$ by $\cF\vert_{O(\tau)}$.
\end{propdef}
\begin{proof}
    This is straightforward. 
\end{proof}

\subsection{\texorpdfstring{Non-dicritical singularities and condition $(\dagger)$}{Non-dicritical singularities and condition (dagger)}}
In this subsection, we first introduce the definition of non-dicritical singularities, which agrees with and generalizes the one in the literature. 
\begin{definition}\label{nddef}
    A foliation $\cF$ of corank $c$ on a normal variety $X$ is called \emph{dicritical} if there exists a prime divisor $E$ over $X$ which is not foliation invariant and the center $c_X(E)$ in $X$ has dimension at most $c-1$. 
    
    $\cF$ is \emph{non-dicritical} if it is not dicritical. Equivalently, $\cF$ is non-dicritical if for any prime divisor $E$ over $X$ with $\dim c_X(E)\leq c-1$, $E$ is foliation invariant. 
\end{definition}

\begin{remark}
    If $c=1$ and $\dim X = 3$, this is in agreement with \cite[Definition 2.10]{CS}. 
    If $c=2$ and $\dim X = 3$, this corresponds to the scenario described in the paragraph before \cite[Lemma 2.6]{cascini2025mmp}. 
    If $c=0$, then $\cF=\cT_X$ is non-dicritical since the assumption is vacuous. 
    If $c=\dim X$, then $\cF=0$ represents the foliation by points. 
    In this case, all subvarieties are invariant, and hence $\cF=0$ is also non-dicritical. 
\end{remark}

\begin{definition}\label{defn_dagger}
    Let $N$ be a lattice, $\sigma$ be a rational, strongly convex, polyhedral cone in $N\otimes\bR$, and $W\subseteq N\otimes\bK$ be a $\bK$-vector subspace where $\bK=\bQ$, $\bR$, or $\bC$.
    We say $(\sigma,W)$ is \emph{non-dicritical} if
    \begin{equation}\tag{$\dagger$}
    \on{Relint}(\sigma)\cap W\cap N\neq\emptyset\textrm{ if and only if } \sigma\subseteq W. 
    \end{equation}
\end{definition}

The following is the motivating example of the definition:

\begin{example}[Surface foliation]\label{eg_dicritical_1}
Let $N=\bZ e_1\oplus \bZ e_2$ and $\sigma=\on{Cone}(e_1, e_2)$. 
We consider the toric foliation $\cF$ on $U_\sigma\simeq\bA^2$ given by $\frac{\diff x}{x}-\lambda\frac{\diff y}{y}$ where $\lambda\in\bC^*$ and $\{x,y\}$ is a dual basis of $\{e_1,e_2\}$. 
Then $\cF=\cF_W$ where $W=\bC(\lambda e_1+e_2)$. 
If $\lambda\notin \bQ_{>0}$, then $\cF_W$ has a reduced singularity at the origin in the sense of \cite[Definition 1.1]{brunella2015birational}, which is known to be non-dicritical. 
If $\lambda\in \bQ_{>0}$, then $W$ is generated by a rational ray in $\on{Relint}(\sigma)$ and there is an exceptional divisor over the origin that is not foliation invariant by Corollary~\ref{basics}. 
As a result, $\cF_W$ is dicritical by \cite[Proposition 1.11]{brunella2015birational}. 
On the other hand, one can immediately check that $\lambda\notin \bQ_{>0}$ if and only if $(\dagger)$ is satisfied.
\end{example}

\begin{example}\label{eg_dicritical_2}
    Let $N=\bZ e_1\oplus \bZ e_2\oplus \bZ e_3$, $\sigma=\on{Cone}(e_1, e_2,e_3)$, and $W=\{(b_1,b_2,b_3)\mid b_1-b_2+ib_3=0\}$. 
Then $U_\sigma\cong\bA^3$ and $\cF_W$ is the toric foliation on $\bA^3$ given by $\omega := x_1x_2x_3(\frac{\diff x_1}{x_1}-\frac{\diff x_2}{x_2}+i\frac{\diff x_3}{x_3})$. 
We consider the chart for the blow-up $p$ at $(0,0,1)$ given by 
\[x_1=x'_1,\, x_2=x'_1x'_2,\, x_3-1=x'_1x'_3.\]
Then we have 
\begin{align*}
    p^*\omega &= x'^2_1x'_2(1+x'_1x'_3)\Big(\frac{-\diff x'_2}{x'_2}+i\frac{\diff(x'_1x'_3)}{1+x'_1x'_3}\Big) \\
    &= -x'^2_1(1+x'_1x'_3)\diff x'_2 + ix'^2_1x'_2\diff(x'_1x'_3) \\
    &= x'^2_1\Big(ix'_2x'_3\diff x'_1 - (1+x'_1x'_3)\diff x'_2 + ix'_1x'_2\diff x'_3\Big).
\end{align*}
So the pullback foliation is given by $\widetilde{\omega}:=\frac{1}{x'^2_1}p^*\omega$. 
Note that 
\[\partial = (1+x'_1x'_3)\frac{\partial}{\partial x'_1}+ix'_2x'_3\frac{\partial}{\partial x'_2}\in\ker(\widetilde{\omega}).\] 
As the exceptional divisor $E$ for $p$ is defined by $x'_1=0$ and $\partial x'_1 = 1+x'_1x'_3$, which is not in the ideal generated by $x'_1$,  we have $\partial \cI_E \nsubseteq \cI_E$ and hence $E$ is not foliation invariant with center $(0,0,1)\in U_\sigma$. 
Therefore, the foliation $\cF_W$ is dicritical. 

Note that for any star subdivision for a ray $\rho$ whose primitive generator is in the interior of $\sigma$, the exceptional divisor is foliation invariant as $\rho\nsubseteq W$ by Corollary~\ref{basics}. In other words, it is not enough to examine the non-dicriticality of $\cF_W$ by looking at the exceptional divisors that can be extracted by toric morphisms.
\end{example}

We can also consider the condition $(\dagger)$ for an extended complex.
\begin{definition}\label{def:dagger_extended_complex}
    We say an extended complex $(\Delta,W)$ satisfies the condition $(\dagger)$ if $(\sigma,W(\sigma))$ is non-dicritical for all $\sigma\in\Delta$.
\end{definition}

\begin{lemma}\label{lem_property_dagger}
\begin{enumerate}
    \item  Let $N$ be a lattice, $\sigma$ be a rational, strongly convex, polyhedral cone in $N\otimes\bR$, and $W\subseteq N\otimes\bK$ be a proper $\bK$-vector subspace where $\bK=\bR$ or $\bC$. Then $(\Sigma_\sigma,W)$ satisfies the condition $(\dagger)$ if and only if $\mb R(W\cap N)\cap \sigma$ is a face of $\sigma$.
    \item Let $(\Delta,W)$ be an extended complex satisfying the condition $(\dagger)$. 
    If $\Delta'$ is a subdivision of $\Delta$, then the induced extended complex $(\Delta',W')$ also satisfies the condition $(\dagger)$. 
\end{enumerate}
\end{lemma}
\begin{proof}
    For (1), suppose $(\Sigma_\sigma,W)$ satisfies the condition $(\dagger)$. 
    Then for each cone $\tau\preceq\sigma$, we have either $\on{Relint}(\tau)\cap W\cap N=\emptyset$ or $\on{Relint}(\tau)\cap W\cap N\neq\emptyset$, and by the condition $(\dagger)$, the latter implies that $\tau\subseteq W$. Thus, $\bR(W\cap N)\cap\sigma$ is a face of $\sigma$.

    Suppose $\bR(W\cap N)\cap\sigma$ is a face of $\sigma$. Then for any $\tau\preceq \sigma$ with $\on{Relint}(\tau)\cap W\cap N\neq\emptyset$, we have $\tau\preceq \bR(W\cap N)\cap\sigma$. Hence $\tau\subseteq \bR(W\cap N)\subseteq W$. That is, $(\Sigma_\sigma,W)$ satisfies the condition $(\dagger)$.

    For (2), suppose $\sigma'\in\Delta'$ with $\on{Relint}(\sigma')\cap W'(\sigma')\cap N_{\sigma'}\neq\emptyset$. 
    Let $\sigma\in\Delta$ be the minimal cone such that $\sigma'\subseteq\sigma$. 
    Then we have $\on{Relint}(\sigma')\subseteq\on{Relint}(\sigma)$. 
    Thus, $\on{Relint}(\sigma)\cap W(\sigma)\cap N_{\sigma}$ is not an empty set as it contains $\on{Relint}(\sigma')\cap W'(\sigma')\cap N_{\sigma'}$. 
    Since $(\Delta,W)$ satisfies the condition $(\dagger)$, we have $\sigma\subseteq W(\sigma)$ and thus, $\sigma'\subseteq\sigma\subseteq W(\sigma)$. 
    Hence, $\sigma'\subseteq W(\sigma)\cap (N_{\sigma'}\otimes\bR) = W'(\sigma')$. 
    Therefore, $(\Delta',W')$ also satisfies the condition $(\dagger)$. 
\end{proof}

Let $\cF$ be a strict toroidal foliation on $X$, and let $x\in O(\tau)$ for some $\tau\in\Delta_X$. By an argument similar to the proof of Lemma~\ref{cor:stratum_vs_local_model} (see the references cited there), if $(U_{\sigma,N},W_N)$ is a semi-local model of $\cF$ at $x$, we can identify $\tau$ with a face $\tau\preceq\sigma$ together with the integral structures. In this case, we still have $W(\tau)=\bR(W_N\cap N \cap \tau)$. Lemma~\ref{lem:exist_semi_local_model} states that a semi-local model exists and that it can be chosen to be as simple as possible. In particular, it allows us to show that if $Z$ is a subvariety of $X$, then near a general point $z\in Z$, the blow-up along $Z$ gives a toroidal foliation (see Lemma~\ref{lem:blowup_toroidal_model}).  To streamline the proof, we first point out the following fact:

\begin{lemma}\label{lem:local_isomorphism_of_affine_toric_variety}
    Let $N$ be a lattice of rank $n$, and let $U_{\sigma,N}$ be an affine toric variety. 
    If $\{m_1,\ldots,m_n\}$ is a basis of $M:=\on{Hom}(N,\bZ)$, then for any $c_1,\ldots,c_n\in\bC^*$, $x_i:=\chi^{m_i}\mapsto c_i x_i$ defines an automorphism on $T_N$ that extends to an automorphism on $U_{\sigma,N}$. 
\end{lemma}
\begin{proof}
    It is clear since it defines a torus action. More precisely, it defines an automorphism on $\bC[\sigma^\vee\cap M]$ by sending an element $x$ of degree $a_1m_1+\cdots+a_n m_n$ to $c_1^{a_1}\cdots c_n^{a_n} x$.
\end{proof}

\begin{lemma}\label{lem:minimal_intersection}
    Let $\mathbf 0\in Z\subseteq U\subseteq\bA^n=\{(x_1,\ldots,x_n)\}$ be a analytic subvariety, where $U$ is an analytic open subset, and $Z$ is smooth at $\mathbf 0$. Let $\cF$ be the foliation on $U$ defined by $\diff x_{r+1}\wedge\cdots\wedge\diff x_n$, and suppose
    $\dim_\bC \cT_{Z}(z)\cap\cF(z), z\in Z$, locally attains its minimum at $\mathbf 0$. Then there exist analytic functions $y_1,\ldots,y_n$ defined locally near $\mathbf 0$, such that $\bC[\![x_1,\ldots,x_n]\!]=\bC[\![y_1,\ldots,y_n]\!]$, $\cF$ is defined by $\diff y_{r+1}\wedge\cdots\wedge\diff y_n$, and $Z=\{y_i=0\mid i\in I\}$ in an analytic neighborhood of $\mathbf 0$ for some $I\subseteq\{1,\ldots,n\}$.
\end{lemma}
\begin{proof}
    By possibly shrinking $U$, we may assume that $Z$ is defined by $f_1(x_1,\ldots,x_n)=\cdots=f_s(x_1,\ldots,x_n)=0$, where $s=n-\dim Z$. Consider the matrix $A(z)$ defined by $A_{ij}(z)=\frac{\partial f_i}{\partial x_j}(z)$. Then we may assume that 
    \[
        A(\mathbf 0) =
        \begin{bmatrix}
        B & * & *\\
        \mathbf{0} & \mathbf{0} & C
        \end{bmatrix},
    \]
    where $B$ is a $t \times t$ diagonal matrix of rank $t$, the mid-bottom submatrix is a zero matrix of size $(s-t) \times (r-t)$, and $C$ is a $(s-t)\times(n-r)$ matrix in row-echelon form of full rank $s-t$ with pivots in the first $s-t$ columns. 
    By basic analysis, we may assume that for each $t+1\leq i\leq s$, $f_i$ is analytic near $\mathbf 0$ and does not involve $x_j$ for $1\leq j\leq t$ or $r+1\leq j \leq i-t+r-1$.

    Note that for $z\in Z$ near $\mathbf 0$, $\dim_\bC\cF(z)\cap T_Z(z)=n-r-\lambda$, where $\lambda$ is the rank of the submatrix of $A(z)$ consisting of the first $r$ columns. If we denote the mid-bottom submatrix of $A(z)$ by $R(z)$, we must have $R(z)=0$ for all $z\in Z$ near $\mathbf 0$ from the assumption. Hence, for each $t+1\leq i\leq s$, we can modify $f_i$ so that $f_i$ does not involve $x_j$ for $t+1\leq j\leq r$.  

    Now let 
    \[
        y_i=\left\{\begin{array}{ll}f_i &, 1\leq i \leq t,\\
            x_i &,t+1\leq i \leq r\\
            f_{i-r+t} &,r+1\leq i \leq s-t+r\\
            x_i &, s-t+r+1\leq i\leq n
            \end{array}\right..
    \]
    Then $\bC[\![x_1,\ldots,x_n]\!]=\bC[\![y_1,\ldots,y_n]\!]$, 
    $\diff y_{r+1}\wedge\cdots\wedge\diff y_n=u(x_1,\ldots,x_n)\diff x_{r+1}\wedge\cdots\wedge\diff x_n$ where $u(\mathbf 0)\neq 0$, and $Z$ is defined by $\{y_i=0\mid 1\leq i\leq t\text{ or }r+1\leq i \leq s-t+r\}$.  
\end{proof}

\begin{lemma}\label{lem:exist_semi_local_model}
    Let $\cF$ be a strict toroidal foliation on $X$ with extended complex $(\Delta_X,W)$. Suppose $x\in O(\tau)$ for some $\tau\in\Delta_X$.
    \begin{enumerate}
        \item There exists a semi-local model $(U_{\sigma,N},W_N)$ of $\cF$ at $x$ such that $\bR(W_{N}\cap N\cap \sigma)=W(\tau)+\bR\tilde\tau$, where $\tilde\tau\preceq\sigma$ is the face generated by all the rays $\rho\preceq\sigma$ contained in $W_N$ and $\rho\nsubseteq N_\tau\otimes\bR$. In particular, if $(\Sigma_{\tau,N_\tau}, W(\tau))$ satisfies the condition $(\dagger)$, then so does $(\Sigma_{\sigma,N},W_N)$.

        \item If $Z\subseteq V(\tau)$ is a subvariety that intersects $O(\tau)$, then we can choose a point $x\in Z\cap O(\tau)$ and a semi-local model $(U_{\sigma,N},W_N)$ of $\cF$ at $x$ such that, in addition to the statement in (1), $Z$ formally locally corresponds to $V_{\tau_Z}$ for some $\tau\preceq\tau_Z\preceq \sigma$.
    \end{enumerate}
\end{lemma}

\begin{proof}
    We first prove (1). By considering a local model $(U_{\tau,\widetilde N},W_{\widetilde N})$ of $\cF$ at $x$, we may replace $X$ by $U_{\tau,\widetilde N}$. Choosing a sublattice $N_0\subseteq \widetilde N$ such that $\widetilde N=N_\tau\oplus N_0$, we get $U_{\tau,\widetilde N}\simeq U_{\tau,N_\tau}\times T_{N_0}$, and we may assume that $x=(\gamma_\tau,\mathbf 1)\in U_{\tau,N_\tau}\times T_{N_0}$. Choose bases $\{m_1,\ldots,m_k\}$ and $\{m_1',\ldots,m_\ell'\}$ of $\on{Hom} (N_\tau,\bZ)$ and $\on{Hom} (N_0,\bZ)$, respectively, and write $x_i=\chi^{m_i},y_j=\chi^{m'_j}$. By expressing $W_{\widetilde N}$ as an intersection of hyperplanes of $\widetilde N\otimes\bC$, we may assume that $\cF_{W_{\widetilde N}}$ is defined by $\omega=\omega_1\wedge\cdots\wedge\omega_c$, where $c=\dim X-\on{rank}\cF$ and $\omega_i=\sum_{j=1}^ka_{ij}\frac{\diff x_j}{x_j}+\sum_{j=1}^{\ell}b_{ij}\frac{\diff y_{j}}{y_{j}}$. Let $A$ be the matrix defined by $A_{ij}=a_{ij}$ if $1\leq j\leq k$, and $A_{ij}=b_{i,j-k}$ if $k+1\leq j\leq k+\ell$. By performing row operations, possibly re-indexing the variables, and modifying $\omega_i$ accordingly, we may assume that  
    \begin{equation}\tag{*}
        A =
        \begin{bmatrix}
        B & \textbf{0} & *\\
        \mathbf{0} & C & *
        \end{bmatrix},
    \end{equation}
    where
    \begin{itemize}
        \item  $B$ is an $ m \times k $ matrix in row-echelon form of full rank $m$ with pivots in the first $m$ columns, and
        \item  $C$ is a $(c-m) \times (c-m)$ diagonal square matrix of full rank.
    \end{itemize} 
    We allow $m$ to be $0$ or $c$, and choose all the pivots to be $1$.
    Here we use the fact $\on{rank}(A)=c$.
    
    Note that centered at $x$, $\cF_{W_{\widetilde N}}$ is defined by the wedge product of $\omega_i=\sum_{j=1}^ka_{ij}\frac{\diff x_j}{x_j}+\sum_{j=1}^{\ell}b_{ij}\frac{\diff z_{j}}{z_{j}+1}$, where $z_j=y_j-1$. Assume (*), and consider the following analytic functions:
    \begin{align*}
    &x_i'=x_i(z_1+1)^{b_{i,1}}\cdots(z_\ell+1)^{b_{i,\ell}},\ 1\leq i\leq m,\\
    &x_{i}'=x_{i},\ m+1\leq i\leq k,\\
    &z'_i=(z_1+1)^{b_{m+i,1}}\cdots (z_\ell+1)^{b_{m+i,\ell}}-1,\ 1\leq i\leq c-m-1,\\
    &z'_{i}=z_{i},\ c-m\leq i\leq \ell.
    \end{align*}
    Then $(x_1',\ldots,x_k',z_1',\ldots,z_\ell')$ defines a biholomorphic map from an analytic neighborhood of $x$ to an analytic neighborhood of $(\gamma_{\tau},0)\in U_{\tau,N_\tau}\times\bA^\ell$ by Lemma~\ref{lem:local_isomorphism_of_affine_toric_variety}. Write $\bA^\ell=U_{\sigma_1,N_1}$ through the coordinates $z'_i$, and let $N=N_\tau\oplus N_1$, $\sigma=\tau\times\sigma_1$. Then $\cF_{W_{\widetilde N}}$ maps to the toric foliation $\cF_{W_N}$ on $U_{\sigma,N}$ defined by $\omega'=\omega'_1\wedge\cdots\wedge\omega_c'$, where $\omega'_i=\sum_{j=i}^ka_{ij}\frac{\diff x_j'}{x_j'}$ for $1\leq i\leq m$, $\omega'_i=\diff z'_{i-m}$ for $m+1\leq i\leq c$. Denote by $\widetilde\tau$ the face of $\sigma$ generated by all the rays $\rho\preceq\{0\}\times\sigma_1$ contained in $W_N$.
    Then one can check that $W_N\cap N=W_N\cap N_\tau+W_N\cap N_1$, and $W_N\cap N_1=N\cap\bR\widetilde\tau$. Since $W(\tau) = \bR(W_N\cap N_\tau \cap \tau)$, the assertion $\bR(W_{N}\cap N\cap\sigma)=\bR(W_{N}\cap N\cap(\tau\times\widetilde\sigma))=W(\tau)+\bR\widetilde\tau$ follows.

    Suppose $(\Sigma_{\tau,N_\tau},W(\tau))$ satisfies the condition $(\dagger)$. 
    Then by Lemma~\ref{lem_property_dagger}(1), we have $\bR(W(\tau)\cap N_\tau) \cap \tau$ is a face of $\tau$. 
    Note that 
    \[\bR(W_N\cap N)\cap\sigma=\bR(W_N\cap N\cap\sigma)\cap\sigma=(W(\tau)+\bR\widetilde\tau)\cap(\tau\times\widetilde\sigma)=(W(\tau)\cap\tau)\times \widetilde\tau.\]
    Since $W(\tau)$ is generated by lattice points, we have $W(\tau)\cap \tau=\tau_0$ for some $\tau_0\preceq\tau$. Hence $\bR(W_N\cap N)\cap\sigma=\tau_0\times\widetilde\tau\preceq\sigma$,
    and therefore $(\Sigma_{\sigma,N},W_N)$ satisfies the condition $(\dagger)$ again by Lemma~\ref{lem_property_dagger}(1).

    Now we prove (2). 
    Write $\overline\cF=\cF\vert_{V(\tau)}$, and choose $x\in O(\tau)$ to be a point such that $\dim_\bC(\overline\cF(z)\cap \cT_Z(z))$ locally attains its minimum at $x$. Assume (*), and choose a semi-local $(U_{\sigma,N},W_N)$ of $\cF$ at $x$ as in the proof of Lemma~\ref{lem:exist_semi_local_model}(1). Then $(U_{\sigma_1,N_1},\bC\widetilde\tau)$ is a semi-local model of $\overline\cF$ at $x$. Choose $\widetilde Z \subseteq U_{\sigma_1,N_1}$ defined analytically locally near $\gamma_{\sigma_1}$, such that $\widetilde Z$ maps to $Z$ formally locally near $\gamma_{\sigma_1}$. By applying Lemma~\ref{lem:minimal_intersection}, we get the analytic functions $z''_1,\ldots,z''_\ell$ on $U_{\sigma_1,N_1}$ such that $\cF_{\bC \widetilde\tau}$ is defined by $\diff z''_1\wedge\cdots\wedge\diff z''_{c-m}$, and $\widetilde Z=\{z''_i=0\mid i\in I\}$ near $\gamma_{\sigma_1}$ for some $I\subseteq\{1,\ldots,c-m\}$. This gives a toric foliation $\cF_{\bC\widehat\tau}$ on $U_{\sigma_2,N_2}$ for some $\widehat\tau\preceq\sigma_2$. Replacing $N_1$ by $N_2$ and  $(U_{\sigma_1,N_1},\bC\widetilde\tau)$ by $(U_{\sigma_2,N_2},\bC\widehat\tau)$, we get the semi-local model as required.
\end{proof}

\begin{lemma}\label{lem:blowup_toroidal_model}
    Let $\cF$ be a toroidal foliation on a smooth toroidal embedding $X$ with associated extended complex $(\Delta_X,W)$. Let $Z\subseteq X$ be a subvariety, and let $\pi\colon \widetilde{X}:=\on{Bl}_ZX \to X$ be the blow-up along $Z$ with exceptional divisor $E$. 
    Then there is a Zariski open subset $U$ of $X$ such that $U\cap Z\neq\emptyset$, $\pi^{-1}(U)$ is toroidal, and $\widetilde{\cF}:=\pi^{-1}\cF\vert_{\pi^{-1}(U)}$ is a toroidal foliation with the associated extended complex $(\Delta_{\pi^{-1}(U)},\widetilde{W})$ which will be made precise in the proof. Moreover, if $(\Delta_X,W)$ satisfies the condition $(\dagger)$, then so does $(\Delta_{\pi^{-1}(U)},\widetilde{W})$.
\end{lemma}
\begin{proof}
Let $\tau\in\Delta_X$ be the unique cone such that $Z\cap O(\tau)$ is dense in $Z$. Let $\overline{\cF}=\cF\vert_{O(\tau)}$ and define $\phi\colon (Z\cap O(\tau))\setminus\on{Sing}(Z)\to\bZ_{\geq 0}$ by $\phi(z) = \dim_\bC(\overline{\cF}(z)\cap\cT_{Z}(z))$. Since $\phi$ is upper semi-continuous, $U_Z = \{z\in (Z\cap O(\tau))\setminus\on{Sing}(Z)\mid\phi(z) = \min\phi\}$ is a Zariski open subset of $Z$. Then we take $U = (V\setminus Z)\cup U_Z$ where $V=\bigcup_{\tau'\preceq\tau}O(\tau')$. Take a point $x\in U$, and consider a semi-local model $(U_{\sigma,N},W_N)$ of $\cF$ at $x$ as in Lemma~\ref{lem:exist_semi_local_model}(2). Then for any point $y\in\pi^{-1}(x)$, there is a commutative diagram
\[
\begin{tikzcd}
    \on{Spec}\bC(y)\arrow[r,"\alpha"]\arrow[d] & \on{Bl}_{V_{\tau_Z}}(U_{\sigma,N})\arrow[d]\\
    \on{Spec}\bC(x)\arrow[r] & U_{\sigma,N}
\end{tikzcd}.
\]
Thus after suitable torus action, $\on{Im}\alpha\in\on{Bl}_{V_{\tau_Z}}(U_{\sigma,N})$ gives a local model of $\pi^{-1}(U)$ at $y$. Hence $\pi^{-1}(U)$ is toroidal. One can check that $\Delta_{\pi^{-1}(U)}$ consists of all the faces of $\tau$ not containing $\rho_Z$ and $\on{Cone}(\rho_Z,\tau')\subseteq N\otimes \bR$, where $\rho_Z$ is the ray generated by the barycenter of $\tau_Z$  and $\tau'\preceq \tau$ ranges over all faces of $\tau$ not containing $\rho_Z$. Moreover, $\pi^{-1}\cF\vert_{\pi^{-1}(U)}$ is toroidal with the extended complex $(\Delta_{\pi^{-1}(U)},\widetilde{W})$, where $\widetilde{W}$ is defined by $\widetilde{W}(\tau'')=\bR(W_N\cap N \cap \tau'')$ for any $\tau''\in \Delta_{\pi^{-1}(U)}$. 

For the moreover part, if $(\Delta_X,W)$ satisfies the condition $(\dagger)$, then so does $(\Sigma_{\sigma,N},W_N)$ by Lemma~\ref{lem:exist_semi_local_model}(1). 
Let $W'$ be a map from $\Sigma_{\sigma,N}$ that assigns to each face $\tau\preceq\sigma$ the $\bR$-vector subspace $W'(\tau)=\bR(W_N\cap N\cap \tau)$ of $N\otimes\bR$. 
Then $(\Sigma_{\sigma,N},W')$ is an extended complex satisfying the condition $(\dagger)$. 
Hence by Lemma~\ref{lem_property_dagger}(2), $(\Sigma_\sigma^*(u_Z),W_N)$ satisfies the condition $(\dagger)$ where $\Sigma_\sigma^*(u_Z)$ is the star subdivision of $\Sigma_\sigma$ at the barycenter $u_Z$ of $\tau_Z$. Since each cone of $\Delta_{\pi^{-1}(U)}$ is in $\Sigma_\sigma^*(u_Z)$, we conclude that $(\Delta_{\pi^{-1}(U)},\widetilde{W})$ satisfies the condition $(\dagger)$.
\end{proof}

In preparation for the proof of Proposition~\ref{ND_toroidal}, we state and prove the following lemma which relates the invariance of a divisor to its image. 
\begin{lemma}\label{lem:contained_in_invariant_divisor}
    Let $\cF$ be a toroidal foliation on $X$, whose extended complex $(\Delta_X,W)$ satisfies the condition $(\dagger)$. Let $X'$ be a toroidal embedding with conical complex $\Delta_{X'}$, and let $\phi\colon X' \to X$ be a toroidal morphism. 
    For any $\rho'\in\Delta_{X'}(1)$, if $\phi(V(\rho'))\subseteq V(\rho)$ for some $\rho \in \Delta_X(1)$ where $V(\rho)$ is $\cF$-invariant, then $V(\rho')$ is $\phi^{-1}\cF$-invariant.
\end{lemma}
\begin{proof}
    Denote by $f\colon (\Delta_{X'},W')\to(\Delta_X,W)$ the corresponding map between the extended complexes, and let $\sigma\in \Delta_X$ be the minimal cone containing $f(\rho')$. Then $\rho\preceq \sigma$. Since $\rho\nsubseteq W(\sigma)$, we have $\sigma\nsubseteq W(\sigma)$, and thus by the condition $(\dagger)$, $f(\rho')\nsubseteq W(\sigma)$. Hence $W'(\rho')\subseteq f_{\rho',\sigma}^{-1}(\bR(W(\sigma)\cap (f_{\rho',\sigma}\otimes\bR)(\rho')))=\{0\}$ and $W'(\rho')=\{0\}$. 
    Therefore, $V(\rho')$ is $\phi^{-1}\cF$-invariant by Lemma~\ref{lem:complex_div_inv}.
\end{proof}
\begin{corollary}\label{cor:contained_in_invariant_divisor}
    \begin{enumerate}
        \item Assume as in Lemma~\ref{lem:contained_in_invariant_divisor}. If $\dim(\phi(V(\rho'))) \leq \dim X-\on{rank}\cF-1$, then $V(\rho')$ is foliation invariant.
        \item Let $\cF$ be a toroidal foliation on a smooth toroidal embedding $X$, whose extended complex $(\Delta_X,W)$ satisfies the condition $(\dagger)$.  If $Z$ is a subvariety contained in an $\cF$-invariant divisor $D=V(\rho)$, $\rho \in \Delta_X(1)$, then the exceptional divisor $E$ obtained by blowing up $Z$ is foliation invariant.
    \end{enumerate}
\end{corollary}
\begin{proof}
    For (1), note that $\dim(\phi(V(\rho')))=\dim X-\dim \sigma$ where $\sigma\in\Delta_X$ is the minimal cone containing $f(\rho')$. Thus $\dim\sigma\geq\on{rank}\cF + 1\geq\dim_\bR W(\sigma)+1$, and there exists a ray $\rho\in\sigma(1)$ not contained in $W(\sigma)$ by the condition $(\dagger)$ of $(\Delta_X,W)$. 
    Hence $W(\rho) = (W(\sigma)\cap N_\rho)\otimes\bR = \{0\}$ and $V(\rho)$ is foliation invariant by Lemma~\ref{lem:complex_div_inv}. 
    Therefore, $V(\rho')$ is foliation invariant by Lemma~\ref{lem:contained_in_invariant_divisor}.
    
    For (2), we denote the blow-up by $\pi\colon \on{Bl}_{Z}(X)\to X$, and let $\tau\in\Delta_X$ be the unique cone such that $O(\tau)\cap Z\neq\emptyset$. As in the proof of
    Lemma~\ref{lem:blowup_toroidal_model}, for a general point $z\in Z$, there is a semi-local model $(U_{\sigma,N},W_N)$ of $\cF$ at $z$, such that formally locally near $z$, $Z$ is given by $V_{\tau_Z}$ where $\tau\preceq\tau_Z\preceq \sigma$, and $\pi$ is given by $\psi\colon \on{Bl}_{V_{\tau_Z}}(U_{\sigma,N})\to U_{\sigma,N}$. Since $\rho\preceq \tau$ and $(\Sigma_\sigma,W_N)$ satisfies the condition $(\dagger)$, we may thus apply Lemma~\ref{lem:contained_in_invariant_divisor} to conclude that $\on{Exc}(\psi)$ is foliation invariant, and therefore, $\on{Exc}(\pi)$ is foliation invariant.
\end{proof}

\begin{proposition}\label{ND_toroidal}
    Let $\cF$ be a toroidal foliation of rank $r$ on a normal variety $X$ of dimension $n$.
    Then $\cF$ is non-dicritical if and only if its associated extended complex $(\Delta_X,W)$ satisfies the condition $(\dagger)$. 
    In particular, if $\cF_W$ is a toric foliation on a toric variety $X_\Sigma$, then $\cF_W$ is non-dicritical if and only if $(\Sigma,W)$ satisfies the condition $(\dagger)$. 
\end{proposition}
\begin{proof}
    (Only if part) 
    Suppose $(\Delta_X,W)$ does not satisfies the condition $(\dagger)$. We are going to introduce a foliation non-invariant divisor $E$ over $X$ whose center $c_X(E)$ has dimension $\leq n-r-1$, hence $\cF$ is dicritical. To this end, let $\tau\in \Delta_X$ be a cone such that $(\tau,W(\tau))$ is dicritical, and let $u_0\in N_\tau\cap W(\tau)\cap\on{Relint}(\tau)$ be a primitive element. Choose a subvariety $Z\subseteq V(\tau)$ that intersects $O(\tau)$ and is complementary to $\cF\vert_{V(\tau)}$ in the following sense: 
    \begin{itemize}
        \item $\on{rank}\cF\vert_{V(\tau)} + \dim Z = \dim V(\tau)$ and 
        \item $(\cF\vert_{V(\tau)})(z)\cap \cT_Z(z)=0$ at a general point $z\in Z$.
    \end{itemize}
    Choose a semi-local model $(U_{\sigma,N},W_N)$ of $\cF$ at $z$ as in Lemma~\ref{lem:exist_semi_local_model}(2). Then we have $\tau_Z=\on{Cone}(\tau,\widetilde\tau)$ where 
    \begin{itemize}
        \item $\tau_Z\preceq\sigma$ is the cone so that $Z$ formally locally corresponds to $V_{\tau_Z}$, and 
        \item $\widetilde{\tau}\preceq\sigma$ is the face generated by all the rays $\rho\preceq\sigma$ contained in $W_N$ but $\rho\nsubseteq N_\tau\otimes\bR$.
    \end{itemize}
    Now consider a smooth subdivision $\Delta'$ of $\Delta_X$ such that $\rho_0:=\bR_{\geq 0}u_0\in\Delta'$, and denote the corresponding toroidal morphism by $\pi\colon X'\to X$. 
    Recall that $\sigma = \tau\times\sigma_1$. (See the proof of Lemma~\ref{lem:exist_semi_local_model} for the construction of $\sigma_1$.) 
    Let $\Sigma_\sigma'$ be the subdivision of $\Sigma_{\sigma}$ consisting of $\on{Cone}(\tau',\sigma')$ where $\sigma'\preceq\sigma_1$ and $\tau'\in \Delta'$ with $\tau'\subseteq \tau$. 
    Then we have the following commutative diagram:
    \[
    \begin{tikzcd}
        \pi^{-1}(z)\arrow[r,"\alpha"]\arrow[d,"\pi"] &\widetilde\pi^{-1}(\gamma_\sigma) \arrow[d]\arrow[r,hook] &X_{\Sigma'_\sigma}\arrow[d,"\widetilde\pi"]\\
        \{z\} \arrow[r] & \{\gamma_\sigma\} \arrow[r,hook] & U_{\sigma,N}
    \end{tikzcd}
    \]
    where $\widetilde{\pi}$ is the base change of $\pi$ to $U_{\sigma,N}$. 
    We then choose $z'\in \pi^{-1}(z)$ such that $\alpha(z')=\gamma_{\tau_0}\in\widetilde\pi^{-1}(\gamma_\sigma)$, where $\tau_{0}=\on{Cone}(\rho_0,\sigma_1)$. Then formally locally near $z'$, $\pi^{-1}(Z)$ is given by $\tau_{\textnormal{n-inv}}=\on{Cone}(\rho_0,\widetilde\tau)\preceq \tau_0$ as $\tau_Z=\on{Cone}(\tau,\widetilde{\tau})$. As $\tau_{\textnormal{n-inv}}\subseteq W_N$, the barycenter of $\tau_{\textnormal{n-inv}}$ is also in $W_N$, so the exceptional divisor of $\on{Bl}_{V_{\tau_{\textnormal{n-inv}}}}(X_{\Sigma'_\sigma})\to X_{\Sigma'_\sigma}$ is foliation non-invariant. As a result, if $\widetilde Z$ is the component of $\pi^{-1}(Z)$ containing $z'$, then the exceptional divisor obtained by blowing up $\widetilde Z$ is foliation non-invariant. Finally, since $(\tau,W(\tau))$ is dicritical, we have $\tau \nsubseteq W_N$ and thus $\on{rank}\cF\vert_{V(\tau)}=\dim (W_N+\bC\tau)/\bC\tau>r-
    (n-\dim V(\tau))$. Hence $\dim Z\leq n-r-1$.

    (If part) Suppose $(\Delta_X,W)$ satisfies the condition $(\dagger)$. Let $E$ be an exceptional divisor over $X$ such that $c_X(E)\subseteq X$ has dimension $\leq n-r-1$. We are going to show that $E$ is foliation invariant, and hence, $\cF$ is non-dicritical. First, we choose a smooth subdivision $\Delta_{1}$ of $\Delta_X$, and denote the corresponding toroidal morphism by $\pi\colon X_1\to X$; we may choose $\Delta_1=\Delta_{X_1}$ such that $\on{Exc}(\pi)\subseteq X_1$ is a divisor. If $c_{X_1}(E)$ is a divisor, then we have $c_{X_1}(E) = V(\rho_1)$ for some $\rho_1\in \Delta_{X_1}(1)$ and $\pi(V(\rho_1)) = c_X(E)$. By Corollary~\ref{cor:contained_in_invariant_divisor}(1), $V(\rho_1)$ is foliation invariant, and so is $E$. 
    
    Then we may assume $c_{X_1}(E)$ is not a divisor. 
    Let $\widetilde\pi_1\colon \widetilde X_{2}\to X_1$ be the blow-up of $X_1$ along $c_{X_1}(E)$ and $X_{2}:=\widetilde\pi_1^{-1}(U_1)$, where $U_1\subseteq X_1$ is defined as in Lemma~\ref{lem:blowup_toroidal_model}. 
    If we write $\pi_1:=\widetilde\pi_1\vert_{X_{2}}\colon X_{2}\to X_1$, then
    the pullback foliation $\cF_{2}$ on $X_{2}$ is toroidal, the associated extended complex $(\Delta_{X_{2}},W_{2})$ satisfies the condition $(\dagger)$, and $\on{Exc}(\pi_{1})=V(\rho_2)$ is a prime divisor where $\rho_2\in\Delta_{X_2}(1)$. 
    \begin{claim}
    $V(\rho_2)$ is foliation invariant.
    \end{claim}
    \begin{claimproof}
    Let $\tau\in\Delta$ be the unique cone such that $c_X(E)\cap O(\tau)\neq\emptyset$, and choose a semi-local model $(U_{\sigma,N},W_N)$ of $\cF$ at a general point $z\in Z=c_X(E)$ such that $Z$ is given by $V_{\tau_Z}$ formally locally for some $\tau\preceq\tau_Z\preceq\sigma$.  Let $\Sigma'_\sigma$ be the subdivision of $\Sigma_\sigma$ consisting of $\on{Cone}(\tau',\sigma')$ where $\tau'\in \Delta'$ ranges over all cones $\tau'\subseteq \tau$ and $\sigma'\preceq\sigma_1$ ranges over all faces of $\sigma_1$. Then formally locally near $z$, $\pi_1\colon X_1\to X$ is given by $X_{\Sigma'_\sigma}\to U_{\sigma,N}$, and we have the same diagram as in the only if part. Let $z'\in c_{X_1}(E)\cap\pi^{-1}(z)$ be a general point. Then there exists $\sigma_1\in\Sigma'_\sigma$, such that $\sigma_1\cap \on{Relint}(\tau_Z)\neq\emptyset$, $\alpha(z')\in V_{\sigma_1}$, and $c_{X_1}(E)$ is contained in $V_{\sigma_1}$ formally locally near $z'$. Let $S$ be the set of rays $\rho$ in $\Sigma'_\sigma$ such that there is a common cone in $\Sigma'_\sigma$ containing $\rho$ and $\sigma_1$. Then $S$ spans a vector space of dimension $\geq \dim \tau_Z\geq\on{rank}(\cF)+1$. Hence there is a ray in $S$ not contained in $W_N$, which says that formally locally near $z'$, $V_{\sigma_1}$ is contained in some foliation invariant torus divisor. Now we may apply Corollary~\ref{cor:contained_in_invariant_divisor}(2) to conclude that $V(\rho_2)=\on{Exc}(\pi_1)$ is foliation invariant. $\quad\blacksquare$
\end{claimproof}
    
    Having constructed $X_i$ such that 
    \begin{itemize}
        \item $c_{X_i}(E)$ is not a divisor,
        \item $c_{X_i}(E)\subseteq D_i=V(\rho_i)$ for some $\rho_i\in\Delta_{X_i}(1)$, and 
        \item $D_i$ is foliation invariant,
    \end{itemize}
    let $\widetilde\pi_i\colon \widetilde X_{i+1}\to X_i$ be the blow-up of $X_i$ along $c_{X_i}(E)$, and let $X_{i+1}:=\widetilde\pi_i^{-1}(U)$, where $U\subseteq X_i$ is defined as in Lemma~\ref{lem:blowup_toroidal_model}. If we write $\pi_i:=\widetilde\pi_i\vert_{X_{i+1}}\colon X_{i+1}\to X_i$, then the pullback foliation $\cF_{i+1}$ on $X_{i+1}$ is toroidal, the associated extended complex $(\Delta_{X_{i+1}},W_{i+1})$ satisfies the condition $(\dagger)$, $\on{Exc}(\pi_{i})=V(\rho_{i+1})$ for some $\rho_{i+1}\in\Delta_{X_{i+1}}(1)$, and $\on{Exc}(\pi_{i})$ is foliation invariant by Corollary~\ref{cor:contained_in_invariant_divisor}(2). By Zariski's lemma (cf. \cite[Lemma 2.45]{KM98}), $E=\on{Exc}(\pi_{\ell-1})\subseteq X_\ell$ is foliation invariant for some $\ell$, and this completes the proof.
\end{proof}

\section{Singularities of toric and toroidal foliated pairs}\label{sec:Singularities}
\subsection{Support functions}
We first fix the following definition (cf. \cite[Definition 4.2.11]{CLS11}):

\begin{definition}\label{notation1}
\begin{enumerate}
    \item Let $\Delta$ be a conical complex. (See Section~\ref{sec:conical_cplx}.)
    A function $\phi\colon |\Delta| \to \bR$ is called a \emph{support function} if $\phi$ is linear on each $\sigma\in\Delta$. 
    \item Let $\cF$ be a strict toroidal foliation on $X$ with associated extended complex $(\Delta,W)$, and let $(\cF,D)$ be a toroidal foliated pair. 
    Define $\vp_{(\cF,D)}\colon|\Delta|\to\bR$ to be the support function associated to $-\sum_{\rho\subseteq W(\rho)}V(\rho)+D$; that is, $\vp_{(\cF,D)}$ is linear on each cone of $\Delta$, and $\vp_{(\cF,D)}(v_\rho)=\iota(V(\rho))-\on{mult}_{V(\rho)}D$ for all $\rho\in\Delta(1)$. Note that this actually defines a linear function on each cone $\sigma\in\Delta$: take a local model $(U_{\sigma,N},W_N)$ of $\cF$ at a point of $O(\sigma)$, and suppose that $D$ corresponds to the toric divisor $D_{\sigma,N}$ on $U_{\sigma,N}$. Then $\vp_{(\cF,D)}\vert_{\sigma}=\vp_{K}\vert_\sigma$ where $K=K_{\cF_{W_N}}+D_{\sigma,N}$, which is $\bR$-Cartier by assumption. 
\end{enumerate}
\end{definition}

\begin{lemma}\label{non_neg_support_fcn}
    Let $(\cF,D = \sum_{\rho\in\Delta(1)} d_\rho V(\rho))$ be a toroidal foliated pair on a normal variety $X$. 
    Suppose $d_\rho\leq\iota(V(\rho))$ for all $\rho\in\Delta(1)$, then the support function $\phi_{(\cF,D)}$ is non-negative. 
\end{lemma}
\begin{proof}
    It is straightforward. 
\end{proof}

\begin{proposition}\label{discrepancy}
    Suppose $(\cF,D)$ is a toroidal foliated pair on a normal variety $X$ with associated extended complex $(\Delta,W)$. 
    Let $v_0\in \on{Relint}(\sigma)\cap N_\sigma$ be a primitive element for some $\sigma\in\Delta$ and $\rho_0$ be the ray generated by $v_0$. 
    Let $\Delta^*(v_0)$ be the star subdivision of $\Delta$ at $v_0$ and $\pi\colon X'\to X$ be the associated birational toroidal morphism 
    (cf. Proposition-Definition~\ref{propdef:star_subdiv_complex}(2) and Theorem~\ref{thm:toroidal_subdivision}) with $\cF':=\pi^{-1}\cF$. 
    Then we have
    \begin{enumerate}
        \item two support functions $\phi_{(\cF,D)}$ and $\phi_{(\cF',D')}$ coincide where $D' = \pi^*(K_\cF+D)-K_{\cF'}$, 
        \item $K_{\cF'}+\pi_*^{-1}D$ is $\bR$-Cartier and the support function $\phi_{(\cF',\pi_*^{-1}D)}$ has value $\iota(V(\rho_0))$ at $v_0$, and 
        \item $\phi_{(\cF,D)}(v_0) = \iota(V(\rho_0))+a(V(\rho_0),\cF,D)$. 
    \end{enumerate}
\end{proposition}
\begin{proof}
    Note that $V(\rho_0)$ is the exceptional divisor of $\pi$ and is $\bQ$-Cartier. (For a reference, see \cite[Proposition 11.1.6]{CLS11}.) 
    Thus, we have 
    \begin{equation}\label{eq:discrepancy}
        K_{\cF'}+\pi_*^{-1}D = \pi^*(K_{\cF}+D)+a(V(\rho_0),\cF,D)V(\rho_0)
    \end{equation}
    is $\bR$-Cartier and the support function $\phi_{(\cF',\pi_*^{-1}D)}$ has value $-(-\iota(V(\rho_0)))=\iota(V(\rho_0))$ at $v_0$. 
    This shows (2). 

    To show (1) and (3), we first assume that $(\cF,D)$ is a toric foliated pair on a normal toric variety $X_\Sigma$ where $\Sigma$ is a fan in $N_\bR$. 
    Then (1) follows from \cite[Proposition 6.2.7]{CLS11} and (3) follows from the following observation: 
    \begin{align*}
            a(V(\rho_0),\cF,D) &= \on{ord}_{V(\rho_0)}\big((K_{\cF'}+\pi_*^{-1}D)-\pi^*(K_{\cF}+D)\big) \\
            &= -\phi_{K_{\cF'}+\pi_*^{-1}D}(v_0)+\phi_{\pi^*(K_{\cF}+D)}(v_0) \\ 
            &= -\iota(V(\rho_0))+\phi_{K_\cF+D}(v_0).
    \end{align*}
    
    In general, let $x\in O(\sigma)$ and $(U_{\sigma,\,N},W_N)$ be a semi-local model of $\cF$ at $x$. 
    Then 
    \[\phi_{(\cF,D)}(v_0)=\phi_{K_{\cF_{W_N}}+D_{\sigma,\,N}}(v_0)=\iota(V(\rho_0))+a(V(\rho_0),\cF_{W_N},D_{\sigma,\,N}) = \iota(V(\rho_0))+a(V(\rho_0),\cF,D)\]
    where the first and the last equalities follows from the property of semi-local models. 
    Thus, this shows (3). 
    For (1), it suffices to show $\phi_{(\cF',D')}(v_0) = \phi_{(\cF,D)}(v_0)$, which follows from (3) and the equation~(\ref{eq:discrepancy}). 
\end{proof}

\subsection{Foliated log smooth pairs}
Similar to \cite[Definition 3.1]{CS} and \cite[Section 3.2]{ambro2021positivity}, we introduce the definition for a foliated pair of arbitrary rank to be foliated log smooth as follows: 
\begin{definition}\label{LogSmAnyRank}
\begin{enumerate}
\item
A foliation $\cF$ on a normal variety $X$ is called \emph{foliated log smooth} if $\cF$ is toroidal, $X$ is $\bQ$-factorial, and the associated extended complex $(\Delta,W)$ satisfies the condition $(\dagger)$. 

\item
We say a foliated pair $(\cF,D)$ on a normal variety $X$ is \emph{foliated log smooth} if $(\cF,D)$ is a toroidal foliated pair and $\cF$ is foliated log smooth. 
\end{enumerate}
\end{definition}

\begin{remark}\label{inv_log_sm}
    \begin{enumerate}
        \item When $\cF$ is an algebraically integrable toric foliation, our definition is slightly weaker than the one in \cite[Definition 2.17]{das2023acc} as we do not have any requirement on the base of the morphism that induces $\cF$. 
        \item Let $\cF$ be a corank one foliation on a normal variety $X$. 
        Our definition for a foliated log smooth pair is different from \cite[Definition 3.1]{CS}, which requires $X$ to be smooth. Also, when $X$ is smooth, our definition requires $\mc F$ to have simple singularities of type 1, while loc. cit. allows type 2. 
    \end{enumerate}
\end{remark}

\begin{proposition}\label{SimpleCanonical}
    Suppose $\cF$ is a toroidal foliation on a smooth toroidal embedding $X$ that is foliated log smooth. 
    Then $\cF$ has only canonical singularities.
\end{proposition}
\begin{proof}
    Let $E$ be a divisor on $Y$ over $X$ with center $Z:=c_X(E)\subseteq X$.
    After shrinking around the generic point of $Z$, we can assume that $Z$ is smooth. 
    By Zariski's lemma (cf. \cite[Lemma 2.45]{KM98}), after possibly replacing $Y$ by a higher model, we can assume that $\pi$ is a composition of blow-ups of subvarieties centered on $Z$. 
    We proceed by induction on the number of blow-ups. 
    Thus, it suffices to show that if $\pi\colon \widetilde{X}\to X$ is the blow-up along $Z$, then 
    \begin{enumerate}
        \item $\pi^{-1}\cF$ is foliated log smooth in a neighborhood of $\pi^{-1}(z)$, where $z\in Z$ is a general point, and
        \item $a(E_0,\cF)\geq 0$, where $E_0$ is the exceptional divisor of $\pi$. 
    \end{enumerate}
    
    (1) follows from Lemma~\ref{lem:blowup_toroidal_model}. 
    
    For (2), we first use Lemma~\ref{lem:exist_semi_local_model}(2) to get a semi-local model $(U_{\sigma,N},W_N)$ at a general point $z\in Z$ and $Z$ corresponds to $V_{\tau_Z}$ formally locally where $\tau_Z\preceq\sigma$. 
    Note that $s:=\dim\tau_Z\geq 2$ as $Z$ has codimension at least $2$. 
    Write $N=\oplus_{i=1}^n\bZ e_i$, $\sigma=\on{Cone}(e_1,\ldots,e_n)$, and $\tau_Z=\on{Cone}(e_1,\ldots,e_s)$. 
    Let $u_0=\sum_{i=1}^s e_i$ be the barycenter of $\tau_Z$. 
    Then the log discrepancy $\iota(E_0)+a(E_0,\cF) = \phi_{K_{\cF_{W_N}}}(u_0)\geq 0$ as $\phi_{K_{\cF_{W_N}}}(e_i)\geq 0$ for all $1\leq i\leq n$. 
    Moreover, if $\iota(E_0)=1$, then $u_0\in W_N$ by Lemma~\ref{lem:complex_div_inv}, and hence by the condition $(\dagger)$, we have $e_i\in W_N$ for all $1\leq i\leq s$. 
    Therefore, $\phi_{K_{\cF_{W_N}}}(u)=\sum_{i=1}^s\phi_{K_{\cF_{W_N}}}(e_i)\geq s\geq 2$. 
    To sum up, we have the discrepancy $a(E_0,\cF)\geq 0$. 
\end{proof}

\begin{theorem}\label{Flogsm_lc}
    Let $(\cF,D = \sum d_iD_i)$ be a foliated log smooth toroidal foliated pair on a normal variety $X$. 
    Then $(\cF,D)$ is log canonical if and only if $d_i\leq\iota(D_i)$ for all $i$. 
\end{theorem}
\begin{proof}
    Let $(\Delta,W)$ be the extended complex of the toroidal foliation $\cF$. 
    If $d_i>\iota(D_i)$ for some $i$, then the support function $\phi_{(\cF,D)}$ has a negative value at some primitive element $u\in N_\sigma$ for some $\sigma\in \Delta$. 
    Let $\Delta^*(u)$ be the star subdivision at $u$ and $\pi\colon X'\to X$ be the corresponding birational morphism with the exceptional divisor $E$. 
    By Proposition~\ref{discrepancy}, we have $\iota(E)+a(E,\cF,D) = \phi_{(\cF,D)}(u)<0$ and thus, $(\cF,D)$ is not log canonical. 
    
    Conversely, suppose $d_i\leq\iota(D_i)$ for all $i$. 
    Let $f\colon Y\to X$ be a birational morphism and $E$ be a divisor on $Y$ with center $Z:=c_X(E)\subseteq X$. 
    We will show $\iota(E)+a(E,\cF,D)\geq 0$. 
    
    We first assume that $X$ is smooth. 
    After shrinking around the generic point of $Z$, we can assume that $Z$ is smooth. 
    By Zariski's lemma (cf. \cite[Lemma 2.45]{KM98}), after possibly replacing $Y$ by a higher model, we can assume that $\pi$ is a composition of blow-ups of subvarieties centered on $Z$. 
    We proceed by induction on the number of blow-ups. 
    Thus, it suffices to show that if $\pi\colon \widetilde{X}\to X$ is the blow-up along $Z$, then 
    \begin{enumerate}
        \item $(\pi^{-1}\cF,\widetilde{D})$ is toroidal log smooth in a neighborhood of $\pi^{-1}(z)$, where $\widetilde{D}=\pi^*(K_\cF+D)-K_{\pi^{-1}\cF}$ and $z\in Z$ is a general point, and
        \item $\iota(E_0)+a(E_0,\cF,D)\geq 0$, where $E_0$ is the exceptional divisor of $\pi$. 
    \end{enumerate}
    
    (1) follows from Lemma~\ref{lem:blowup_toroidal_model}. 
    
    For (2), we first use Lemma~\ref{lem:exist_semi_local_model} to get a semi-local model $(U_{\sigma,N},W_N,D_{\sigma,N})$ at a general point $z\in Z$ and $Z$ corresponds to $V_{\tau_Z}$ where $\tau_Z\preceq\sigma$. 
    Note that $\phi_{K_{\cF_{W_N}}+D_\sigma}$ is non-negative by Lemma~\ref{non_neg_support_fcn}, we have $\iota(E_0)+a(E_0,\cF,D) = \phi_{K_{\cF_{W_N}}+D_{\sigma,N}}(u_0)\geq 0$ where $u_0$ is the barycenter of $\tau_Z$. 

    Now suppose $X$ is not smooth. 
    Let $\Delta'$ be a smooth conical complex which is a subdivision of $\Delta$. 
    Then this subdivision introduces a birational morphism $\pi\colon Y\to X$ with $Y$ smooth. 
    Then we define $D_Y = \pi^*(K_\cF+D)-K_{\pi^{-1}\cF}$. 
    Thus, for any prime $\pi$-exceptional divisor $V(\rho)$ with $\rho\in\Delta'(1)\setminus\Delta(1)$, we have $\iota(V(\rho)) + a(V(\rho),\cF,D) = \phi_{(\cF,D)}(u)\geq 0$ by Proposition~\ref{discrepancy} and Lemma~\ref{non_neg_support_fcn} where $u$ is the primitive element in $N_\rho$. 
    Hence, $\on{ord}_{V(\rho)}D_Y = -a(V(\rho),\cF,D)\leq\iota(V(\rho))$. 
    Therefore, $(\pi^{-1}\cF,D_Y)$ is log canonical and so is $(\cF,D)$. 
\end{proof}

\subsection{Foliated log resolution}

\begin{definition}\label{fol_log_res_defn}
    A birational morphism $\pi\colon Y\to X$ is a \emph{foliated log resolution} of a foliated pair $(\cF,D)$ on a normal variety $X$ if $E:=\on{Exc}(\pi)$ is a divisor and the foliated pair $(\pi^{-1}\cF,\pi_*^{-1}D+E)$ is foliated log smooth.
\end{definition}

\begin{remark}
    By \cite[Theorem 2.1 and Proposition 4.4]{AK00}, if $\cF$ is algebraically integrable, any foliated pair $(\cF,\Delta)$ admits a foliated log resolution. 
    In Theorem~\ref{FLogRes}, we show that every toroidal foliated pair admits a foliated log resolution.
\end{remark}

\begin{proposition}\label{dagger_resolution}
    Let $\Sigma$ be a fan in $N\otimes\bR$ and $W\subseteq N\otimes\bK$ be a $\bK$-vector subspace where $\bK = \bR$ or $\bC$. 
    Then there is a simplicial fan $\Sigma'$ in $N\otimes\bR$ obtained from $\Sigma$ by performing a sequence of star subdivisions such that $(\Sigma',W)$ satisfies the condition $(\dagger)$. 
\end{proposition}
\begin{proof}
    By \cite[Exercise 11.1.10]{CLS11}, after performing a sequence of barycentric subdivisions, we have a simplicial fan refining $\Sigma$. 
    Thus we can assume that $\Sigma$ is simplicial. 
    Note that $\Sigma$ remains simplicial if we perform more star subdivisions.
    
    We will construct a sequence of fans $\Sigma_1,\ldots,\Sigma_n$ such that each $\Sigma_k$ is obtained by performing a sequence of star subdivisions starting from $\Sigma_{k-1}$, and $(\sigma,W(\sigma))$ is non-dicritical for all $\sigma\in\Sigma_k$ with $\dim\sigma\leq k$. It is clear that we can take $\Sigma_1=\Sigma$ as $(\rho,W(\rho))$ is non-dicritical for all $\rho\in\Delta(1)$. Suppose we have constructed $\Sigma_1,\ldots,\Sigma_{k-1}$ for some $k\geq 2$.
    
    \begin{claim}
        For any $\sigma\in\Sigma_{k-1}(k)$ with $
        \on{Relint}(\sigma)\cap W\cap N\neq\emptyset$, we have either $\sigma\subseteq W$ or there is a unique rational ray in $W$ that intersects $\on{Relint}(\sigma)$. 
    \end{claim}
    \begin{claimproof}
        Suppose there exist distinct rational rays $\rho_1$ and $\rho_2$ in $\sigma\cap W$, each of which is not contained in any proper face of $\sigma$.
        Then $(\bR\rho_1+\bR\rho_2)\cap(\sigma\setminus\on{Relint}(\sigma))$ consists of exactly two distinct rational rays in $W$, say $\rho'_1$, $\rho'_2$. 
        There exist proper faces $\tau_1$ and $\tau_2$ of $\sigma$ such that $\rho'_i\cap\on{Relint}(\tau_i)\neq\emptyset$ for $i=1$, $2$. 
        If there exists a proper face $\tau$ of $\sigma$ containing $\tau_1$ and $\tau_2$, then we have 
        $\rho_i\subseteq\on{Cone}(\rho'_1,\rho'_2)\subseteq\on{Cone}(\tau_1,\tau_2)\subseteq\tau$, which is absurd. 
        Hence $\langle \tau_1,\tau_2\rangle=\sigma$. 
        By assumption, we have $\tau_i\subseteq W$ for $i=1$, $2$ and thus $\sigma\subseteq W$. 
        This completes the proof of the claim. $\quad\blacksquare$
    \end{claimproof}
    
    Let $S_k$ be the set of rational rays $\rho\subseteq N\otimes\bR$ such that $\rho$ is the unique rational ray in $W$ that intersects $\on{Relint}(\sigma)$ for some $\sigma\in\Delta_{k-1}(k)$.
    Consider $\Sigma_k$ obtained by by performing a sequence of star subdivisions starting from $\Sigma_{k-1}$ along the rays in $S_k$ in any order. Then one can check that $(\sigma,W(\sigma))$ is non-dicritical for all $\sigma\in\Sigma_k$ with $\dim\sigma\leq k$. 
    Continuing the process, we get $\Sigma_n$ as required. 
\end{proof}

\begin{remark}\label{rem:dagger_for_conical_complex}
    Proposition~\ref{dagger_resolution} applies to extended complexes equally well.
\end{remark}

\begin{theorem}\label{FLogRes}
    Let $(\cF,D)$ be a toroidal foliated pair on a normal variety $X$.
    Then we have the following:
    \begin{enumerate}
        \item There is a birational morphism $\pi\colon Y \to X$ such that $\on{Exc}(\pi)$ is a divisor and $(\pi^{-1}\cF,\pi_*^{-1}D+\on{Exc}(\pi))$ is foliated log smooth.
        \item Moreover, we can make $Y$ smooth. 
        \item Suppose further that $(\cF,D)$ is a toric foliated pair on a toric variety $X$, then we can choose $\pi$ to be a toric morphism between toric varieties.
    \end{enumerate}
\end{theorem}
\begin{proof}
    Let $(\Delta,W)$ be the associated extended complex of $\cF$. 
    By Proposition~\ref{dagger_resolution} and Remark~\ref{rem:dagger_for_conical_complex}, there is a subdivision $\Delta'$ of $\Delta$ such that $(\Delta',W')$ satisfies the condition $(\dagger)$. 
    Let $\pi\colon X'\to X$ be the birational toroidal morphism associated to the subdivision (cf. Theorem~\ref{thm:toroidal_subdivision}). 
    Then $\pi^{-1}\cF$ is foliated log smooth and so is $(\pi^{-1}\cF,\pi_*^{-1}D+\on{Exc}(\pi))$. 

    Moreover, let $\Delta''$ be a subdivision of $\Delta'$ such that $\Delta''$ is smooth. 
    Then we get a birational morphism $\psi\colon X''\to X$ such that $X''$ is smooth and $\psi^{-1}\cF$ is foliated log smooth. 
\end{proof}

\begin{definition}\label{defn_F_dlt}
    A foliated pair $(\cF,D)$ on a normal variety $X$ is \emph{foliated divisorial log terminal (F-dlt)} if 
    \begin{enumerate}
        \item each irreducible component of $D$ is non-$\cF$-invariant and has a coefficient between $0$ and $1$, and 
        \item there exists a foliated log resolution $\pi\colon Y\to X$ of $(\cF,D)$ which only extracts divisors $E$ of discrepancy $> -\iota(E)$.
    \end{enumerate}
\end{definition}

Similar to \cite[Remark 3.7 and Lemma 3.8]{CS}, we have the following properties for F-dlt foliated pairs:
\begin{proposition}\label{F_dlt_property}
    Let $(\cF,D)$ be an F-dlt foliated pair on a normal variety $X$. 
    Then $(\cF,D)$ is log canonical. 
    Moreover, it is foliated log smooth at the generic point of any lc center of $(\cF,D)$. 
\end{proposition}
\begin{proof}
    The proof follows the same arguments as in the one of \cite[Remark 3.7 and Lemma 3.8]{CS}. 

    As $(\cF,D)$ is F-dlt, there is a foliated log resolution $\pi\colon Y\to X$ which only extracts divisors $E$ of discrepancy $>-\iota(E)$. 
    Let $\cG:=\pi^{-1}\cF$ and $\Gamma:=\pi_*^{-1}D$. 
    Then we may write $K_\cG+\Gamma+F = \pi^*(K_\cF+D)+G$ where $F$, $G$ are $\pi$-exceptional effective divisors without common components. 
    Note that no component of $\Gamma+F$ is $\cG$-invariant. 
    By Theorem~\ref{Flogsm_lc}, $(\cG,\Gamma+F)$ is log canonical. 
    Thus, for any divisor $E$ over $X$, we have 
    \[a(E,\cF,D)\geq a(E,\cG,\Gamma+F)\geq -\iota(E).\]
    Therefore, $(\cF,D)$ is log canonical. 

    Now let $Z$ be an lc center of $(\cF,D)$. 
    Then there is a divisor $T$ whose center on $X$ is $Z$ and with discrepancy $a(T,\cF,D)=-\iota(T)$. 
    For the sake of contradiction, we assume that $(\cF,D)$ is not foliated log smooth at the generic point of $Z$. 
    So $\pi$ is not an isomorphism at the generic point of $\pi^{-1}(Z)$. 
    Since $\on{Exc}(\pi)$ is a divisor, there exists a $\pi$-exceptional prime divisor $E$ which contains the center of $T$ on $Y$. 
    
    If $E$ is $\cG$-invariant, then $E$ is contained in the support of $G$ and thus, $a(T,\cG,\Gamma+F) < a(T,\cF,D)=-\iota(T)$, which contradicts the log canonicity of $(\cG,\Gamma+F)$. 

    If $E$ is non-$\cG$-invariant, then $E$ is contained in the support of $F$. 
    Note that $\lfloor F\rfloor=0$. 
    Then there is a $\delta>0$ such that if $F':= F+\delta E$, then we have $\lfloor F+\delta E\rfloor=0$, $(\cG,\Gamma+F')$ is foliated log smooth, and $a(T,\cG,\Gamma+F')<a(T,\cG,\Gamma+F)\leq a(T,\cF,D)=-\iota(T)$, which is impossible as $(\cG,\Gamma+F')$ is log canonical by Theorem~\ref{Flogsm_lc}. 
\end{proof}

\subsection{Toric description for various singularities}
In this subsection, we provide toric descriptions for many singularities and show the relations among them.

The following Proposition generalizes \cite[Lemma 3.1]{ambro2021positivity}. 
\begin{proposition}\label{lc}
    Let $(\cF,D=\sum_id_iD_i)$ be a toroidal foliated pair on a normal variety $X$. 
    Then $(\cF,D)$ is log canonical if and only if $d_i\leq\iota(D_i)$ for all $i$. 
    In particular, a toric foliated pair $(\cF_W,D=\sum_{\rho\in\Sigma(1)} d_\rho D_\rho)$ on a toric variety $X_\Sigma$ of a fan $\Sigma$ in $N\otimes\bR$ is log canonical if and only if $d_\rho\leq 1$ for $\rho\subseteq W$ and $d_\rho\leq 0$ for $\rho\nsubseteq W$. 
\end{proposition}
\begin{proof}
    Arguing as in the proof of Theorem~\ref{Flogsm_lc}, we have $(\cF,D)$ is not log canonical if $d_i>\iota(D_i)$ for some $i$. 

    Conversely, suppose $d_i\leq \iota(D_i)$ for all $i$. 
    Let $(\Delta,W)$ be the associated extended complex of $\cF$. 
    By Theorem~\ref{FLogRes}, there is a birational morphism $\pi\colon Y\to X$ such that $Y$ is smooth and $(\pi^{-1}\cF,\pi_*^{-1}\Delta+\on{Exc}(\pi))$ is foliated log smooth. 
    Let $D'=\pi^*(K_\cF+D)-K_{\pi^{-1}\cF}$. 
    By Lemma~\ref{non_neg_support_fcn} and Proposition~\ref{discrepancy}, the support function $\phi_{(\pi^{-1}\cF,D')} = \phi_{(\cF,D)}$ is non-negative. 
    Moreover, for any $\pi$-exceptional prime divisor $E$, the log discrepancy $\iota(E)+a(E,\cF,D) = \phi_{(\cF,D)}(u)\geq 0$ where $u$ is a primitive element. 
    Thus, $\on{mult}_E D'=-a(E,\cF,D)\leq\iota(E)$. 
    Therefore, by Theorem~\ref{Flogsm_lc}, $(\pi^{-1}\cF,D')$ is log canonical and so is $(\cF,D)$. 
\end{proof}

For any vector subspace $W\subseteq N \otimes\bR$ and a full-dimensional cone $\sigma\subseteq N \otimes\bR$ where $N$ is a lattice, we define 
\begin{align*}
    I_{\sigma,\,W} &=\{\rho\mid \rho\in\sigma(1)\textnormal{ and } \rho\subseteq W\}\textnormal{ and} \\
    \Pi_{\sigma,\,W} &= \on{Conv}(0, v_\rho\mid\rho\in I_{\sigma,\,W}) + \on{Cone}(v_\rho\mid\rho\in\sigma(1)\setminus I_{\sigma,\,W})
\end{align*}
where $v_\rho$ is the primitive element of $\rho\cap N$, the sum is the Minkowski sum, and the first summand is the convex hull of the set containing $0$ and $v_\rho$ for $\rho\in I_{\sigma,\,W}$.

\begin{proposition}\label{can_term_toric}
    Let $\cF$ be a toroidal foliation on a normal variety $X$ with associated extended complex $(\Delta,W)$, such that $K_\cF$ is $\bR$-Cartier (or equivalently, $(\cF,0)$ is a toroidal foliated pair). 
    Then we have the following:
    \begin{enumerate}
        \item For each $\sigma\in\Delta$, $\Pi_{\sigma,\,W(\sigma)}$ has a unique facet ($=$codimension one face) $F_\sigma$ that does not contain the origin. 
        \item $\cF$ is canonical if and only if for any $\sigma\in\Delta$, all the non-zero elements of $\Pi_{\sigma,\,W(\sigma)}\cap W(\sigma)\cap N_\sigma$ are contained in $F_\sigma$.
        \item For any $\sigma\in\Delta$, $\cF$ is terminal at the generic point of $V(\sigma)$ if and only if $\Pi_{\sigma,\,W(\sigma)}\neq\sigma$ and $\on{Relint}(\sigma)\cap\Pi_{\sigma,\,W(\sigma)}\cap W(\sigma)\cap N_\sigma=\emptyset$. 
    \end{enumerate}
\end{proposition}
\begin{proof}
    For (1), as $K_{\cF}$ is $\bR$-Cartier, the support function $\phi_{(\cF,0)}$ is linear on each $\sigma\in\Delta$. 
    Hence, there exists $m_\sigma\in \on{Hom}(N_\sigma,\bZ)\otimes\bR$ such that $\phi_{(\cF,0)}(u) = \langle m_\sigma, u\rangle$ for $u\in\sigma$. 

    Now we fix a cone $\sigma\in\Sigma$. 
    By \cite[Lemma 7.1.1]{CLS11}, all vertices of $\Pi_{\sigma,\,W(\sigma)}$ are contained in $\{0,v_\rho\mid\rho\in I_{\sigma,\,W(\sigma)}\}$. 
    Note that we have $\langle m_\sigma, v_\rho\rangle = \phi_{(\cF,0)}(v_\rho) = 1$ for $\rho\in I_{\sigma,\,W}$ and $\langle m_\sigma, v_\rho\rangle = \phi_{(\cF,0)}(v_\rho) = 0$ for $\rho\in\sigma(1)\setminus I_{\sigma,\,W(\sigma)}$. 
    So 
    \begin{align*}
        F_\sigma &:= \Pi_{\sigma,\,W(\sigma)}\cap\{u\mid\langle m_\sigma, u\rangle=1\} \\
         &= \on{Conv}(v_\rho\mid\rho\in I_{\sigma,\,W(\sigma)}) + \on{Cone}(v_\rho\mid\rho\in\sigma(1)\setminus I_{\sigma,\,W(\sigma)})
    \end{align*}
    is a facet of $\Pi_{\sigma,\,W(\sigma)}$, which contains all the vertices of $\Pi_{\sigma,\,W(\sigma)}$ except the origin. 
    Hence, any facet of $\Pi_{\sigma,\,W(\sigma)}$ other than $F_\sigma$ must contain the origin, and (1) follows. 

    For (2), suppose $\cF$ is canonical. 
    Let $u$ be an element in $\Pi_{\sigma,\,W(\sigma)}\cap W(\sigma)\cap N_\sigma$. Write $u=nu'$ where $n\in \bN$ and $u'$ is primitive. 
    Let $\Delta'=\Delta^*(u')$ the star subdivision at $u'$ and $\rho_u$ be the ray generated by $u$. 
    Then we have a birational morphism $X'\to X$ with the exceptional divisor $E=V(\rho_u)$. 
    Thus, the discrepancy $0\leq a(V(\rho_u),\cF) = \phi_{(\cF,0)}(u')-\iota(V(\rho_u))\leq \frac{1}{n}-1$. 
    Hence, $n=1$, $u=u'$, and $\phi_{(\cF,0)}(u)=1$. 
    We thus have $u\in F_\sigma$. 
        
    Conversely, assume that for any $\sigma\in\Delta$, all the non-zero elements of $\Pi_{\sigma,\,W(\sigma)}\cap W(\sigma)\cap N_\sigma$ are contained in $F_\sigma$. 
    By Theorem~\ref{FLogRes}, there is a birational morphism $\pi\colon X'\to X$ induced by a subdivision $\Delta'$ of $\Delta$ such that $X'$ is smooth and $\pi^{-1}\cF$ is foliated log smooth. 
    Let $D':=\pi^*K_\cF - K_{\pi^{-1}\cF}$. 
    Note that $D'$ is supported on $V(\rho)$ where $\rho\in\Delta'(1)\setminus\Delta(1)$. 
    Moreover, $-D'$ is effective as $\on{mult}_{V(\rho)}D' = -a(V(\rho),\cF) = -(\phi_{(\cF,0)}(v_\rho)-\iota(V(\rho)))\leq 0$ by assumption. This also shows that each of the $\pi$-exceptional divisor has non-negative discrepancy.
    Since $\pi^{-1}\cF$ is foliated log smooth on a smooth variety, it is canonical by Theorem~\ref{SimpleCanonical}. 
    As $-D'$ is effective, $(\pi^{-1}\cF,D')$ is also canonical. 
    Therefore, $\cF$ is canonical. 

    For (3), suppose on the contrary that either
    \begin{enumerate}
        \item $\Pi_{\sigma,\,W(\sigma)}=\sigma$ (hence $I_{\sigma,\,W(\sigma)}=\emptyset$), or 
        \item $\on{Relint}(\sigma)\cap\Pi_{\sigma,\,W(\sigma)}\cap W(\sigma)\cap N_\sigma\neq\emptyset$. 
    \end{enumerate}
    In the first case, let $u\in \on{Relint}(\sigma)\cap N_\sigma$ be a primitive element in $N_\sigma$; in the second case, let $u\in\on{Relint}(\sigma)\cap\Pi_{\sigma,\,W(\sigma)}\cap W(\sigma)\cap N_\sigma$ be a primitive element in $N_\sigma$. 
    Let $\Delta'=\Delta^*(u)$ be the star subdivision of $\Delta$ at $u$ and $\rho_u=\bR_{\geq 0}u$ be the ray generated by $u$. 
    Then we have a birational morphism $X'\to X$ with the exceptional divisor $E=V(\rho_u)$ whose center on $X$ is $V(\sigma)$. The discrepancy $a(E,\cF)$ is $\phi_{(\cF,0)}(u)-\iota(V(\rho_u))\leq 0$ in the first case, and $\leq 1-1=0$ in the second case, both of which contradict the assumption that $\cF$ is terminal at the generic point of $V(\sigma)$. 
    Therefore, we have $\Pi_{\sigma,\,W(\sigma)}\neq\sigma$ and $\on{Relint}(\sigma)\cap\Pi_{\sigma,\,W(\sigma)}\cap W(\sigma)\cap N_\sigma=\emptyset$. 

    Conversely, suppose $\Pi_{\sigma,\,W(\sigma)}\neq\sigma$ and $\on{Relint}(\sigma)\cap\Pi_{\sigma,\,W(\sigma)}\cap W(\sigma)\cap N_\sigma=\emptyset$. 
    By Theorem~\ref{FLogRes}, there is a subdivision $\Delta'$ of $\Delta$ inducing a birational morphism $\pi\colon X'\to X$ such that $X'$ is smooth and $\pi^{-1}\cF$ is foliated log smooth.  
    Let $D' = \pi^*K_{\cF} - K_{\pi^{-1}\cF}$. 
    \begin{claim}
        $-D'$ is effective and $\on{Supp}D'$ is the union of all $\pi$-exceptional divisors after base change via $\eta\to X$, where $\eta$ is the generic point of $V(\sigma)$. 
    \end{claim}
    \begin{claimproof}
        Since $\Pi_{\sigma,\,W(\sigma)}\neq\sigma$, there is a ray $\rho\in\sigma(1)$ such that $\rho\subseteq W(\sigma)$ and thus, $\phi_{K_{\cF}}$ is positive on $\on{Relint}(\sigma)$. After base change, every $\pi$-exceptional divisor is of the form $V(\rho)$ where $\rho\in\Delta'(1)\setminus\Delta(1)$ and $\rho\in\on{Relint}(\sigma)$. Let $v_\rho$ be the primitive element of $\rho$ in $N_\rho$. 
        If $\iota(V(\rho))=0$, then $\phi_{(\cF,0)}(v_\rho)>0=\iota(V(\rho))$. 
        If $\iota(V(\rho))=1$, then $v_\rho\in W(\sigma)$ and thus, $v_\rho\notin\Pi_{\sigma,\,W(\sigma)}$ as $\on{Relint}(\sigma)\cap\Pi_{\sigma,\,W(\sigma)}\cap W(\sigma)\cap N_\sigma=\emptyset$. 
        Hence, $\phi_{(\cF,0)}(v_\rho)>1=\iota(V(\rho))$. 
        Therefore, in both cases, we have 
        \[\on{mult}_{V(\rho)}D' = -a(V(\rho),\cF) = -(\phi_{(\cF,0)}(v_\rho) - \iota(V(\rho)))< 0.\] 
        This completes the proof of the claim. $\quad\blacksquare$
    \end{claimproof}
        
    Let $E$ be a divisor exceptional over $X$ with center $V(\sigma)$. 
    If $E$ is a divisor on $X'$, then $a(E,\cF)=-\on{mult}_E D'>0$. 
    If $E$ is not a divisor on $X'$, then 
    \[a(E,\cF)=a(E,\pi^{-1}\cF,D')> a(E,\pi^{-1}\cF)\geq 0\] 
    where the equality follows from $\pi^*K_{\cF} = K_{\pi^{-1}\cF}+D'$, the last inequality holds true by Proposition~\ref{SimpleCanonical}, and the first inequality follows since $-D'$ is effective and the center of $E$ on $X'$ is contained in $\on{Supp}D'$, the union of all the $\pi$-exceptional divisors. 
\end{proof}

\begin{proposition}\label{Fdlt_prop}
    Let $(\cF,D)$ be a toroidal foliated pair on a normal variety $X$ with associated extended complex $(\Delta,W)$. 
    Suppose $D$ is effective. 
    Then $(\cF,D)$ is F-dlt if and only if the following statements hold true:
    \begin{enumerate}
        \item $\on{Supp}(D)\subseteq\bigcup_{\rho\in\Delta(1),\,\rho\subseteq W(\rho)}V(\rho)$ and $0\leq\on{mult}_{V(\rho)}D\leq 1$ for any $\rho\in\Delta(1)$. 
        \item For any $\sigma\in\Delta$ satisfying $\phi_{(\cF,D)}\vert_\sigma=0$, we have $\sigma$ is simplicial and $(\sigma,W(\sigma))$ is non-dicritical. 
    \end{enumerate}
\end{proposition}
\begin{proof}
    Note that the condition (1) is equivalent to Definition~\ref{defn_F_dlt}(1). 

    Suppose $(\cF,D)$ is F-dlt. 
    Then there is a foliated log resolution $\pi\colon Y\to X$ such that $a(E,\cF,D)>-\iota(E)$ for any $\pi$-exceptional divisor $E$. 
    For any cone $\sigma\in\Delta$ satisfying $\phi_{(K_{\cF}+D)}\vert_\sigma=0$, we have that $V(\sigma)$ is an lc center of $(\cF,D)$. 
    By Proposition~\ref{F_dlt_property}, $(\cF,D)$ is foliated log smooth at the generic point of $V(\sigma)$. 
    Then $\sigma$ is simplicial and $(\sigma,W(\sigma))$ is non-dicritical. 

    Conversely, suppose that if $\sigma\in\Delta$ and $\phi_{(\cF,D)}\vert_\sigma=0$, then $\sigma$ is simplicial and $(\sigma,W(\sigma))$ is non-dicritical. 
    \begin{claim}
        $(\Delta,W)$ satisfies the condition $(\dagger)$.
    \end{claim}
    \begin{claimproof}
    Fix a cone $\sigma\in\Delta$. 
    If $\dim\sigma=1$, then $(\sigma,W(\sigma))$ is automatically non-dicritical. 
    So we may assume $\dim\sigma\geq 2$. 
    
    As $\phi_{(\cF,D)}\vert_\sigma$ is linear, we may write $\phi_{(\cF,D)}\vert_\sigma(u) = \langle m, u\rangle$ for some non-zero $m\in \on{Hom}(N_\sigma,\bZ)\otimes\bR$ and for $u\in \sigma$. 
    Note that $\sigma_0:=\{u\in\sigma\mid\langle m,u\rangle=0\}$ is a face of $\sigma$ as $\langle m,u\rangle\geq 0$ for all $u\in\sigma$. 
    As $\phi_{(\cF,D)}\vert_{\sigma_0}=0$, we have $\sigma_0$ is simplicial and $(\sigma_0,W(\sigma_0))$ is non-dicritical by assumption. 
    
    Let $v_1,\ldots,v_s\in N_\sigma$ be a minimal set of primitive generators for $\sigma$ and $\sigma_0=\on{Cone}(v_1,\ldots,v_\ell)$ with $\ell\leq s$. 
    Let $\rho_i$ be the ray generated by $v_i$. 
    Note that for all $i\geq \ell+1$, we have $0<\langle m,v_i\rangle = \phi_{(\cF,D)}(v_i)\leq \iota(V(\rho_i))$, $\iota(V(\rho_i))=1$, and thus $v_i\in W(\sigma)$. 

    We now suppose $\on{Relint}(\sigma)\cap W(\sigma)\cap N_\sigma\neq\emptyset$ and will show $\sigma\subseteq W(\sigma)$.  Let $v\in\on{Relint}(\sigma)\cap W(\sigma)\cap N_\sigma$ be a primitive element. 
    We may write $v=\sum_{i=1}^s a_iv_i$ so that all $a_i$ are positive rational numbers. 
    Let \[v':=\sum_{i=1}^{\ell} a_iv_i = v-\sum_{i=\ell+1}^s a_iv_i\in\sigma_0\cap W(\sigma_0).\] 
    Then we have $nv'\in\on{Relint}(\sigma_0)\cap W(\sigma_0)\cap N_{\sigma_0}$ for some $n\in\bN$. As $(\sigma_0,W(\sigma_0))$ is non-dicritical, we have $\sigma_0\subseteq W(\sigma_0)$ because of the existence of $nv'$. That is, $v_i\in W(\sigma)$ for all $1\leq i\leq \ell$, as required. $\quad\blacksquare$
    \end{claimproof}
    
    We will construct a foliated log resolution of $(\cF,D)$ which only extracts divisor $E$ of discrepancy $>-\iota(E)$. 
    
    By \cite[Exercise 11.1.10]{CLS11}, there is a simplicial conical complex $\beta(\Delta)$ which is a subdivision of $\Delta$. 
    We slightly modify the construction of $\beta(\Delta)$ as follows:
    Let $S=\{\sigma\in\Delta\mid \sigma \mbox{ is not simplicial}\}$. 
    We order the cone in $S$ as $\sigma_1,\ldots,\sigma_\ell$ so that $\dim\sigma_1\leq\cdots\leq\dim\sigma_\ell$. 
    Let $v_{\sigma_i}$ be the primitive element of $\bR_{\geq 0}(\sum_{\rho\in\sigma_i(1)}v_\rho)\cap N_{\sigma_i}$. 
    Then $\beta(\Delta)$ is obtained from $\Delta$ by performing the sequence of star subdivisions at $v_{\sigma_1},\ldots,v_{\sigma_\ell}$.

    Let $\pi\colon Y\to X$ be the birational morphism induced by the subdivision $\beta(\Delta)$. Then $(\pi^{-1}\cF,\pi^{-1}_*D+\on{Exc}(\pi))$ is foliated log smooth as $Y$ is $\bQ$-factorial, and the associated extended complex $(\beta(\Delta),W')$ of $\pi^{-1}\cF$ satisfies the condition $(\dagger)$ by Lemma~\ref{lem_property_dagger}(2). Suppose $\rho\in\beta(\Delta)(1)\setminus\Delta(1)$, and let $\tau\in\Delta$ be the minimal cone that contains $\rho$. Then by construction, we must have $\phi_{(\cF,D)}\vert_{\tau}\neq 0$, and therefore, $a(V(\rho),\cF,D)+\iota(V(\rho))=\phi_{(\cF,D)}(v_\rho)>0$, since $v_\rho\in\on{Relint}(\tau)$ and $\phi_{(\cF,D)}\vert_{\on{Relint}(\tau)}> 0$. That is, every $\pi$-exceptional divisor $E$ has discrepancy $>-\iota(E)$. 
    \end{proof}

\begin{proposition}\label{Fdlt_ND_prop}
    Let $(\cF,D)$ be a toroidal foliated pair on a normal variety $X$ with associated extended complex $(\Delta,W)$. 
    If $(\cF,D)$ is F-dlt, then $\cF$ is non-dicritical. 
\end{proposition}
\begin{proof}
    This is the content of the claim in the proof of Proposition~\ref{Fdlt_prop}.
\end{proof}

\begin{proposition}\label{CanImplyND}
    Let $\cF$ be a toroidal foliation on a normal variety $X$ with associated extended complex $(\Delta,W)$, such that $K_\cF$ is $\bR$-Cartier. 
    If $\cF$ is canonical, then it is non-dicritical. 
\end{proposition}
\begin{proof}
    Suppose $\sigma\in\Delta$, and $v\in\on{Relint}(\sigma)\cap N_\sigma\cap W(\sigma)$. We are going to show that $\sigma\subseteq W(\sigma)$. Let $v_1,\ldots,v_s$ be the primitive generators of rays in $\sigma(1)$. Then we may assume that $v_i\in W(\sigma)$ if and only if $1\leq i\leq \ell$ for some $\ell\leq s$, and write $v=\sum_{i=1}^s a_iv_i$ where $a_i$ are all positive rational numbers. Let
    \[
        v':=\sum_{i=\ell+1}^sa_iv_i=v-\sum_{i=1}^\ell a_iv_i\in W(\sigma).
    \]
    Then $nv'\in \Pi_{\sigma,\,W(\sigma)} = \on{Conv}(0,v_i\mid i\leq\ell)+\on{Cone}(v_i\mid i>\ell)$, and hence, $nv'\in\Pi_{\sigma,W(\sigma)}\cap W(\sigma)\cap N_\sigma$ for some $n\in\bN$. On the other hand, since $\phi_{k_\cF}(nv')=0$, we have $nv'\not\in F_\sigma=\Pi_{\sigma,W(\sigma)}\cap\{u\mid\phi_{K_\cF}(u)=1\}$. By Proposition~\ref{can_term_toric}, we must have $nv'=0$, and hence, $\ell=s$ as required.
\end{proof}

The following proposition is analogous to \cite[Theorem 0.2]{spicer2022local} and \cite[Lemma 2.9]{cascini2025mmp} with generalization to toroidal cases. 
\begin{proposition}
     Let $\cF$ be a toroidal foliation on a $\bQ$-factorial normal variety $X$ with associated extended complex $(\Delta,W)$, such that $K_\cF$ is $\bR$-Cartier. 
    \begin{enumerate}
        \item If $\cF$ is terminal, then it is smooth in codimension $2$, that is the singular locus of $\cF$ has codimension at least $3$. 
        \item Suppose $\on{rank}(\cF)=1$. 
        If $\cF$ is terminal at the generic point of $V(\sigma)$ for some cone $\sigma\in\Delta$, then $V(\sigma)\nsubseteq\on{Sing}(\cF)$. 
    \end{enumerate}
\end{proposition}
\begin{proof}
    For (1), let us consider any cone $\sigma=\on{Cone}(u_1,u_2)\in\Delta(2)$ where $u_1$, $u_2$ are primitive. 
    Since $\cF$ is terminal at the generic point of $V(\sigma)$, we have $\Pi_{\sigma,\,W(\sigma)}\neq\sigma$ by Proposition~\ref{can_term_toric}(3). 
    Consequently, one of $u_1$ and $u_2$ is contained in $W(\sigma)$. 
    Therefore, we have the following two cases:
    \begin{itemize}
        \item If both of them are contained in $W(\sigma)$, then $\sigma\subseteq W(\sigma)$ and thus, $V(\sigma)$ is not contained in the singular locus of $\cF$ by Proposition~\ref{singularlocus}. 
        \item If only one of them is contained in $W(\sigma)$, let us say $u_1\in W(\sigma)$, then $W(\sigma)\cap\bR\sigma=\bR u_1$ and thus $V(\sigma)$ is not contained in the singular locus of $\cF$ by Proposition~\ref{singularlocus}. 
    \end{itemize}

    For (2), let $\sigma=\on{Cone}(u_1,\ldots,u_s)\in\Delta$ where $s\geq\dim\sigma$. 
    Since $\cF$ is terminal at the generic point of $V(\sigma)$, we have $\Pi_{\sigma,\,W(\sigma)}\neq\sigma$ by Proposition~\ref{can_term_toric}(3). 
    Consequently, one of $u_1,\ldots,u_s$ is contained in $W(\sigma)$, say $u_1$. 
    As $\on{rank}(\cF)=1$, we have $\dim W(\sigma)\leq 1$ and thus $W(\sigma)=\bR u_1$. 
    Thus, by Proposition~\ref{singularlocus}, $V(\sigma)$ is not contained in the singular locus of $\cF$. 
\end{proof}

\subsection{F-dlt modification}
Following \cite[Definition 3.28]{CS}, we introduce the following definition of an F-dlt modification for a foliated pair of any rank. 
Moreover, we show that any toric foliated pair admits an F-dlt modification. 
\begin{definition}
    Let $(\cF,D=\sum_ia_i D_i)$ be a foliated pair on a normal variety where $D$ is effective. 
    We denote \[\widetilde{D}=\sum_i\min\{a_i,\iota(D_i)\}D_i.\] 
    An \emph{F-dlt modification} for $(\cF,D)$ is a birational projective morphism $\pi\colon Y\to X$ such that if $\cG$ is the pullback foliation on $Y$ then the foliated pair $(\cG,\pi_*^{-1}\widetilde{D}+\sum_i\iota(E_i)E_i)$ is F-dlt, where the sum is over all $\pi$-exceptional divisors, and $K_\cG+\pi_*^{-1}D+\sum_i\iota(E_i)E_i+F=\pi^*(K_\cF+D)$ for some effective $\pi$-exceptional divisor $F$ on $Y$. 
\end{definition}

The existence of F-dlt modifications is shown for corank 1 foliated pairs on normal projective varieties of dimensions at most three in \cite[Theorem 8.1]{CS}. 
We demonstrate the existence of F-dlt modifications for toroidal foliated pairs of any rank. 

\begin{theorem}\label{Fdlt_modification}
    Let $(\cF,D)$ be a toroidal foliated pair on a normal variety $X$ with associated extended complex $(\Delta,W)$. 
    Assume that $D$ is effective. 
    Then $(\cF,D)$ admits an F-dlt (toroidal) modification $\pi\colon Y\to X$ such that $Y$ is $\bQ$-factorial. 
\end{theorem}
\begin{proof}
    By \cite[Proposition 11.1.7]{CLS11}, there is a simplicial subdivision $\Delta'$ of $\Delta$ such that $\Delta'(1)=\Delta(1)$. 
    Then the associated birational morphism $p\colon X' \to X$ is small and projective. 
    Let $\cF'=p^{-1}\cF$ and $D'=p_*^{-1}D$. 
    Then we have $p^*(K_{\cF}+D)=K_{\cF'}+D'$ and $\phi_{(\cF,D)} = \phi_{(\cF',D')}$. 

    Let $S=\{\sigma'\in\Delta':\phi_{(\cF',\widetilde{D'})}\vert_{\sigma'}=0\}$. 
    Now we perform a sequence of star subdivisions as demonstrated in Proposition~\ref{dagger_resolution}, with respect to the cones in $S$. Let $\Delta''$ be the resulting subdivision, $\psi\colon X''\to X'$ be the induced birational morphism, $(\Delta'',W'')$ be the associated extended complex of $\cF'':=\psi^{-1}\cF'$, and $D'':=\psi^{-1}_*\widetilde{D'}=(p\circ\psi)_*^{-1}\widetilde{D}$. 
    One can check that if $\sigma''\in\Delta''$ and $\phi_{(\cF',\widetilde{D'})}\vert_{\sigma''}=0$, then 
    $(\sigma'',W''(\sigma''))$ is non-dicritical.
    Note that by construction, we have $\psi^*(K_{\cF'}+\widetilde{D'})=K_{\cF''}+D''+\sum_i\iota(E_i)E_i$. 
    Hence, they give the same support function, and by Proposition~\ref{Fdlt_prop}, $(\cF'',(p\circ\psi)_*^{-1}\widetilde{D}+\sum_i\iota(E_i)E_i)$ is F-dlt. Thus, $p\circ \psi$ is a desired modification. 
\end{proof}

\section{Toric foliated minimal model program}\label{sec:toric_FMMP}

Throughout this section, we assume that  $\mc F_W$ is a toric foliation on a complete $\mb Q$-factorial toric variety $X_\Sigma$ of dimension $n$. Hence, $W\subseteq N\otimes\bC$ is a vector subspace and $\Sigma$ is a complete simplicial fan in $N\otimes\bR$. 
The minimal model program can be carried out for any  ${\mb Q}$-divisor $D$ on $X_\Sigma$ (see \cite[Chapter 14]{matsuki2002mmp} or \cite[Section 15.4 and 15.5]{CLS11}). That is, the necessary contractions and flips exist, any sequence of flips terminates, and if at some point the divisor becomes nef then at that point it becomes semi-ample. 

\subsection{Preliminaries}\label{subsec_notations} 

Let $R\subseteq\overline{\on{NE}}(X_\Sigma)$ be an extremal ray.
By \cite[Theorem 6.3.20]{CLS11}, $R=\mb R_{\geq 0} [V_\omega]$ for some $\omega\in\Sigma(n-1)$. 
By \cite[Theorem 14-1-9]{matsuki2002mmp} or \cite[Proposition 15.4.1]{CLS11}, there is a toric variety $X_{\Sigma_0}$ and a toric morphism $\phi_R\colon X_\Sigma\to X_{\Sigma_0}$ such that for any $\tau\in\Sigma(n-1)$, $\phi_R(V_\tau)$ is a point if and only if $[V_\tau]\in R$. The fan $\Sigma_0$ is obtained by ``removing'' the walls $\omega\in\Sigma(n-1)$ such that $[V_\omega]\in R$. 

For any $\omega\in \Sigma(n-1)$ with $[V_\omega]\in R$, as $X_\Sigma$ is simplicial, we can write $\omega=\on{Cone}(v_1, \ldots, v_{n-1})$ where each $v_i\in N$ is the primitive generator for some ray $\rho_i\in\Sigma(1)$. 
Since $\Sigma$ is complete, there are two primitive vectors $v_n$ and $v_{n+1}$ in $N$ such that 
\begin{align*}
\sigma^{n+1} &= \on{Cone}(v_1, \ldots, v_{n})\quad\textnormal{and}\\
\sigma^n &= \on{Cone}(v_1, v_2, \ldots, v_{n-1}, v_{n+1}) 
\end{align*}
are $n$-dimensional cones in $\Sigma(n)$. 
There is a unique non-trivial linear relation $\sum_{i=1}^{n+1}a_{i}v_{i}=0$ 
with $a_{n+1}=1$ and $a_i\in\bQ$ for each $i$. 
After re-indexing, we may assume that 
\[a_i\begin{cases}
    <0 & \textnormal{if $1\leq i\leq \alpha$} \\
    =0 & \textnormal{if $\alpha+1\leq i\leq \beta$} \\
    >0 & \textnormal{if $\beta+1\leq i\leq n+1$}
\end{cases}\]
for some $\alpha$, $\beta\in\bZ_{\geq 0}$. The following notations will be used in the subsequent discussions: 

\begin{itemize}
    \item $\sigma(\omega)=\on{Cone}(v_1,\ldots,v_{n+1})$.
    \item $\sigma^j=\on{Cone}(v_1,\ldots, \widehat{v_j},\ldots,v_{n+1})$ for $j\in [1,n+1]\cap\bN$.
    \item $\sigma_J=\on{Cone}(v_j\mid j\in J)$ for any subset $J\subseteq [1,n+1]\cap\bN$.
    \item $J_-=[1,\alpha]\cap\bN$, $J_0=[\alpha+1,\beta]\cap\mb N$, and $J_+=[\beta+1,n+1]\cap\bN$.
\end{itemize}

We recall some facts for later use. There is a decomposition $\sigma(\omega)=\bigcup_{j\in J_+}\sigma^j$ with $\sigma^j\in\Sigma$ for any $j\in J_+$ \cite[Proposition 14-2-1]{matsuki2002mmp}. The exceptional locus $\on{Exc(\phi_R)}$ is $V_{\sigma_{J_-}}$ \cite[Corollary 14-2-2]{matsuki2002mmp}. In particular, $\sigma_{J_-}$ is independent of the choice of $\omega$. 

For any $\omega\in\Sigma(n-1)$ not necessarily generating an extremal ray in $\overline{\on{NE}}(X_\Sigma)$, we can still use the notation above.

\begin{lemma}\label{lc_positive_ray}
    Let $(\cF_W,D)$ be a log canonical toric foliated pair on $X_\Sigma$ with $D\geq 0$, and let $\omega\in\Sigma(n-1)$ be a wall such that $(K_{\mc F_W}+D)\cdot V_\omega<0$. Notation as above. Then there exists $\ell\in J_+$ such that $v_\ell\in W$. 
\end{lemma}

\begin{proof}
    By Proposition~\ref{lc} and the assumption that $D\geq 0$, we have $-(K_{\mc F_W}+D)$ is effective.
    Then $-(K_{\mc F_W}+D)\cdot V_\omega>0$ implies that there is a component $D_\rho$ of $\on{Supp}(-(K_{\mc F_W}+D))\subseteq\on{Supp}(-K_{\mc F_W})$ such that $D_\rho\cdot V_\omega>0$. We must have $\rho=\mb R_{\geq 0}v_\ell$ for some $\ell\in J_+$ by \cite[Lemma~14-1-7]{matsuki2002mmp}, which is still true in this situation. Since $D_\rho$ is non-$\mc F_W$-invariant, we have $v_\ell\in W$.
\end{proof}

\subsection{Divisorial contractions}

In this subsection, we assume that $\alpha=1$. By \cite[Proposition 15.4.5]{CLS11}, this corresponds to $\phi_R\colon X_\Sigma\to X_{\Sigma_0}$ being a divisorial contraction. In this case, the fan $\Sigma$ is the star subdivision of $\Sigma_0$ for $v_1$ (see the proof of \cite[Proposition 15.4.5]{CLS11}), and $\Sigma_0$ is simplicial. 
\begin{proposition}\label{div_contr}
    Let $(\cF_W,D)$ be a log canonical toric foliated pair on $X_\Sigma$ with $D\geq 0$, and let $R\subseteq\overline{\on{NE}}(X_\Sigma)$ be a $(K_{\cF_W}+D)$-negative extremal ray.  
    Assume that $\phi_R\colon X_\Sigma\to X_{\Sigma_0}$ is a divisorial contraction. 
    Then we have the following: 
    \begin{enumerate}
        \item If $\cF_{W}$ on $X_\Sigma$ is non-dicritical, then so is $(\phi_R)_*\cF_{W} = \cF_{W,\,\Sigma_0}$ on $X_{\Sigma_0}$. 
        \item If $(\cF_{W},D)$ is F-dlt, then so is $(\cF_{W,\,\Sigma_0},D_0)$ where $D_0=(\phi_R)_*D$. 
    \end{enumerate}
    
\end{proposition}
\begin{proof}
Notation as in subsection~\ref{subsec_notations}. 
    For (1), it suffices to show that $(\Sigma_0,W)$ satisfies the condition $(\dagger)$, that is, $(\tau,W)$ is non-dicritical for all $\tau\in\Sigma_0$. Note that a full-dimensional cone in $\Sigma_0$ contains $v_1$ if and only if it is of the form $\sigma(\omega)$ for some $\omega\in\Sigma(n-1)$ with $[V_\omega]\in R$. If there is no full-dimensional cone in $\Sigma_0$ containing both $v_1$ and $\tau$, then $\tau\in\Sigma$ and $(\tau,W)$ is non-dicritical by the condition $(\dagger)$ for $(\Sigma,W)$.
    Assume that $\tau\subseteq\sigma(\omega)$ for some $\omega\in\Sigma(n-1)$ with $[V_\omega]\in R$. We can write $\tau=\sigma_J$ for some $J\subseteq [2,n+1]\cap\bN$. 
    If $J_+\nsubseteq J$, then we can choose $j\in J_+\setminus J$ so that $\sigma_J\preceq\sigma^j\in \Sigma$.
    Then $(\sigma_J,W)$ is non-dicritical by the condition $(\dagger)$ for $(\Sigma,W)$. 
        
    Suppose that $J_+\subseteq J$, and that there is an element $v\in\on{Relint}(\sigma_J)\cap W\cap N$.  Write $v=\sum_{j\in J}c_jv_j$ where $c_j\in\bQ\setminus\{0\}$ for each $j$. 
    By Lemma~\ref{lc_positive_ray}, there exists $\ell\in J_+$ such that $v_\ell\in W$.
     Since $|J|\geq|J_+|\geq 2$, $v':=v-c_\ell v_\ell$ is a non-zero element in $\on{Relint}(\sigma_{J\setminus\{\ell\}})\cap W\cap N\otimes\bQ$, and  
    therefore, $\on{Relint}(\sigma_{J\setminus\{\ell\}})\cap W\cap N\neq\emptyset$. 
    By $\sigma_{J\setminus\{\ell\}}\preceq \sigma^\ell\in\Sigma$ and the
    condition $(\dagger)$ for $(\Sigma,W)$, we conclude that $\sigma_{J\setminus\{\ell\}}\subseteq W$.
    Hence $\sigma_J=\sigma_{J\setminus\{\ell\}}+\mb R_{\geq 0}v_\ell\subseteq W$.
       
    For (2), since $X_\Sigma$ is $\mb Q$-factorial, by Proposition~\ref{Fdlt_prop}, Proposition~\ref{Fdlt_ND_prop} and Proposition~\ref{lc}, $(\mc F_W,D)$ is F-dlt if and only if it is non-dicritical and log canonical. The assertion now follows from (1) as MMP preserves log canonical singularities.
\end{proof}

\subsection{Fiber type contractions}

In this subsection, we assume that $\alpha=0$. By \cite[Corollary 14-2-2]{matsuki2002mmp}, this corresponds to $\phi_R\colon X_\Sigma\to X_{\Sigma_0}$ being a fiber type contraction. In this case, $U:=\sigma_{J_+}$ is a vector subspace of $N\otimes\bR$, and $\sigma(\omega)=\sigma_{J_0}\times U$. 
Let $\overline{N}=N/(N\cap U)$ (hence $\overline{N}\otimes\bR=N\otimes\bR/U$).
$\Sigma_0$ is a complete, simplicial fan in $\overline{N}\otimes\bR$, whose collection of full-dimensional cones is $\{\sigma(\omega)\mid \omega\in\Sigma(n-1), [V_\omega]\in R\}$. 
See \cite[Theorem 14-1-9 and its proof, Proposition 14-2-1]{matsuki2002mmp} for details.  

\begin{proposition}\label{MFS}
    Let $(\cF_W,D)$ be a log canonical toric foliated pair on $X_\Sigma$ with $D\geq 0$, and let $R\subseteq\overline{\on{NE}}(X_\Sigma)$ be a $(K_{\cF_W}+D)$-negative extremal ray. 
    Assume that  $\phi_R\colon X_\Sigma\to X_{\Sigma_0}$ is a fiber type contraction, and  
    denote the linear subspace $\sigma_{J_+}\subseteq N\otimes\bR$ by $U$.
    Then we have the following: 
    \begin{enumerate}
        \item If the foliation $\cF_W$ on $X_\Sigma$ is non-dicritical, then $U\subseteq W$ and any fiber of $\phi_R$ that intersects $T_N$ is tangent to $\cF_W$.
        \item Let $\overline{W} = W/(U\otimes\bC)\subseteq \overline{N}\otimes\bC$. Then $\phi_R^{-1}\cF_{\overline{W}}=\mc F_W$ and $\cF_{\overline{W}}$ is non-dicritical on $X_{\Sigma_0}$. 
        \item If $(\cF_W,D)$ is F-dlt, then so is $(\cF_{\overline{W}},\overline{D})$ where $\overline{D}=(\phi_R)_*D$. 
    \end{enumerate}
\end{proposition}
\begin{proof}
    Notation as in subsection~\ref{subsec_notations}. 
    For (1), by Lemma~\ref{lc_positive_ray}, there exists $\ell\in J_+$ such that $v_\ell\in W$.
    The element $\sum_{j\in J_+\setminus\{\ell\}}a_iv_i = -a_{\ell}v_{\ell}$ is in $\on{Relint}(\sigma_{J_+\setminus\{\ell\}})\cap W\cap (N\otimes\bQ)$, and therefore, $\on{Relint}(\sigma_{J_+\setminus\{\ell\}})\cap W\cap N\neq\emptyset$.
    By $\sigma_{J_+\setminus\{\ell\}}\preceq \sigma^\ell\in\Sigma$ and the condition $(\dagger)$ for $(\Sigma,W)$, we have $\sigma_{J_+\setminus\{\ell\}}\subseteq W$. Hence $\sigma_{J_+}=\sigma_{J_+\setminus\{\ell\}}+\mb R_{\geq 0}v_\ell\subseteq W$.
    
    Let $W'=U\otimes_\mb R\mb C$. Then on $T_N$, fibers of $\phi_R$ correspond to leaves of $\mc F_{W'}$ by Proposition~\ref{restriction}. Since $U\subseteq W$, we have $\mc F_{W'}\subseteq \mc F_W$ on $T_N$, implying that any fiber of $\phi_R$ on $T_N$ is tangent to $\mc F_W$.        
    
    For (2), $\phi_R^{-1}\cF_{\overline{W}}=\mc F_W$ is obvious since it is true on $T_N$.
    Let $\tau\in\Sigma_0$. Then there exists $\omega\in\Sigma(n-1)$ with $[V_\omega]\in R$ such that $\tau$ is of the form $\sigma_J+U$ for some $J\subseteq J_0$.
    Suppose that $\on{Relint}(\tau)\cap\overline{W}\cap\overline{N}\neq\emptyset$. Then $(\on{Reint}(\sigma_J)+U)\cap W\cap N\neq\emptyset$, say, $v\in(\on{Reint}(\sigma_J)+U)\cap W\cap N$. Then $v=u_1+u_2=w$ where $u_1\in\on{Reint}(\sigma_J)$, $u_2\in U$, and $w\in W$. Since $\on{Reint}(\sigma_J)$ and $U$ are rational, we actually have $u_1,u_2\in N\otimes\bQ$. Then $u_1=w-u_2\in\on{Reint}(\sigma_J)\cap W\cap (N\otimes\bQ)$ as $u_2\in W$ by (1). By the condition $(\dagger)$ for $(\Sigma,W)$ and the fact that $\sigma_{J}\preceq\sigma^{n+1}\in\Sigma$, we conclude that $\sigma_J\subseteq W$ and hence $\sigma_J+U\subseteq W$ by (1). In other words, $\tau\subseteq\overline{W}$.     
    
    (3) is similar to Proposition~\ref{div_contr}(2). 
\end{proof}

\subsection{Flipping contraction}
In this subsection, we assume that $\alpha\geq 2$, which corresponds to $\phi_R\colon X_\Sigma\to X_{\Sigma_0}$ being a flipping contraction by \cite[Corollary 14-2-2]{matsuki2002mmp}. In this case, $\sigma(\omega)$ is a strictly convex cone which is not simplicial. 
The collection of full-dimensional cones of $\Sigma_0$ is $\{\sigma(\omega)\mid \omega\in\Sigma(n-1),[V_\omega]\in R\}\cup\{\sigma\in\Sigma(n)
\mid \sigma\nsubseteq\sigma(\omega)\ \text{for any}\ \omega\in\Sigma(n-1)\ \text{with}\ [V_\omega]\in R\}$. In particular, $X_{\Sigma_0}$ is not $\mb Q$-factorial. There exists a simplicial refinement $\Sigma^+$ of $\Sigma_0$ with $\Sigma^+(1)=\Sigma(1)=\Sigma_0(1)$ which satisfies the following: a curve $V_{\omega^+}$ on $X_{\Sigma^+}$ is contracted by $\phi_R^+\colon X_{\Sigma^+}\rightarrow X_{\Sigma_0}$ if and only if $[V_{\omega^+}]\in -R$,
where we identify $N^1(X_{\Sigma^+})$ with $N^1(X_{\Sigma})$ by taking the strict transforms
of divisors and hence identify $N_1(X_{\Sigma^+})$ with $N_1(X_{\Sigma})$ as their duals.
We say $\phi_R^+$ is the flip of $\phi_R$. To be more precise, given $\omega\in\Sigma(n-1)$ with $[V_\omega]\in R$, we have
\begin{align*}
\sigma(\omega)&=\bigcup_{j\in J_+}\sigma^j\ \text{in}\ \Sigma,\ \text{and}\\
\sigma(\omega)&=\bigcup_{j\in J_-}\sigma^j\ \text{in}\ \Sigma^+.
\end{align*} 
See \cite[Corollary 14-2-2(iii), Proposition 14-2-4 and its proof]{matsuki2002mmp} for details.

\begin{proposition}\label{flip}
    Let $(\cF_W,D)$ be a log canonical toric foliated pair on $X_\Sigma$ with $D\geq 0$, and let $R\subseteq\overline{\on{NE}}(X_\Sigma)$ be a $(K_{\cF_W}+D)$-negative extremal ray. 
    Assume that  $\phi_R\colon X_\Sigma\to X_{\Sigma_0}$ is a flipping contraction. Let $\phi_R^+\colon X_{\Sigma^+}\to X_{\Sigma_0}$ be the flip of $\phi_R$ as described above and write $\psi\colon X_{\Sigma}\dashrightarrow X_{\Sigma^+}$. 
    Then we have the following: 
    \begin{enumerate}
        \item If the foliation $\cF_W$ on $X_\Sigma$ is non-dicritical, then so are $\cF_{W,\,\Sigma_0}$ and $\cF_{W,\,\Sigma^+}$. 
        \item If $(\cF_{W},D)$ is F-dlt, then so is $(\cF_{W,\,\Sigma^+},D^+)$ where $D^+=\psi_*D$. 
    \end{enumerate}
\end{proposition}
\begin{proof} 
    Notation as in subsection~\ref{subsec_notations}. 
    For (1), it suffices to show that $\mc F_{W,\Sigma_0}$ is non-dicritical, since the condition $(\dagger)$ is preserved under taking refinements. By assumption and the discussion at the beginning of this subsection, we need only to check that $(\tau,W)$ is non-dicritical for any $\tau\in\Sigma_0$ where $\tau\subseteq\sigma(\omega)$ for some $\omega\in\Sigma(n-1)$ with $[V_{\omega}]\in R$. We can write $\tau=\sigma_J$ for some $J\subseteq [1,n+1]\cap\mb N$ as $\Sigma(1)=\Sigma_0(1)$. 
        
    If $J_+\nsubseteq J$, then we can choose $j\in J_+\setminus J$ so that $\sigma_J\preceq\sigma^j\in\Sigma$. Hence $(\sigma_J,W)$ is non-dicritical by the condition $(\dagger)$ for $(\Sigma,W)$.
    Now assume that $J_+\subseteq J$. Suppose we have an element $v\in\on{Relint}(\sigma_J)\cap W\cap N$. By Lemma~\ref{lc_positive_ray}, there exists $\ell\in J_+$ such that $v_\ell\in W$. Then there is a constant $c\in\bQ$ such that $v'=v-cv_{\ell}\in\on{Relint}(\sigma_{J\setminus\{\ell\}})\cap W\cap (N\otimes\bQ)$ as $|J|\geq |J_+|\geq 2$, and therefore,  $\on{Relint}(\sigma_{J\setminus\{\ell\}})\cap W\cap N\neq\emptyset$. Since $\sigma_{J\setminus\{\ell\}}\preceq\sigma^\ell\in\Sigma$, by the condition $(\dagger)$ for $(\Sigma,W)$, we have $\sigma_{J\setminus\{\ell\}}\subseteq W$. Hence $\sigma_J=\sigma_{J\setminus\{\ell\}}+\mb R_{\geq 0}v_\ell\subseteq W$. That is, $(\sigma_J,W)$ is non-dicritical.

    (2) is similar to Proposition~\ref{div_contr}(2).
\end{proof}

\subsection{Cone Theorem}
The goal of this subsection is to prove the cone theorem for log canonical toric foliated pairs (Theorem~\ref{conethm}). 
In Definition~\ref{tangency}, the notion of tangency is discussed when the subvariety is not completely contained in the singular locus.  
The following definition removes this restriction and allows us to talk about tangency for an arbitrary subvariety. For any coherent sheaf $\mc H$ on a normal variety $X$, we write $\mc H(p):=\mc H_p\otimes_{\mc O_{X,p}}\mb C(p)$. 

\begin{definition}\label{defn_tang_general}
    Let $\cF$ be a foliation of any rank on a normal variety $X$. 
    A subvariety $Z\subseteq X$ is \emph{tangent to $\cF$} if there exist a birational morphism $\pi\colon X'\to X$ and a prime divisor $E\subseteq X'$ with $c_X(E)=Z$ which satisfy the following: 
    For any general point $q\in E$, the composition map $\cT_E(q)\cap\cF'(q)\hookrightarrow\cT_E(q) \to \cT_Z(p)$ is surjective where $p=\pi(q)$ and $\cF'=\pi^{-1}\cF$.
\end{definition}

\begin{remark}\label{rmk_tang_general}\hspace{1em}
    \begin{enumerate}
        \item In Definition~\ref{defn_tang_general}, it suffices to find one point  $q\in E\setminus\left(\on{Sing}(X')\cup\on{Sing}(E)\cup\on{Sing}(\cF')\right)$ such that $\pi(q):=p\notin\on{Sing}(Z)$ and that the composition map $\cT_E(q)\cap\cF'(q)\hookrightarrow\cT_E(q) \to \cT_Z(p)$ is surjective.

        \item When the subvariety $Z$ is not contained in $\on{Sing}(X)\cup\on{Sing}(\cF)$, then Definition~\ref{tangency} and Definition~\ref{defn_tang_general} are equivalent. 

        \item 
        Tangency of a subvariety possibly contained in the singular locus has been discussed in the literature under extra assumptions.
        \begin{itemize}
            \item \cite[Definition 3.2]{wang2023toric}: $\dim Z=1$.
            The definition in loc. cit. is slightly different from ours.  
            \item \cite[Subsection 3.4]{ambro2021positivity}: $\mc F$ is algebraically integrable.
            \item \cite[Definition 2.12]{CS}: $\mc F$ is non-dicritical of corank one.
            In Proposition~\ref{tang_CS}, we show that, under the same assumption, this coincides with Definition~\ref{defn_tang_general}. 
        \end{itemize}
    \end{enumerate} 
\end{remark}

The following Lemma is the toroidal analog of \cite[Remark 2.16]{CS}. 
\begin{lemma}\label{lem_tang_exc_inv}
    Let $\cF$ be a non-dicritical toroidal foliation of corank one on a normal variety $X$ with extended complex $(\Delta_X,W)$, and let $Z\subseteq X$ be a subvariety. 
    Assume that there is a prime divisor $E$ over $X$ such that $c_X(E)=Z$ and $E$ is foliation invariant.
    Then, for any prime divisor $G$ over $X$ with $c_X(G)=Z$, we have that $G$ is foliation invariant.  
\end{lemma}

\begin{proof}
    If $Z$ is a divisor, then there is nothing to prove. Hence we assume that $Z$ has codimension at least two. 
    Since $c_X(G)=c_X(E)=Z$, after replacing $Z$ by its Zariski open subset, there is a birational morphism $\pi\colon \widetilde{X}\to X$ such that $\widetilde{X}$ is smooth, $\widetilde{\cF}:=\pi^{-1}\cF$ is toroidal and non-dicritical whose associated extended complex is $(\Delta',W')$, and there is a sequence of smooth exceptional divisors $E=V(\rho_1)$, $\ldots$, $G=V(\rho_k)$ where $\rho_i\in\Delta'(1)$ for $1\leq i\leq k$ such that $V(\rho_j)\cap V(\rho_{j+1})$ is a prime divisor in $V(\rho_j)$ and $V(\rho_{j+1})$ that dominates $Z$ for each $1\leq j\leq k-1$.
    It suffices to prove that $G$ is foliation invariant assuming that $G=D_2$.

    Suppose on the contrary that $G$ is foliation non-invariant. 
    Then $\cG:=\widetilde{\cF}\vert_{G}$ is a foliation of corank one on $G$. 
    Moreover, $\dim Z\geq 1$ as $\cF$ is non-dicritical. 
    By Lemma~\ref{lem:blowup_toroidal_model}, $(\Delta',W')$ satisfies the condition $(\dagger)$ and thus $E\cap G$ is not contained in $\on{Sing}(\widetilde{\cF})$ by Proposition~\ref{singularlocus}. 
    Moreover, $E\cap G$ is invariant under $\cG$ as $E$ is invariant under $\widetilde{\cF}$. 
    Also, since the general fiber of $\psi$ is tangent to $\widetilde{\cG}$, by \cite[Lemma 6.7]{ARAUJO201370}, there is an induced foliation $\cH$ of corank one on $Z$ such that $\psi^{-1}\cH=\cG$. 
    As $E\cap G$ is codimension one in $G$ and is invariant under $\cG$, $E\cap G = \psi^{-1}(S)$ where $S=\psi(E\cap G)$ is invariant under $\cH$. 
    Thus $S$ is codimension one in $Z$, which contradicts to that $E\cap G$ dominates $Z$.   
\end{proof}

\begin{proposition}\label{tang_CS}
    Let $\cF$ be a non-dicritical toroidal foliation of corank one on a normal variety $X$. 
    A subvariety $Z\subseteq X$ is tangent to $\cF$ if and only if for any birational morphism $\pi\colon X'\to X$ and any prime divisor $E$ on $X'$ such that $E$ dominates $Z$, we have $E$ is invariant under the pullback foliation $\cF'=\pi^{-1}\cF$. 
\end{proposition}
\begin{proof}
    We suppose that $Z$ is tangent to $\cF$. 
    Then there exist a birational morphism $\pi\colon X'\to X$ and a prime divisor $E\subseteq X'$ with $c_X(E)=Z$ such that for any general point $p\in Z$, the map $\cT_E(q)\cap\cF'(q)\hookrightarrow\cT_E(q) \to \cT_Z(p)$ is surjective for some general point $q\in(\pi\vert_E)^{-1}(p)$. 
    
    By Lemma~\ref{lem_tang_exc_inv}, it suffices to show the following: 
    \begin{claim}
        $E$ is $\cF'$-invariant. 
    \end{claim}
    \begin{claimproof}
        As $\cF$ is non-dicritical and of corank one, the general fiber $F$ of $\pi\vert_E$ is tangent to $\cF'$. 
        Let $F=(\pi\vert_E)^{-1}(p)$ for some $p\in Z$ and $q\in F$ be a general point. 
        Then we have $\cT_F(q) = \ker(\cT_E(q)\to\cT_Z(p))\subseteq\cF'(q)$ as $F$ is tangent to $\cF'$. 
        Hence $\cT_E(q) = \cT_E(q)\cap\cF'(q) + \cT_F(q)\subseteq\cF'(q)$ and therefore, they are equal since both are $k(p)$-vector spaces of dimension $n-1$.  
        As $p$ and $q$ are general, we have $E$ is $\cF'$-invariant. 
        This completes the proof of the claim. $\quad\blacksquare$
    \end{claimproof}

    On the other hand, suppose for any birational morphism $\pi\colon X'\to X$ and any prime divisor $E$ on $X'$ such that $E$ dominates $Z$, we have $E$ is invariant under the pullback foliation $\cF'$. 
    Then for a general point $p\in Z$ and a general point $q\in (\pi\vert_E)^{-1}(p)$, we have $\cT_E(q) = \cF'(q)$. 
    Hence the map $\cT_E(q)\cap\cF'(q) = \cT_E(q) \to \cT_Z(p)$ is surjective and therefore, $Z$ is tangent to $\cF$. 
\end{proof}

\begin{proposition}\label{property_tangent}
    Let $X$ be a normal variety and $\cF\subseteq\cG$ be two foliations on $X$. 
    We have some properties:
    \begin{enumerate}
        \item If a subvariety $Z\subseteq X$ is tangent to $\cF$, then $Z$ is tangent to $\cG$. 
        \item Let $\pi\colon Y\to X$ be a birational morphism and $Z\subseteq Y$ be a subvariety tangent to $\pi^{-1}\cF$. 
        Then $\pi(Z)$ is tangent to $\cF$. 
    \end{enumerate}
\end{proposition}
\begin{proof}
    For (1), if $Z$ is tangent to $\cF$, then there exist a birational morphism $\pi\colon X'\to X$ and a prime divisor $E\subseteq X'$ with $c_X(E)=Z$ such that for a general point $p\in Z$ and a general point $q\in(\pi\vert_E)^{-1}(p)$, the map $\phi\colon \cT_E(q)\cap\cF'(q)\hookrightarrow\cT_E(q)\to \cT_Z(p)$ is surjective where $\cF'=\pi^{-1}\cF$. 
    Let $\cG'=\pi^{-1}\cG$. 
    Since $\cF\subseteq\cG$, we have $\cF'(q)\subseteq\cG'(q)$ and thus $\psi\colon\cT_E(q)\cap\cG'(q)\hookrightarrow\cT_E(q)\to\cT_Z(p)$ is surjective as $\phi$ factors through $\psi$. 

    For (2), since $Z$ is tangent to $\pi^{-1}\cF$, there exist a birational morphism $\psi\colon Y'\to Y$ and a prime divisor $E\subseteq Y'$ with $c_Y(E)=Z$ such that for a general point $p\in Z$ and a general point $q\in(\pi\vert_E)^{-1}(p)$, the map $\cT_E(q)\cap\widetilde{\cF}(q)\hookrightarrow\cT_E(q)\to\cT_Z(p)$ is surjective where $\widetilde{\cF} = (\pi\circ\psi)^{-1}\cF$. 
    We may assume that $p$ and $q$ are general so that $\pi(p)$ and $q\in(\pi\circ\psi\vert_E)^{-1}(\pi(p))$ are general. 
    Then $\cT_Z(p)\to\cT_{\pi(Z)}(\pi(p))$ is surjective and so is $\cT_E(q)\cap\widetilde{\cF}(q)\hookrightarrow\cT_E(q)\to\cT_Z(p)\to\cT_{\pi(Z)}(\pi(p))$. 
    Therefore, $\pi(Z)$ is tangent to $\cF$. 
\end{proof}

We have the following proposition generalizing \cite[Lemma 3.3]{wang2023toric}:

\begin{proposition}\label{tang_combinatoric}
    Let $\cF_W$ be a toric foliation on a toric variety $X_\Sigma$ of a fan $\Sigma$ in $N\otimes\bR$ where $W\subseteq N\otimes\bC$ is a complex vector subspace. 
    Then for any cone $\tau\in\Sigma$, $V_\tau$ is tangent to $\cF_W$ if and only if $W+\bC\tau=N\otimes\bC$. 
\end{proposition}
\begin{proof}
    Suppose $W+\bC\tau\neq N\otimes\bC$. 
    Then we can choose a complex vector subspace $W'\subseteq N\otimes\bC$ of dimension $n-1$ such that $W\subseteq W'$ and $W'+\bC\tau\neq N\otimes\bC$. 
    Thus, $\bC\tau\subseteq W'$ and hence $(\tau,W')$ is non-dicritical. 
    We pick a primitive element $u\in\on{Relint}(\tau)\cap N$ and let $\Sigma'$ be the star subdivision of $\Sigma$ for the ray $\rho=\bR_{\geq 0}u$. 
    Then on $X_{\Sigma'}$, $D_\rho$ is non-$\cF_{W',\,\Sigma'}$-invariant as $\rho\subseteq\tau\subseteq W'$ and therefore on $X_\Sigma$, $V_\tau$ is not tangent to $\cF_{W'}$ by Proposition~\ref{tang_CS}. 
    Therefore, by Proposition~\ref{property_tangent}(1), $V_\tau$ is not tangent to $\cF_W$ either. 

    On the other hand, if $W+\bC\tau = N\otimes\bC$, then we can choose a complex vector subspace $W''\subseteq W$ such that $W''+\bC\tau=N\otimes\bC$ and $W''\cap \bC\tau=\{0\}$. 
    Taking a toric resolution $X_{\Sigma''}\to X_\Sigma$, $\tau$ is divided into several smooth cones. 
    Let $\tau'$ be one of those cones whose dimension is $\dim\tau$. 
    Then we have $W''+\bC\tau'=N\otimes\bC$, and $W''\cap\bC\tau'=\{0\}$. 
    Since $W''\cap\bC\tau'=\{0\}$, we have that $V_{\tau'}$ is $\cF_{W'',\,\Sigma''}$-invariant and $V_{\tau'}\nsubseteq\on{Sing}(\cF_{W'',\,\Sigma''})$ by Proposition~\ref{singularlocus}. 
    Note that $\dim V_{\tau'}=n-\dim\tau=\dim_\bC W''$. 
    So $V_{\tau'}$ is tangent to $\cF_{W'',\,\Sigma''}$. 
    Therefore, $V_\tau$ is tangent to $\cF_{W''}$ by Proposition~\ref{property_tangent}(2). 
    Hence, by Proposition~\ref{property_tangent}(1), $V_\tau$ is tangent to $\cF_W$. 
\end{proof}

\begin{example}
    Let $N=\bZ e_1\oplus \bZ e_2\oplus \bZ e_3$, $\sigma=\on{Cone}(e_1,e_2,e_3)$, $\tau=\on{Cone}(e_1,e_2)$, and $W=\{(b_1,b_2,b_3)\in\bC^3\mid b_1+\pi b_2=0\}$. 
    Then $U_\sigma\cong\bA^3$ and $\cF_W$ is the toric foliation on $\bA^3$ given by $W$. 
    Note that $V_\tau\subseteq\on{Sing}(\cF_W)$ by Proposition~\ref{singularlocus}. 
    Let $\Sigma'$ be a refinement of $\Sigma$ and $\tau'\in\Sigma'\setminus\Sigma$. 
    Then we have a morphism $\pi\colon X_{\Sigma'}\to X_\Sigma$. 
    Assume that $V_{\tau'}$ dominates $V_\tau$. 
    Then $\tau'$ is a cone contained in $\tau$ of the same dimension, which implies $\mb C\tau=\mb C \tau'$. Since $W\cap\mb C\tau'$ is generated by an irrational ray in $N\otimes\bR$, we have $V_{\tau'}\subseteq\on{Sing}(\cF_{W,\,\Sigma'})$ by Proposition~\ref{singularlocus}. 
\end{example}

\begin{corollary}
    Let $\cF$ be a toroidal foliation on a normal variety $X$ with associated extended complex $(\Delta,W)$. 
    A subvariety $Z$ is tangent to $\cF$ if and only if for any general point $z$ on $Z$, there exist a semi-local model $(U_{\sigma,N},W_N)$ and a cone $\tau_Z\preceq\sigma$ such that $Z$ formally locally corresponds to $V_{\tau_Z}$ and $W_N+\bC\tau_Z = N\otimes\bC$. 
\end{corollary}
\begin{proof}
    (If part) Let $\pi\colon X'\to X$ be the blow-up along $Z$ with exceptional divisor $E$ and $\cF'=\pi^{-1}\cF$. 
    For any general point $z\in Z$, we take the base change to $U_{\sigma,N}$. 
    Then we may assume that $\pi$ is a toric morphism $X'_{\Sigma'}\to U_{\sigma,N}$ between toric varieties where $\Sigma'$ is a fan in $N\otimes\bR$ refining $\Sigma$ with support $|\Sigma'|=\sigma$. 
    Also there is a ray $\rho\in\Sigma'(1)$ so that its primitive element is contained in $\on{Relint}(\tau_Z)$. 
    Note that the complex vector subspaces $(W_N+\bC\rho)/\bC\rho$, $\bC\tau_Z/\bC\rho$, and $N\otimes\bC/\bC\rho$ of $N\otimes\bC/\bC\rho$ gives the foliations $\cF'\vert_E$, the foliation induced by the fibration $\pi\vert_E\colon E\to Z$, and $\cT_E$ on $E$, respectively. 
    As $W_N+\bC\tau_Z = N\otimes\bC$, we have $(W_N+\bC\rho)/\bC\rho + \bC\tau_Z/\bC\rho = N\otimes\bC/\bC\rho$ and hence the map $(W_N+\bC\rho)/\bC\rho \hookrightarrow N\otimes\bC/\bC\rho \to N\otimes\bC/\bC\tau_Z$ is surjective. 
    Therefore, for a general point $q\in E$, we have $\cT_E(q)\cap\cF'(q)\hookrightarrow\cT_E(q)\to \cT_Z(\pi(q))$ is surjective. 

    (Only if part) Now suppose $E$ is a divisor over $X$ with center $Z$ on $X$. 
    If for a general point $z\in Z$, any semi-local model $(U_{\sigma,N},W_N)$ satisfies $W_N+\bC\tau_Z\neq N\otimes\bC$. 
    After base change to $U_{\sigma,N}$, we may assume that $E$ is a divisor over $V_{\tau_Z}$. 
    By Proposition~\ref{tang_combinatoric}, $V_{\tau_Z}$ is not tangent to $\cF_{W_N}$ and thus, for any general point $q\in E$, the map $\cT_E(q)\cap\cF'(q)\hookrightarrow\cT_E(q)\to\cT_Z(\pi(q))$ is not surjective where $\cF'=\pi^{-1}\cF_{W_N}$. 
    By Remark~\ref{rmk_tang_general}(1), $Z$ is not tangent to $\cF$. 
\end{proof}

\begin{theorem}[Cone Theorem]\label{conethm}
    Let $(\cF_W,D)$ be a log canonical toric foliated pair on a complete $\bQ$-factorial toric variety $X_\Sigma$ with $D\geq 0$. 
    Then 
    $\overline{\on{NE}}(X)_{K_{\cF_W}+D<0} = \sum\bR_{\geq 0}[M_i]$ 
    where $M_i$ are torus invariant rational curves tangent to $\cF_W$. 
\end{theorem}
\begin{proof}
    Let $R\subseteq\overline{\on{NE}}(X_\Sigma)$ be a $(K_{\cF_W}+D)$-negative extremal ray, and let $\omega\in\Sigma(n-1)$ be a wall such that $[V_\omega]\in R$. Notation as in subsection~\ref{subsec_notations}.
    Then by Lemma~\ref{lc_positive_ray}, there exists $\ell\in J_+$ such that $v_\ell\in W$.
    Without loss of generality, we may assume that $\ell\neq n+1$. Let $J=\{1,2,\ldots,\widehat{\ell},\ldots,n\}$. 
    By \cite[Proposition 14-1-5(i)]{matsuki2002mmp}, $[V_{\sigma_J}]$ lies in $R$. 
    We have $N\otimes\bC=\mb C v_\ell+\mb C\sigma_J\subseteq W+\mb C \sigma_J\subseteq N\otimes\bC$. Hence  
    $V_{\sigma_{J}}$ is tangent to $\mc F_W$ by Proposition~\ref{tang_combinatoric} and $R=\mb R_{\geq 0}[V_{\sigma_J}]$. 
\end{proof}

\appendix
\section{Appendix}\label{sec:Appendix}

In \cite{wang2023toric}, an alternative definition for non-dicriticality is presented. 
It is not immediately clear whether our definition is equivalent to the one provided in that work. 
However, we can establish the following equivalence for toric foliations on a $\bQ$-factorial toric variety:
\begin{proposition}\label{wang_defn}
    Let $X_\Sigma$ be a $\bQ$-factorial toric variety of a fan $\Sigma$ in $N\otimes\bR$ and $\cF_W$ be a toric foliation on $X_\Sigma$ where $W\subseteq N\otimes\bC$ is a complex vector subspace. 
    Then $(\Sigma,W)$ satisfies the condition $(\dagger)$ if and only if $\cF_W$ is non-dicritical in the sense of \cite[Definition 3.6]{wang2023toric}, which requires all exceptional divisors over the singular locus of the foliation is foliation invariant. 
\end{proposition}
\begin{proof}
    (If part) Suppose $(\Sigma,W)$ does not satisfy the condition $(\dagger)$. 
    Then there is a cone $\tau\in\Sigma$ such that $\on{Relint}(\tau)\cap W\cap N\neq\emptyset$ and $\tau\nsubseteq W$. 
    Let $u\in\on{Relint}(\tau)\cap W\cap N$ and $\Sigma'$ be the star subdivision of $\Sigma$ for the ray $\rho := \bR_{\geq 0}u$. 
    Then we have a birational morphism $\pi\colon X'_{\Sigma'} \to X_\Sigma$ with an exceptional divisor $D_\rho$ whose center on $X_\Sigma$ is $V_\tau$. 
    Since $\rho\subseteq W$, $D_\rho$ is non-$\cF_{W,\,\Sigma'}$-invariant by Corollary~\ref{basics}. 
    We will show that $\cF_W$ is dicritical in the sense of \cite[Definition 3.6]{wang2023toric} by showing $V_\tau\subseteq\on{Sing}(\cF_W)$. 
    
    Let $t=|\{\rho\in\Sigma(1)\mid \rho\subseteq W\cap\tau\}|$ and $\{\rho\in\Sigma(1)\mid \rho\subseteq W\cap\tau\} = \{\rho_1,\ldots,\rho_t\}$. 
    In particular, $\rho_i\in\tau(1)$ for all $1\leq i\leq t$. 
    So $t\leq \dim\tau$. 
    As $\tau\nsubseteq W$, we have $t<\dim\tau$ as $\tau$ is simplicial. 
    Since $u\in\on{Relint}(\tau)$, the vectors $u_i$ together with $u$ are linearly independent over $\bR$ and over $\bC$. 
    Thus, $\dim_\bC(W\cap\bC\tau)\geq t+1$ as $u_i$, $u\in W\cap\bC\tau$. 
    Hence, 
    \[\dim_\bC W + \dim_\bR\tau - \dim_\bC(W+\bC\tau) = \dim_\bC(W\cap\bC\tau) > t\] 
    and therefore, $\dim_\bC W+\dim_\bR\tau - t > \dim_\bC(W+\bC\tau)$. 
    By Proposition~\ref{singularlocus}, we have $V_\tau\subseteq\on{Sing}(\cF_W)$. 

    (Only if part) Now suppose that $E$ is an exceptional divisor whose center on $X_\Sigma$ is $Z$, which is contained in $\on{Sing}(\cF_W)$. 
    Since $\on{Sing}(\cF_W)$ is a torus invariant closed subset of codimension at least $2$, there exists a cone $\tau\in\Sigma(\ell)$ with $\ell\geq 2$ such that $Z\subseteq V_\tau\subseteq\on{Sing}(\cF_W)$. 
    We will show that $E$ is foliation invariant. 

    If $\tau\subseteq W$, then $|\{\rho\in\Sigma(1)\mid\rho\subseteq\tau\cap W\}| = \dim_\bR\tau$ as $\tau$ is simplicial. 
    Thus 
    \[\dim_\bR\tau+\dim_\bC W - |\{\rho\in\Sigma(1)\mid\rho\subseteq\tau\cap W\}| = \dim_\bC W = \dim_\bC(W+\bC\tau)\] 
    and hence, by Proposition~\ref{singularlocus}, $V_\tau\nsubseteq\on{Sing}(\cF_W)$, a contradiction.  
    Therefore, $\tau\nsubseteq W$. 
    
    Let $\Sigma'$ be a smooth fan in $N\otimes\bR$ refining $\Sigma$ and $f\colon X_{\Sigma'}\to X_{\Sigma}$ be the associated toric morphism, which is a toric resolution of $X_\Sigma$. 
    As $\dim\tau\geq 2$, we can consider a finer smooth fan of $\Sigma$ and assume that \[S:=\{\rho\in\Sigma'(1)\setminus\Sigma(1)\mid \rho\cap\on{Relint}(\tau)\neq\emptyset\}\neq\emptyset.\] 
    Then since $c_{X_{\Sigma'}}(E)$ is irreducible, we have $c_{X_{\Sigma'}}(E)\subseteq D_\rho$ for some $\rho\in S$, and by construction $f(D_\rho)=V_\tau\subseteq\on{Sing}(\cF_W)$. 
    Note that $c_{X_{\Sigma'}}(E)\subseteq D_\rho$. 
    By construction $f(D_\rho)=V_\tau\subseteq\on{Sing}(\cF_W)$. 
    
    If $D_\rho$ is not foliation invariant, then $\rho\subseteq W$ by Corollary~\ref{basics}. 
    Since $\rho\cap\on{Relint}(\tau)\neq\emptyset$ and $(\Sigma,W)$ satisfies the condition $(\dagger)$, we have $\tau\subseteq W$, which is impossible. 
    Thus, $D_\rho$ is foliation invariant. 
    Therefore, as $c_{X_{\Sigma'}}(E)\subseteq D_\rho$, $E$ is foliation invariant by Corollary~\ref{cor:contained_in_invariant_divisor}(2) and Zariski's lemma (cf. \cite[Lemma 2.45]{KM98}). 
\end{proof}

\bibliographystyle{alpha}
\bibliography{ToricToroidalFol}

\end{document}